\documentclass[a4paper,oneside,reqno,final]{amsart}
\usepackage[utf8]{inputenc}
\usepackage[a4paper,margin=3cm]{geometry} 
\usepackage{bm}
\usepackage{amsmath}
\usepackage{amsthm}
\usepackage{amscd}
\usepackage{amssymb}
\usepackage{MnSymbol}
\usepackage{latexsym}
\usepackage{eucal}
\usepackage{dsfont}
\usepackage{mathtools}
\usepackage{enumitem}
\usepackage[notref,notcite]{showkeys}
\usepackage{todonotes}

\usepackage{verbatim}
\usepackage{graphicx}
\usepackage{float}
\usepackage{array,booktabs}

\usepackage[colorlinks,pdftex]{hyperref}
\hypersetup{
linkcolor=black,
citecolor=black,
pdftitle={},
pdfauthor={Franziska K\"uhn},
pdfkeywords={},
}

\makeatletter
\renewcommand\section{\@startsection{section}{1}{0mm}{-1.5\baselineskip}{\baselineskip}{\normalsize\bfseries\sffamily}}
\renewcommand\subsection{\@startsection{subsection}{1}{0mm}{-\baselineskip}{\baselineskip}{\normalsize\bfseries\sffamily}}
\makeatother

\makeatletter
\def\@fnsymbol#1{\ensuremath{\ifcase#1\or *\or **\or \dagger\or \ddagger\or
   \mathsection\or \mathparagraph\or \|\or \dagger\dagger
   \or \ddagger\ddagger \else\@ctrerr\fi}}

\newlength{\preskip}
\newlength{\postskip}
\makeatother

\newtheoremstyle{theorem}{\preskip}{\postskip}{\itshape}{}{\bfseries}{}
{.5em}{\textbf{\thmname{#1}\thmnumber{ #2} (\thmnote{ #3})}}
\newtheoremstyle{definition}{\preskip}{\postskip}{\normalfont}{0pt}{\bfseries}{}{.5em}{}
\newtheoremstyle{remark}{\preskip}{\postskip}{\normalfont}{0pt}{\bfseries}{}{.5em}{}

\swapnumbers
\theoremstyle{theorem} \newtheorem{thm}{Theorem}[section]
\theoremstyle{theorem} \newtheorem{lem}[thm]{Lemma}
\theoremstyle{theorem} \newtheorem{prop}[thm]{Proposition}
\theoremstyle{theorem} \newtheorem{kor}[thm]{Corollary}
\theoremstyle{definition} \newtheorem{defn}[thm]{Definition}
\theoremstyle{remark} 
\theoremstyle{remark} 
\theoremstyle{definition} 
\theoremstyle{definition} \newtheorem*{ack}{Acknowledgements}
\theoremstyle{remark} \newtheorem{bem}[thm]{Remark}
\theoremstyle{remark} 
\theoremstyle{definition}  
\theoremstyle{definition}  
\theoremstyle{definition}

\DeclareMathOperator \spt {supp}

\newcommand{\I}{\mathds{1}}
\newcommand\floor[1]{\left\lfloor #1 \right\rfloor}
\newcommand\fa{\qquad \text{for all \ }}
\newcommand{\firstpara}[1]{ \ \textup{#1\ \ }}
\newcommand{\para}[1]{\bigskip\textup{#1\ \ }}

\newcommand\mc[1] {\mathcal{#1}}
\newcommand\mbb[1] {\mathds{#1}}

\newcommand{\eps}{\varepsilon}

\begin{document}

\title[Schauder estimates for equations associated with Feller generators]{Schauder estimates for Poisson equations associated with non-local Feller generators}
\author[F.~K\"{u}hn]{Franziska K\"{u}hn} 
\address[F.~K\"{u}hn]{Institut de Math\'ematiques de Toulouse, Universit\'e Paul Sabatier III Toulouse, 118 Route de Narbonne, 31062 Toulouse, France.  \emph{On leave from:} TU Dresden, Fachrichtung Mathematik, Institut f\"{u}r Mathematische Stochastik, 01062 Dresden, Germany.}
\email{franziska.kuhn@math.univ-toulouse.fr}

\begin{abstract}
	We show how H\"older estimates for Feller semigroups can be used to obtain regularity results for solutions to the Poisson equation $Af=g$ associated with the (extended) infinitesimal generator of a Feller process. The regularity of $f$ is described in terms of H\"older--Zygmund spaces of variable order and, moreover, we establish Schauder estimates.  Since H\"{o}lder estimates for Feller semigroups have been intensively studied in the last years, our results apply to a wide class of Feller processes, e.\,g.\ random time changes of L\'evy processes and solutions to L\'evy-driven stochastic differential equations. Most prominently, we establish Schauder estimates for the Poisson equation associated with the fractional Laplacian of variable order. As a by-product, we obtain new regularity estimates for semigroups associated with stable-like processes.
\end{abstract}
\maketitle

\section{Introduction} \label{intro}

Let $(X_t)_{t \geq 0}$ be an $\mbb{R}^d$-valued Feller process with semigroup $P_t f(x) = \mbb{E}^x f(X_t)$, $x \in \mbb{R}^d$. In this paper, we study the regularity of functions in the abstract H\"older space \begin{equation*}
	F_1 := \left\{f \in \mc{B}_b(\mbb{R}^d); \sup_{t \in (0,1)} \sup_{x \in \mbb{R}^d} \left| \frac{P_t f(x)-f(x)}{t} \right| < \infty \right\},
\end{equation*}
the so-called Favard space of order $1$, cf.\ \cite{butzer,engel}. It is known that for any $f \in F_1$ the limit \begin{equation}
	A_e f(x) := \lim_{t \to 0} \frac{\mbb{E}^x f(X_t)-f(x)}{t} \label{intro-eq2}
\end{equation}
exists up to a set of potential zero, cf.\ \cite{foellmer74}, and this gives rise to the extended infinitesimal generator $A_e$  which maps the Favard space $F_1$ into the space of bounded Borel measurable functions $\mc{B}_b(\mbb{R}^d)$, cf.\ Section~\ref{def} for details. It is immediate from Dynkin's formula that $A_e$ extends the (strong) infinitesimal generator $A$ of $(X_t)_{t \geq 0}$, in particular $F_1$ contains the domain $\mc{D}(A)$ of the infinitesimal generator. We are interested in the following questions: \begin{itemize}
	\item What does the existence of the limit \eqref{intro-eq2} tell us about the regularity of $f \in F_1$? In particular: How smooth are functions in the domain of the infinitesimal generator of $(X_t)_{t \geq 0}$?
	\item If $f \in F_1$ is a solution to the equation $A_e f = g$ and $g$ has a certain regularity, say $g$ is H\"{o}lder continuous of order $\delta \in (0,1)$, then what additional information do we get on the smoothness of $f$?
\end{itemize} 
Our aim is to describe the regularity of $f$ in terms of H\"{o}lder spaces of variable order.   More precisely, we are looking for a mapping $\kappa: \mbb{R}^d \to (0,2)$ such that \begin{equation*}
	f \in F_1 \implies f \in \mc{C}_b^{\kappa(\cdot)}(\mbb{R}^d)
\end{equation*}
where $\mc{C}_b^{\kappa(\cdot)}(\mbb{R}^d)$ denotes the H\"{o}lder--Zygmund space of variable order equipped with the norm \begin{equation*}
	\|f\|_{\mc{C}_b^{\kappa(\cdot)}(\mbb{R}^d)} := \|f\|_{\infty} + \sup_{x \in \mbb{R}^d} \sup_{0<|h| \leq 1} \frac{|f(x+2h)-2f(x+h)+f(x)|}{|h|^{\kappa(x)}},
\end{equation*}
cf.\ Section~\ref{def} for details. If $A_e f=g \in \mc{C}_b^{\delta}(\mbb{R}^d)$ for some $\delta>0$, then it is natural to expect that $f$ ``inherits'' some regularity from $g$, i.\,e.\ \begin{equation*}
	f \in F_1, A_e f = g \in \mc{C}_b^{\delta}(\mbb{R}^d) \implies f \in \mc{C}_b^{\kappa(\cdot)+\varrho}(\mbb{R}^d)
\end{equation*}
for some constant $\varrho=\varrho(\delta)>0$. Moreover, we are interested in establishing Schauder estimates, i.\,e.\ estimates of the form 
\begin{equation}
	\|f\|_{\mc{C}_b^{\kappa(\cdot)}(\mbb{R}^d)} \leq C (\|f\|_{\infty} + \|A_e f\|_{\infty}) \quad \text{and} \quad
\|f\|_{\mc{C}_b^{\kappa(\cdot)+\varrho}(\mbb{R}^d)} \leq C' (\|f\|_{\infty} + \|A_e f\|_{\mc{C}_b^{\delta}(\mbb{R}^d)}). \label{intro-eq6}
\end{equation}
Let us mention that the results, which we present in this paper, do \emph{not} apply to Feller semigroups with a roughening effect (see e.\,g.\ \cite{hairer12} for examples of such semigroups); we study exclusively Feller semigroups with a smoothing effect (see below for details). \par
The toy example, which we have in mind, is the stable-like Feller process $(X_t)_{t \geq 0}$ with infinitesimal generator $A$, \begin{equation}
	Af(x) = c_{d,\alpha(x)} \int_{y \neq 0} \left( f(x+y)-f(x)-y \cdot \nabla f(x) \I_{(0,1)}(|y|)\right) \, \frac{1}{|y|^{d+\alpha(x)}} \, dy, \quad f \in C_c^{\infty}(\mbb{R}^d), \label{intro-eq3}
\end{equation}
which is, rougly speaking, a fractional Laplacian of variable order, i.\,e.\ $A=-(-\Delta)^{\alpha(\bullet)/2}$. Intuitively, $(X_t)_{t \geq 0}$ behaves locally like an isotropic stable L\'evy process but its index of stability depends on the current position of the process. In view of the results in \cite{reg-levy,ihke}, it is an educated guess that any function $f \in \mc{D}(A)$ is ``almost'' locally H\"{o}lder continuous with H\"{o}lder exponent $\alpha(\cdot)$, in the sense that \begin{equation}
	|f(x+2h)-f(x+h)+f(x)| \leq C_{f,\eps} |h|^{\alpha(x)-\eps}, \qquad x,h \in \mbb{R}^d \label{intro-eq4}
\end{equation}
for any small $\eps>0$. We will show that this is indeed true and, moreover, we will establish Schauder estimates for the equation $-(-\Delta)^{\alpha(\bullet)/2}f=g$, cf.\ Theorem~\ref{ex-21} and Corollary~\ref{ex-23}.  \par
Let us comment on related literature. For some particular examples of Feller generators $A$ there are Schauder estimates for solutions to the integro-differential equation $Af=g$ available in the literature; for instance, Bass obtained Schauder estimates for a class of stable-like operators ($\nu(x,dy)=c(x,y) |y|^{-d-\alpha}$ with $c:\mbb{R}^2 \to (0,\infty)$ bounded and $\inf_{x,y} c(x,y)>0$) and Bae \& Kassmann \cite{bae} studied operators with functional order of differentiability ($\nu(x,dy) = c(x,y)/(|y|^d\varphi(y) \,dy)$ for ``nice'' $\varphi$). The recent article \cite{reg-levy} establishes Schauder estimates for a large class of L\'evy generators using gradient estimate for the transition density $p_t$ of the associated L\'evy process. Moreover, we would like to mention the article \cite{ihke} which studies a complementary question -- namely, what are sufficient conditions for the existence of the limit \eqref{intro-eq2} in the space $C_{\infty}(\mbb{R}^d)$ of continuous functions vanishing at infinity -- and which shows that certain H\"older space of variable order are contained in the domain of the (strong) infinitesimal generator. Schauder estimates have interesting applications in the theory of stochastic differential equations (SDES), they can be used to obtain uniqueness results for solutions to SDEs driven by L\'evy processes and to study the convergence of the Euler--Maruyama approximation, see e.g.\ \cite{chen17,euler-maruyama,priola15} and the references therein. \par
This paper consists of two parts. In Section~\ref{main} we show how regularity estimates on Feller semigroups can be used to establish Schauder estimates \eqref{intro-eq6} for functions $f$ in the Favard space of a Feller process $(X_t)_{t \geq 0}$.  Our first result, Proposition~\ref{feller-7}, states that if the semigroup $P_t u(x) := \mbb{E}^x u(X_t)$ satisfies\begin{equation*}
	\|P_t u\|_{\mc{C}_b^{\kappa}(\mbb{R}^d)} \leq c t^{-\beta} \|u\|_{\infty}, \qquad t \in (0,1), \,u \in \mc{B}_b(\mbb{R}^d)
\end{equation*}
for some $\beta \in [0,1)$ and $\kappa>0$, then $F_1 \subseteq \mc{C}_b^{\kappa}(\mbb{R}^d)$ and \begin{equation*}
	\|f\|_{\mc{C}_b^{\kappa}(\mbb{R}^d)} \leq C \left( \|f\|_{\infty} + \|A_e f\|_{\infty} \right) \fa f \in F_1.
\end{equation*}
Proposition~\ref{feller-7} has interesting applications but it does, in general, not give optimal regularity results but rather a worst-case estimate on the regularity of $f \in F_1$; for instance, if $(X_t)_{t \geq 0}$ is an isotropic stable-like process with infinitesimal generator $A = -(-\Delta)^{\alpha(\bullet)/2}$, cf.\ \eqref{intro-eq3}, then an application of Proposition~\ref{feller-7} shows
\begin{equation*}
	|f(x+2h)-2f(x+h)+f(x)| \leq C_{f,\eps} |h|^{\alpha_0-\eps}, \qquad x,h \in \mbb{R}^d, \,f \in \mc{D}(A)
\end{equation*}
where $\alpha_0 := \inf_{x \in \mbb{R}^d} \alpha(x)$, and this is much weaker than the regularity \eqref{intro-eq4} which we would expect. Our main result in Section~\ref{main} is a ``localized'' version of Proposition~\ref{feller-7} which takes into account the local behaviour of the Feller process $(X_t)_{t \geq 0}$ and which allows us to describe the local regularity of a function $f \in F_1$,  cf.\ Theorem~\ref{feller-9} and Corollary~\ref{feller-11}. As an application, we obtain a regularity result for solutions to the Poisson equation $A_e f=g$ with $g \in \mc{C}_b^{\delta}(\mbb{R}^d)$, cf.\ Theorem~\ref{feller-12}. \par 

In the second part of the paper, Section~\ref{ex}, we illustrate the results from Section~\ref{main} with several examples. Applying the results to isotropic-stable like processes, we establish Schauder estimates for the Poisson equation $-(-\Delta)^{\alpha(\bullet)/2}f=g$ associated with the fractional Laplacian of variable order, cf.\ Theorem~\ref{ex-21} and Corollary~\ref{ex-23}. Schauder estimates of this type seem to be a novelty in the literature. As a by-product of the proof, we obtain H\"{o}lder estimates for semigroups of isotropic stable-like processes which are of independent interest, see Section~\ref{iso-reg}. Furthermore, we present Schauder estimates for random time changes of L\'evy processes (Proposition~\ref{ex-9}) and solutions to L\'evy-driven SDEs (Proposition~\ref{ex-5}) and discuss possible extensions. 

\section{Basic definitions and notation} \label{def}

We consider the Euclidean space $\mbb{R}^d$ with the canonical scalar product $x \cdot y := \sum_{j=1}^d x_j y_j$ and the Borel $\sigma$-algebra $\mc{B}(\mbb{R}^d)$ generated by the open balls $B(x,r)$ and closed balls $\overline{B(x,r)}$.  As usual, we set $x \wedge y := \min\{x,y\}$ and $x \vee y := \max\{x,y\}$ for $x,y \in \mbb{R}$. If $f$ is a real-valued function, then $\spt f$ denotes its support, $\nabla f$ the gradient and $\nabla^2 f$ the Hessian of $f$. For two stochastic processes $(X_t)_{t \geq 0}$ and $(Y_t)_{t \geq 0}$ we write $(X_t)_{t \geq 0} \stackrel{d}{=} (Y_t)_{t \geq 0}$ if $(X_t)_{t \geq 0}$ and $(Y_t)_{t \geq 0}$ have the same finite-dimensional distributions.

\emph{Function spaces:} $\mc{B}_b(\mbb{R}^d)$ is the space of bounded Borel measurable functions $f: \mbb{R}^d \to \mbb{R}$. The smooth functions with compact support are denoted by $C_c^{\infty}(\mbb{R}^d)$, and $C_{\infty}(\mbb{R}^d)$ is the space of continuous functions $f: \mbb{R}^d \to \mbb{R}$ vanishing at infinity. Superscripts $k\in\mbb{N}$ are used to denote the order of differentiability, e.\,g.\ $f \in C_{\infty}^k(\mbb{R}^d)$ means that $f$ and its derivatives up to order $k$ are $C_{\infty}(\mbb{R}^d)$-functions. For $U \subseteq \mbb{R}^d$ and $\alpha: U \to [0,\infty)$ bounded we define H\"{o}lder--Zygmund spaces of variable order by \begin{equation*}
	\mc{C}^{\alpha(\cdot)}(U) := \bigg\{f \in C(U); \forall x \in U: \, \, \sup_{\substack{0<|h| \leq 1 \\ x \pm h \in U}} \frac{|\Delta_h^k f(x)|}{|h|^{\alpha(x)}} < \infty \bigg\}
\end{equation*}
and \begin{equation*}
	\mc{C}_b^{\alpha(\cdot)}(U) := \bigg\{f \in C_b(U); \|f\|_{\mc{C}_b^{\alpha(\cdot)}(U)} := \sup_{x \in U} |f(x)| + \sup_{\substack{x \in U, 0<|h| \leq1 \\ \overline{B(x,k|h|)} \subset U}} \frac{|\Delta_h^k f(x)|}{|h|^{\alpha(x)}} < \infty \bigg\}
\end{equation*}
where $k \in \mbb{N}$ is the smallest number which is strictly larger than $\|\alpha\|_{\infty}$ and \begin{equation}
	\Delta_h f(x) := f(x+h)-f(x), \qquad \Delta_h^m f(x) := \Delta_h \Delta_h^{m-1} f(x), \quad m \geq 2, \label{def-eq3}
\end{equation}
are the iterated difference operators. Moreover, we set \begin{equation*}
	\mc{C}_b^{\alpha(\cdot)+}(U) := \bigcup_{\eps>0} \mc{C}_b^{\alpha(\cdot)+\eps}(U) \quad \text{and} \quad \mc{C}_b^{\alpha(\cdot)-}(U) := \bigcap_{\eps>0} \mc{C}_b^{\max\{\alpha(\cdot)-\eps,0\}}(U).
\end{equation*}
Clearly, \begin{equation*}
	\mc{C}_b^{\alpha(\cdot)+}(U) \subseteq \mc{C}_b^{\alpha(\cdot)}(U) \subseteq \mc{C}_b^{\alpha(\cdot)-}(U) \quad \text{and} \quad \mc{C}_b^{\alpha(\cdot)}(U) \subseteq \mc{C}^{\alpha(\cdot)}(U).
\end{equation*}
If $\alpha(x)=\alpha$ is constant, then we write $\mc{C}^{\alpha}(U)$ and $\mc{C}_b^{\alpha}(U)$ for the associated H\"{o}lder--Zygmund spaces. For $U=\mbb{R}^d$ and $\alpha \notin \mbb{N}$ the H\"{o}lder--Zygmund space $\mc{C}_b^{\alpha}(\mbb{R}^d)$ is the ``classical'' H\"{o}lder space $C_b^{\alpha}(\mbb{R}^d)$ equipped with the norm \begin{equation*}
	\|f\|_{C_b^{\alpha}(\mbb{R}^d)} := \|f\|_{\infty} +\sum_{j=0}^{\lfloor \alpha \rfloor}
		\sum_{\substack{\beta \in \mbb{N}_0^d \\ |\beta| = j}} \|\partial^{\beta} f\|_{\infty}
		+   \max_{\substack{\beta \in \mbb{N}_0^d \\ |\beta| = \lfloor \alpha \rfloor}}   \sup_{x \neq y} \frac{|\partial^{\beta} f(x)-\partial^{\beta} f(y)|}{|x-y|^{\alpha-\lfloor \alpha \rfloor}},
\end{equation*}
cf.\ \cite[Section 2.7]{triebel78}. For $\alpha=1$ it is possible to show that $\mc{C}_b^1(\mbb{R}^d)$ is strictly larger than the space of bounded Lipschitz continuous functions, cf.\ \cite[p.~148]{stein}, which is in turn strictly larger than $C_b^1(\mbb{R}^d)$.

\emph{Feller processes:} A Markov process $(X_t)_{t \geq 0}$ is a \emph{Feller process} if the associated transition semigroup $P_t f(x) := \mbb{E}^x f(X_t)$ is a \emph{Feller semigroup}, see e.\,g.\ \cite{ltp,jacob123} for details.  Without loss of generality, we may assume that $(X_t)_{t \geq 0}$ has right-continuous sample paths with finite left-hand limits. Following \cite[II.5.(b)]{engel} we call \begin{equation}
	F_1 := F_1^X := \left\{ f \in \mc{B}_b(\mbb{R}^d); \sup_{t \in (0,1)} \left\| \frac{P_t f-f}{t} \right\|_{\infty}  < \infty \right\} \label{fav}
\end{equation}
the \emph{Favard space of order $1$}. 	The \emph{(strong) infinitesimal generator} $(A,\mc{D}(A))$ is defined by \begin{align*}
	\mc{D}(A) &:= \left\{f \in C_{\infty}(\mbb{R}^d); \exists g \in C_{\infty}(\mbb{R}^d): \, \, \lim_{t \to 0} \left\| \frac{P_tf-f}{t} -g \right\|_{\infty}=0 \right\} \\
	Af &:= \lim_{t \to 0} \frac{P_t f-f}{t}, \quad f \in \mc{D}(A).
\end{align*}
If $\mc{D}(A)$ is rich, in the sense that $C_c^{\infty}(\mbb{R}^d) \subseteq \mc{D}(A)$, then a result by Courr\`ege \& van Waldenfels, see e.\,g.\ \cite[Theorem 2.21]{ltp}, shows that $A|_{C_c^{\infty}(\mbb{R}^d)}$ is a pseudo-differential operator, \begin{equation}
	Af(x) = -q(x,D) f(x) :=  - \int_{\mbb{R}^d} q(x,\xi) e^{ix \cdot \xi} \hat{f}(\xi) \, d\xi, \qquad f \in C_c^{\infty}(\mbb{R}^d), \,x \in \mbb{R}^d \label{pseudo}
\end{equation}
where $\hat{f}(\xi)  := (2\pi)^{-d} \int_{\mbb{R}^d} e^{-ix \cdot \xi} f(x) \, dx$ is the Fourier transform of $f$ and \begin{equation}
	q(x,\xi) = q(x,0) -i b(x) \cdot \xi + \frac{1}{2} \xi \cdot Q(x) \xi + \int_{y \neq 0} \left(1- e^{iy \cdot \xi} +iy \cdot \xi \I_{(0,1)}(|y|) \right) \, \nu(x,dy). \label{symbol}
\end{equation}
is a continuous negative definite \emph{symbol}. If \eqref{pseudo} holds, then we say that $(X_t)_{t \geq 0}$  is a Feller process with symbol $q$. We assume from now on that $q(x,0)=0$. For each fixed $x \in \mbb{R}^d$, $(b(x),Q(x),\nu(x,dy))$ is a L\'evy triplet, i.\,e.\ $b(x) \in \mbb{R}^d$, $Q(x) \in \mbb{R}^{d \times d}$ is symmetric positive semidefinite and $\nu(x,\cdot)$ is a measure on $\mbb{R}^d \backslash \{0\}$ satisfying $\int_{y \neq 0} \min\{1,|y|^2\} \, \nu(x,dy)<\infty$. The symbol $q$ has \emph{bounded coefficients} if \begin{equation*}
	\sup_{x \in \mbb{R}^d} \left(  |b(x)| + |Q(x)| + \int_{y \neq 0} \min\{1,|y|^2\} \, \nu(x,dy) \right)<\infty;
\end{equation*}
by \cite[Lemma 6.2]{schnurr}, $q$ has bounded coefficients if, and only if, $\sup_{x \in \mbb{R}^d} \sup_{|\xi| \leq 1} |q(x,\xi)| < \infty$.  If $(X_t)_{t \geq 0}$ is a Feller process with symbol $q$, then \begin{equation}
	\mbb{P}^x \left( \sup_{s \leq t} |X_s-x| > r \right) \leq ct \sup_{|y-x| \leq r} \sup_{|\xi| \leq r^{-1}} |q(y,\xi)|, \qquad r>0, \, t>0, \, x \in \mbb{R}^d\label{max}
\end{equation}
holds for an absolute constant $c>0$; this maximal inequality goes back to Schilling \cite{rs-growth}, see also \cite[Theorem 5.1]{ltp} or \cite[Lemma 1.29]{matters}. If the symbol $q(\xi)=q(x,\xi)$ of a Feller process $(L_t)_{t \geq 0}$ does not depend on $x \in \mbb{R}^d$, then $(L_t)_{t \geq 0}$ is a \emph{L\'evy process}. By \cite[Theorem 2.6]{ltp} this is equivalent to saying that $(L_t)_{t \geq 0}$ has stationary and independent increments. Later on, we will use that any Feller process $(X_t)_{t \geq 0}$ with infinitesimal generator $(A,\mc{D}(A))$ solves the $(A,\mc{D}(A))$-martingale problem, i.\,e.\ \begin{equation*}
	M_t := f(X_t)-f(X_0)-\int_0^t Af(X_s) \, ds
\end{equation*}
is a $\mbb{P}^x$-martingale for any $x \in \mbb{R}^d$ and $f \in \mc{D}(A)$. Our standard reference for Feller processes are the monographs \cite{ltp,jacob123}, and for further information on martingale problems we refer the reader to \cite{ethier,hoh}.  \par
In the remaining part of this section we define the extended infinitesimal generator  and state some results which we will need later on. Following \cite{meyer76} we define the \emph{extended (infinitesimal) generator} $A_e$ in terms of the $\lambda$-potential operator $R_{\lambda}$, that is, \ $f \in \mc{D}(A_e)$ and $g=A_e f$ if, and only if,  \begin{enumerate}
		\item\label{gen-1-i} $f \in \mc{B}_b(\mbb{R}^d)$ and $g$ is a measurable function such that $\|R_{\lambda}(|g|)\|_{\infty}< \infty$ for some (all) $\lambda>0$,
		\item\label{gen-1-iii} $f = R_{\lambda}(\lambda f-g)$ for all $\lambda>0$.
	\end{enumerate}
The mapping $g=A_e f$ is defined up to a set of potential zero, i.e.\ up to a set $B \in \mc{B}(\mbb{R}^d)$ which satisfies $\mbb{E}^x \int_{(0,\infty)} \I_B(X_t) \, dt = 0$ for all $x \in \mbb{R}^d$. We will often choose a representative with a certain property; for instance, if we write ``$A_e f$ is continuous'', this means that there exists a continuous function $g$ such that \eqref{gen-1-i},\eqref{gen-1-iii} hold. In abuse of notation we set  \begin{equation*}
	\|A_e f\|_{\infty} := \inf\{c>0; |A_e f| \leq c \, \, \text{up to a set of potential zero}\}.
\end{equation*}
Clearly, the extended infinitesimal generator $(A_e,\mc{D}(A_e))$ extends the (strong) infinitesimal generator $(A,\mc{D}(A))$.  The following result is essentially due to Airault \& F\"{o}llmer \cite{foellmer74} and shows the connection to the Favard space of order $1$, cf.\ \eqref{fav}.

\begin{thm} \label{gen-3}
	Let $(X_t)_{t \geq 0}$ be a Feller process with semigroup $(P_t)_{t \geq 0}$ and extended generator $(A_e,\mc{D}(A_e))$. The associated Favard space $F_1$ of order $1$ satisfies \begin{equation*}
		F_1 = \{f \in \mc{D}(A_e); \|A_e f\|_{\infty}< \infty\}.
	\end{equation*}
	If $f \in F_1$ then \begin{equation}
		\sup_{t \in (0,1)} \frac{1}{t} \|P_t f-f\|_{\infty} = \|A_e f\|_{\infty} \label{gen-eq5}
	\end{equation}
	and, moreover, Dynkin's formula \begin{equation}
		\mbb{E}^x f(X_{\tau})-f(x) = \mbb{E}^x \left( \int_0^{\tau} A_e f(X_s) \, ds \right) \label{gen-eq6}
	\end{equation}
	holds for any $x \in \mbb{R}^d$ and any stopping time $\tau$ such that $\mbb{E}^x \tau<\infty$.
 \end{thm}

The next corollary shows how the Favard space can be defined in terms of the stopped process $X_{t \wedge \tau_r^x}$. It plays an important role in our proofs since we will frequently use stopping techniques. 

\begin{kor} \label{gen-5}
	Let $(X_t)_{t \geq 0}$ be a Feller process with semigroup $(P_t)_{t \geq 0}$,  extended generator $(A_e,\mc{D}(A_e))$ and symbol $q$. Denote by \begin{equation*}
	\tau_r^x := \inf\{t>0; |X_t-x|>r\}
\end{equation*}
the exit time of $(X_t)_{t \geq 0}$ from the closed ball $\overline{B(x,r)}$. If $q$ has bounded coefficients, then the following statements are equivalent for any $f \in \mc{B}_b(\mbb{R}^d)$. 
\begin{enumerate}
	\item\label{gen-5-i} $f \in F_1$, i.\,e.\ $f \in \mc{D}(A_e)$ and $\sup_{t \in (0,1)} t^{-1} \|P_t f-f\|_{\infty} = \|A_e f\|_{\infty}< \infty$,
	\item\label{gen-5-iii} There exists $r>0$ such that \begin{equation*}
		K_r(f) := \sup_{t \in (0,1)} \frac{1}{t} \sup_{x \in \mbb{R}^d} |\mbb{E}^x f(X_{t \wedge \tau_r^x})-f(x)|<\infty. \end{equation*}
\end{enumerate}
If one (hence both) of the conditions is satisfied, then \begin{equation}
	A_e f(x) = \lim_{t \to 0} \frac{\mbb{E}^x f(X_{t \wedge \tau_r^x})-f(x)}{t},
	\label{gen-eq7}
\end{equation}
up to a set of potential zero, for any $r>0$. In particular, $\|A_e f\|_{\infty} \leq K_r(f)$ for $r>0$.
\end{kor}

For the proof of Theorem~\ref{gen-3} and Corollary~\ref{gen-5} and some further remarks we refer to the appendix.

\section{Main results} \label{main}

Let $(X_t)_{t \geq 0}$ be a Feller process with semigroup $(P_t)_{t \geq 0}$. Throughout this section, \begin{equation*}
	F_1^X := F_1 := \left\{f \in \mc{B}_b(\mbb{R}^d); \sup_{t \in (0,1)} \left\| \frac{P_tf-f}{t} \right\|_{\infty} < \infty \right\}
\end{equation*}
is the Favard space of order $1$ associated with $(X_t)_{t \geq 0}$. By Theorem~\ref{gen-3}, we have \begin{equation*}
	F_1 = \{f \in \mc{D}(A_e); \|A_e f\|_{\infty}< \infty\}
\end{equation*}
where $A_e$ denotes the extended infinitesimal generator. The results which we present in this section will be proved in Section~\ref{proofs}. \par \medskip 
Our first result, Proposition~\ref{feller-7}, shows how regularity estimates for the semigroup $(P_t)_{t \geq 0}$ can be used to obtain Schauder estimates of the form \begin{equation*}
\|f\|_{\mc{C}_b^{\kappa}(\mbb{R}^d)} \leq C (\|f\|_{\infty}+\|A_e f\|_{\infty}), \qquad f \in F_1.
\end{equation*}

\begin{prop} \label{feller-7}
	Let $(X_t)_{t \geq 0}$ be a Feller process with semigroup $(P_t)_{t \geq 0}$, extended generator $(A_e,\mc{D}(A_e))$ and Favard space $F_1$. If there exist constants $M>0$, $T>0$, $\kappa \geq 0$ and $\beta \in (0,1)$ such that 
	\begin{equation}
		\|P_t u\|_{\mc{C}_b^{\kappa}(\mbb{R}^d)} \leq M  t^{-\beta} \|u\|_{\infty} \label{feller-eq15}
	\end{equation}
	for all $u \in \mc{B}_b(\mbb{R}^d)$ and $t \in (0,T]$, then  \begin{equation*}
		F_1\subseteq \mc{C}_b^{\kappa}(\mbb{R}^d) 
	\end{equation*}
	and \begin{equation*}
	\|f\|_{\mc{C}^{\kappa}_b(\mbb{R}^d)} \leq C (\|f\|_{\infty} + \|A_ef\|_{\infty}), \qquad f \in F_1,
	\end{equation*}
	for some constant $C=C(T,M,\kappa,\beta)$.
\end{prop}

Since the domain $\mc{D}(A)$ of the (strong) infinitesimal generator of $(X_t)_{t \geq 0}$ is contained in $F_1$, Proposition~\ref{feller-7} gives, in particular, $\mc{D}(A) \subseteq \mc{C}_b^{\kappa}(\mbb{R}^d)$.  \par 

Proposition~\ref{feller-7} is a useful tool but it does, in general, not give optimal regularity results. Since Feller processes are inhomogeneous in space, the regularity of $f \in F_1$ will, in general, depend on the space variable $x$, e.\,g.\ \begin{equation}
	|\Delta_h^2 f(x)| = |f(x+2h)-2f(x+h)+f(x)| \leq C |h|^{\kappa(x)}, \qquad |h| \leq 1, \label{feller-st1} 
\end{equation}
and therefore it is much more natural to use H\"older--Zygmund spaces of variable order to describe the regularity; this is also indicated by the results obtained in \cite{ihke}. \par
Our second result, Theorem~\ref{feller-9}, shows how H\"older estimates for Feller semigroups can be used to establish local H\"older estimates \eqref{feller-st1}. Before stating the result, let us explain the idea. Let $(X_t)_{t \geq 0}$ be a Feller process with symbol $q$ and Favard space $F_1^X$, and fix $x \in \mbb{R}^d$.  Let $(Y_t)_{t \geq 0}$ be another Feller process which has the same behaviour as $(X_t)_{t \geq 0}$ in a neighbourhood of $x$, in the sense that its symbol $p$ satisfies \begin{equation}
p(z,\xi) = q(z,\xi), \qquad  z \in B(x,\delta), \,\xi \in \mbb{R}^d \label{feller-eq16}
\end{equation}
for some $\delta>0$. The aim is to choose $(Y_t)_{t \geq 0}$ in such a way that its semigroup $(T_t)_{t \geq 0}$ satisfies a ``good'' regularity estimate \begin{equation*}
\|T_t u\|_{\mc{C}_b^{\kappa}(\mbb{R}^d)} \leq M t^{-\beta} \|u\|_{\infty}, \qquad u \in \mc{B}_b(\mbb{R}^d);
\end{equation*}
here ``good'' means that $\kappa$ is large. Because of \eqref{feller-eq16} it is intuitively clear that \begin{equation}
	|\mbb{E}^z f(X_t)- f(z)| \approx |\mbb{E}^z f(Y_t)-f(z)| \quad \text{for $z$ close to $x$ and ``small'' $t$}. \label{feller-eq17}
\end{equation}
If $\chi$ is a truncation function such that $\I_{B(x,\eps)} \leq \chi \leq  \I_{B(x,2\eps)}$ for small $\eps>0$, then it is, because of \eqref{feller-eq17}, natural to expect that for any $f \in F_1^X$ the truncated mapping $g := f \cdot \chi$ is in the Favard space $F_1^Y$ associated with $(Y_t)_{t \geq 0}$, i.\,e.\ \begin{equation*}
	\sup_{t \in (0,1)} \sup_{z \in \mbb{R}^d} t^{-1}|\mbb{E}^z(f \cdot \chi)(Y_t)-(f \cdot \chi)(z)| < \infty.
\end{equation*}
Since, by Proposition~\ref{feller-7}, $g \in F_1^Y \subseteq \mc{C}_b^{\kappa}(\mbb{R}^d)$, and $g = f$ in a neighbourhood of $x$, this entails that $f(\cdot)$ is $\kappa$-H\"older continuous in a neighbourhood of $x$.  Since $\kappa=\kappa(x)$ depends on the point $x \in \mbb{R}^d$, which we fixed at the beginning, this localizing procedure allows us to obtain local H\"older estimates \eqref{feller-st1} for $f$.

\begin{thm} \label{feller-9} 
	Let $(X_t)_{t \geq 0}$ be a Feller process with extended generator $(A_e,\mc{D}(A_e))$ and Favard space $F_1^X$ such that \begin{equation*}
	A_e f(z) = -q(z,D) f(z), \qquad f \in C_c^{\infty}(\mbb{R}^d), \,z \in \mbb{R}^d,
	\end{equation*}
	for a continuous negative definite symbol $q$, cf.\ \eqref{pseudo}. Let $x \in \mbb{R}^d$ and $\delta \in (0,1)$ be such that there exists a Feller process $(Y_t^{(x)})_{t \geq 0}$ with the following properties: 
	\begin{enumerate}[label*=\upshape (C\arabic*),ref=\upshape C\arabic*] 
		\item\label{C1} The infinitesimal generator $(L^{(x)},\mc{D}(L^{(x)}))$ of $(Y_t^{(x)})_{t \geq 0}$ equals when restricted to $C_c^{\infty}(\mbb{R}^d)$ a pseudo-differential operator with negative definite symbol $p^{(x)}$, \begin{equation*}
		p^{(x)}(z,\xi) = -i b^{(x)}(z) \cdot \xi + \int_{y \neq 0} \left(1-e^{iy \cdot \xi}+i y \cdot \xi \I_{(0,1)}(|y|) \right) \, \nu^{(x)}(z,dy), \quad z,\xi \in \mbb{R}^d;
		\end{equation*}
		$p^{(x)}$ has bounded coefficients and \begin{equation}
		p^{(x)}(z,\xi) = q(z,\xi) \fa \xi \in \mbb{R}^d, \,|z-x| \leq 4\delta. \label{feller-eq21}
		\end{equation}
		\item\label{C2} The $(L^{(x)},C_c^{\infty}(\mbb{R}^d))$-martingale problem is well-posed.
		\item\label{C4} There exist constants $M(x)>0$,  $\kappa(x)  \in [0,2]$ and $\beta(x) \in (0,1)$ such that the semigroup $(T_t^{(x)})_{t \geq 0}$ associated with $(Y_t^{(x)})_{t \geq 0}$ satisfies \begin{equation*}
			\|T_t^{(x)} u\|_{\mc{C}_b^{\kappa(x)}(\mbb{R}^d)} \leq M(x)  t^{-\beta(x)} \|u\|_{\infty}
		\end{equation*}
		for all $u \in \mc{B}_b(\mbb{R}^d)$, $t \in (0,1)$.
	\end{enumerate}
	If $f \in F_1^X$  and $\varrho(x) \in [0,1]$ are such that \begin{equation} \label{feller-eq22}
	\|f\|_{\mc{C}_b^{\varrho(x)}(\overline{B(x,4\delta)})} < \infty \quad \text{and} \quad \sup_{|z-x| \leq 4 \delta} \int_{|y| \leq 1} |y|^{1+\varrho(x)} \, \nu^{(x)}(z,dy)<\infty,
	\end{equation}
	then \begin{equation}
	|\Delta_h^2 f(x)| \leq C |h|^{\kappa(x)} \left(\|f\|_{\infty}+\|A_ef\|_{\infty}+\|f\|_{\mc{C}_b^{\varrho(x)}(\overline{B(x,4\delta)})}\right) \label{feller-eq23}
	\end{equation}
	for all $|h| \leq \delta/2$. The finite constant $C>0$ depends continuously on $M(x) \in [0,\infty)$, $\beta(x) \in [0,1)$ and $K(x) \in [0,\infty)$, \begin{equation*}
	K(x) :=  \sup_{z \in \mbb{R}^d} \left( |b^{(x)}(z)| + \int_{y \neq 0} \min\{1,|y|^2\} \, \nu^{(x)}(z,dy) \right)+ \sup_{|z-x| \leq 4 \delta}  \int_{y \neq 0} \min\{|y|^{\varrho(x)+1},1\} \, \nu^{(x)}(z,dy).
	\end{equation*}
\end{thm}

\begin{bem} \label{feller-10} \begin{enumerate}[wide, labelwidth=!, labelindent=0pt]
		\item\label{feller-10-i} The assumption $f \in C_b^{\varrho(x)}(\overline{B(x,4\delta)})$ is an a-priori estimate on the regularity of $f$. If the semigroup $(P_t)_{t \geq 0}$ of $(X_t)_{t \geq 0}$ satisfies a regularity estimate of the form \eqref{feller-eq15}, then such an a-priori estimate can be obtained from Proposition~\ref{feller-7}. Note that, by \eqref{feller-eq22}, there is a trade-off between the required a-priori regularity of $f$ and the roughness of the measures $\nu^{(x)}(z,dy)$, $z \in \overline{B(x,4\delta)}$. If the measures $\nu^{(x)}(z,dy)$ only have a weak singularity at $y=0$, in the sense that \begin{equation*}
			\sup_{|z-x| \leq 4\delta} \int_{|y| \leq 1} |y| \, \nu^{(x)}(z,dy)<\infty,
		\end{equation*}
		then we can choose $\varrho(x)=0$, i.\,e.\  it suffices that $f$ is continuous. In contrast, if (at least) one of the measures has a strong singularity at $y=0$, then we need a higher regularity of $f$ (in a neighbourhood of $x$).
		\item\label{feller-10-iii} It is not very restrictive to assume that $(Y_t^{(x)})_{t \geq 0}$ has bounded coefficients since $(Y_t^{(x)})_{t \geq 0}$ is only supposed to mimic the behaviour of $(X_t)_{t \geq 0}$ in a neighbourhood of $x$, cf.\ \eqref{feller-eq21}. We are, essentially, free to choose the behaviour of the process far away from $x$.  In dimension $d=1$ it is, for instance, a natural idea is to consider \begin{equation*}
		p^{(x)}(z,\xi) := \begin{cases} q(x-4\delta,\xi), & z \leq x-4 \delta, \\ q(z,\xi), & |z-x| < 4 \delta; \\ q(x+4\delta,\xi), & z \geq x+4\delta \end{cases} 
		\end{equation*}
		note that $p^{(x)}$ has bounded coefficients even if $q$ has unbounded coefficients.
		\item\label{feller-10-iv} Condition \eqref{C2} is automatically satisfied if $C_c^{\infty}(\mbb{R}^d)$ is a core for the infinitesimal generator of $(Y_t^{(x)})_{t \geq 0}$, see e.\,g.\ \cite[Proposition 3.9.3]{kol} or \cite[Theorem 1.38]{matters}.
		\item\label{feller-10-ii} It is possible to extend Theorem~\ref{feller-9} to Feller processes with a non-vanishing diffusion part.  The idea of the proof is similar but we need to impose stronger assumptions on the regularity on $f$, e.\,g.\ that $f|_{B(x,4\delta)}$ is differentiable.
\end{enumerate} \end{bem}

As a direct consequence of Theorem~\ref{feller-9} we obtain the following corollary.

\begin{kor} \label{feller-11}
	Let $(X_t)_{t \geq 0}$ be a Feller process with extended generator $(A_e,\mc{D}(A_e))$ and symbol $q$. If there exist $U \subseteq \mbb{R}^d$ open, $\delta>0$ and  $\varrho: U \to [0,1]$ such that for any $x \in U$ the assumptions of Theorem~\ref{feller-9} hold, then the Favard space of order $1$ satisfies \begin{equation*}
	   \mc{C}^{\varrho(\cdot)}(U) \cap F_1 \subseteq \mc{C}^{\kappa(\cdot)}(U).
	\end{equation*}
	If additionally \begin{equation}
	\sup_{x \in U} (M(x)+K(x)) <\infty \quad \text{and} \quad \sup_{x \in U} \beta(x)<1, \label{main-eq9}
	\end{equation}
	then $\mc{C}_b^{\varrho(\cdot)}(U) \cap F_1 \subseteq \mc{C}_b^{\kappa(\cdot)}(U)$ and there exists a constant $C>0$ such that \begin{equation}
	\|f\|_{\mc{C}_b^{\kappa(\cdot)}(U)} \leq C \left(\|f\|_{\infty}+\|A_e f\|_{\infty}+\|f\|_{\mc{C}^{\varrho(\cdot)}_b(U)} \right) \fa f \in \mc{C}_b^{\varrho(\cdot)}(U) \cap F_1; \label{main-eq11}
	\end{equation}
	in particular, the the infinitesimal generator $(A,\mc{D}(A))$ satisfies $\mc{C}_b^{\varrho(\cdot)}(U) \cap \mc{D}(A) \subseteq \mc{C}_b^{\kappa(\cdot)}(U)$ and \eqref{main-eq11} holds for any $f \in \mc{C}_b^{\varrho(\cdot)}(U) \cap \mc{D}(A)$.
\end{kor}

In many examples, see e.\,g.\ Section~\ref{ex}, it is possible to choose the mapping $\varrho$ in such a way that $F_1 \subseteq \mc{C}_b^{\varrho(\cdot)}(U)$; in this case, Corollary~\ref{feller-11} shows that $F_1 \subseteq \mc{C}^{\kappa(\cdot)}(U)$ (resp.\ $F_1 \subseteq \mc{C}_b^{\kappa(\cdot)}(U)$) and the Schauder estimate \eqref{main-eq11} holds for any function $f \in F_1$. In our applications we will even have $\|f\|_{\mc{C}_b^{\varrho(\cdot)}(U)} \leq c(\|f\|_{\infty}+\|A_e f\|_{\infty})$ and therefore \eqref{main-eq11} becomes \begin{equation*}
	\|f\|_{\mc{C}_b^{\kappa(\cdot)}(U)} \leq C' \left(\|f\|_{\infty}+\|A_e f\|_{\infty} \right) \fa f \in F_1.
\end{equation*}
In Section~\ref{ex} we will apply Corollary~\ref{feller-11} to isotropic stable-like processes, i.\,e.\ Feller processes with symbol of the form $q(x,\xi) = |\xi|^{\alpha(x)}$. The study of the domain $\mc{D}(A)$ of the infinitesimal generator $A$ is particularly interesting since $A$ is an operator of variable order. We will show that any function $f \in \mc{D}(A)$ satisfies the H\"older estimate of variable order \begin{equation*}
	|\Delta_h^2 f(x)| \leq C_{\eps} |h|^{\alpha(x)-\eps} (\|f\|_{\infty}+\|Af\|_{\infty}), \qquad |h| \leq 1, \,x \in \mbb{R}^d,
\end{equation*}
for $\eps>0$, cf.\ Theorem~\ref{ex-21} for the precise statement. \par \medskip

Our final result in this section is concerned with Schauder estimates for solutions to the equation $A_e f = g$ for H\"{o}lder continuous mappings $g$. To establish such Schauder estimates we need additional assumptions on the regularity of the symbol and improved regularity estimates for the semigroup of the ``localizing'' Feller process $(Y_t^{(x)})_{t \geq 0}$ in Theorem~\ref{feller-9}.

\begin{thm} \label{feller-12}
Let $(X_t)_{t \geq 0}$ be a Feller process with extended generator $(A_e,\mc{D}(A_e))$ and Favard space $F_1^X$ such that \begin{equation*}
	A_e f(z) = -q(z,D) f(z), \qquad f \in C_c^{\infty}(\mbb{R}^d), \,z \in \mbb{R}^d,
\end{equation*}
for a continuous negative definite symbol $q$. Assume that there exists $\delta \in (0,1)$ such that for any $x \in \mbb{R}^d$ there exists a Feller process $(Y_t^{(x)})_{t \geq 0}$ with symbol \begin{equation}
	p^{(x)}(z,\xi) = -ib^{(x)}(z) \cdot \xi + \int_{y \neq 0} \left(1-e^{iy \cdot \xi} + iy \cdot \xi \I_{(0,1)}(|y|) \right) \, \nu^{(x)}(z,dy), \qquad z, \xi \in \mbb{R}^d, \label{main-eq15}
\end{equation}
satisfying \eqref{C1}-\eqref{C4} in Theorem~\ref{feller-9}. Assume additionally that the following conditions hold for absolute constants $C_1,C_2>0$. 
\begin{enumerate}[label*=\upshape (S\arabic*),ref=\upshape S\arabic*]
	\item\label{S1} For any $x,z \in \mbb{R}^d$ there exists $\alpha^{(x)}(z) \in (0,2)$ such that \begin{equation*}
		\nu^{(x)}(z,dy) \leq C_1 |y|^{-d-\alpha^{(x)}(z)} \, dy \quad \text{on $B(0,1)$}
	\end{equation*}
	and $0<\inf_{x,z \in \mbb{R}^d} \alpha^{(x)}(z) \leq \sup_{x,z \in\mbb{R}^d} \alpha^{(x)}(z) < 2$. 
	\item\label{S2} There exists $\theta \in (0,1]$ such that \begin{equation}
		|b^{(x)}(z)-b^{(x)}(z+h)| \leq C_2 |h|^{\theta}, \qquad x,z,h \in \mbb{R}^d, \label{main-eq16}
	\end{equation}
	and the following statement holds true for any $r \in (0,1)$ and $x,z \in \mbb{R}^d$: If $u: \mbb{R}^d \to \mbb{R}$ is a measurable mapping such that \begin{equation*}
		|u(y)| \leq c_u \min\{|y|^{\alpha^{(x)}(z)+r},1\}, \qquad y \in \mbb{R}^d,
	\end{equation*}
	for some $c_u>0$, then there exist $C_{3,r}>0$ and $H_r>0$ (not depending on $u$, $x,z$) such that \begin{equation}
		\left| \int u(y) \, \nu^{(x)}(z,dy) - \int u (y) \, \nu^{(x)}(z+h,dy) \right|
		\leq C_{3,r} c_u |h|^{\theta} \fa |h| \leq H_r. \label{main-eq17}
	\end{equation}
	\item\label{S3} There exists $\Lambda>0$ such that the semigroup $(T_t^{(x)})_{t \geq 0}$ of the Feller process $(Y_t^{(x)})_{t \geq 0}$ satisfies \begin{equation}
		\|T_t^{(x)} u\|_{\mc{C}_b^{\lambda+\kappa(x)}(\mbb{R}^d)} \leq M(x) t^{-\beta(x)} \|u\|_{\mc{C}_b^{\lambda}(\mbb{R}^d)}, \quad u \in \mc{C}_b^{\lambda}(\mbb{R}^d), \,t \in (0,1), \label{main-eq19}
	\end{equation}
	for any $x \in \mbb{R}^d$ and $\lambda \in [0,\Lambda]$; here $M(x)$, $\kappa(x)$ and $\beta(x)$ denote the constants from \eqref{C4}.
	\item\label{S5} The mapping $\kappa:\mbb{R}^d \to (0,\infty)$ is uniformly continuous and bounded away from zero, i.\,e.\ $\kappa_0 := \inf_{x \in \mbb{R}^d} \kappa(x)>0$.
	\item\label{S4} It holds that \begin{equation*}
		 \sup_{x \in \mbb{R}^d} M(x)< \infty \qquad \sup_{x \in \mbb{R}^d} \beta(x)<1 \qquad
		\sup_{x,z \in \mbb{R}^d} \left( |b^{(x)}(z)| + \int_{|y| \geq 1} \,\nu^{(x)}(z,dy) \right)< \infty.
	\end{equation*}
\end{enumerate}
Let $\varrho: \mbb{R}^d \to [0,2]$ be a uniformly continuous function satisfying \begin{equation}
	\sigma := \inf_{x \in \mbb{R}^d} \inf_{|z-x| \leq 4 \delta} \left(1+\varrho(x)-\alpha^{(x)}(z) \right) >0. \label{main-eq21}
\end{equation}
If $f \in F_1^X$ is such that $f \in \mc{C}_b^{\varrho(\cdot)}(\mbb{R}^d)$ and \begin{equation*}
	A_e f = g \in \mc{C}_b^{\lambda}(\mbb{R}^d)
\end{equation*}
for some $\lambda \in [0,\Lambda]$, then $f \in \mc{C}_b^{(\kappa(\cdot)+\min\{\theta,\lambda,\sigma\})-}(\mbb{R}^d)$, i.\,e.\ \begin{equation}
	f \in \bigcap_{\eps \in (0,\kappa_0)} \mc{C}_b^{\kappa(\cdot)+\min\{\theta,\lambda,\sigma\}-\eps}(\mbb{R}^d). \label{main-eq25}
\end{equation}
Moreover,  the Schauder estimate \begin{equation}
	\|f\|_{\mc{C}_b^{\kappa(\cdot)+\min\{\theta,\lambda,\sigma\}-\eps}(\mbb{R}^d)} \leq C_{\eps} \left(\|A_ef\|_{\mc{C}_b^{\lambda}(\mbb{R}^d)} + \|f\|_{\mc{C}_b^{\varrho(\cdot)}(\mbb{R}^d)}\right) \label{main-eq27}
\end{equation}
holds for any $\eps \in (0,\kappa_0)$ and some finite constant $C_{\eps}$ which does not depend on $f$, $g$.
\end{thm}

\begin{bem} \label{feller-125} \begin{enumerate}[wide, labelwidth=!, labelindent=0pt] 
	\item\label{feller-125-i} In our examples in Section~\ref{ex} we will be able to choose $\varrho$ in such a way that $\alpha^{(x)}(z)-\varrho(z)$ is arbitrarily small for $x \in \mbb{R}^d$ and $z \in \overline{B(x,4\delta)}$, and therefore the constant $\sigma$ in \eqref{main-eq21} will be close to $1$. Noting that $\theta \leq 1$, it follows that we can discard $\sigma$ in \eqref{main-eq25} and \eqref{main-eq27} i.\,e.\ we get \begin{equation}
		f \in \mc{C}_b^{\kappa(\cdot)+\min\{\theta,\lambda\}-\eps}(\mbb{R}^d), \qquad \eps \in (0,\kappa_0). \label{main-eq29}
	\end{equation}
	We would like to point out that it is, in general, \emph{not} possible to improve this estimate and to obtain that $f \in \mc{C}_b^{\kappa(\cdot)+\lambda-\eps}(\mbb{R}^d)$, $\eps \in (0,\kappa_0)$. To see this consider a Feller process $(X_t)_{t \geq 0}$ with symbol $q(x,\xi) = i b(x) \xi$, $x, \xi \in \mbb{R}$, for a mapping $b \in C_b(\mbb{R}^d)$ with $\inf_x b(x)>0$. If we define \begin{equation*}
		f(x) := \int_0^x \frac{1}{b(y)} \, dy, \qquad x \in \mbb{R}^d,
	\end{equation*}
	then $A_e f= b \, f'=1$ is smooth. However, the regularity of $f$ clearly depends on the regularity of $b$, \begin{equation*}
		\text{regularity of $f$} \, \approx 1 + \, \text{regularity of $b$}
	\end{equation*}
	which means that $f$ is \emph{less} regular than $A_e f$. 
	\item\label{feller-125-ii} It suffices to check \eqref{main-eq19} for $\lambda=\Lambda$; for $\lambda \in (0,\Lambda)$ the inequality then follows from the interpolation theorem, see e.\,g. \cite[Section 1.3.3]{triebel78} or \cite[Theorem 1.6]{lunardi}, and the fact that $\mc{C}_b^{\gamma}(\mbb{R}^d)$ can be written as a real interpolation space, see \cite[Theorem 2.7.2.1]{triebel78} for details.
	\item\label{feller-125-v} \eqref{main-eq17} is an assumption on the regularity of $z \mapsto \nu^{(x)}(z,dy)$. If $\nu^{(x)}(z,dy)$ has a density, say $m^{(x)}(z,y)$, with respect to Lebesgue measure, then a sufficient condition for \eqref{main-eq17} is \begin{equation*}
		\int_{y \neq 0} \min\{1,|y|^{\alpha^{(x)}(z)+r}\} |m^{(x)}(z,y)-m^{(x)}(z+h,y)| \, dy \leq C_{3,r} |h|^{\theta}.
	\end{equation*}
	\item\label{feller-125-iii} Condition \eqref{S1} is not strictly necessary for the proof of Theorem~\ref{feller-12}; essentially we need suitable upper bounds for \begin{equation*}
		\int_{|y| \leq r} |y|^{\gamma} \, \nu^{(x)}(z,dy) \quad \text{and} \quad \int_{r < |y| \leq R} |y|^{\gamma} \, \nu^{(x)}(z,dy)
	\end{equation*}
	where $0<r<R<1$, $x,z \in \mbb{R}^d$ and $\gamma \in (0,3)$. 
	\item\label{feller-125-iv} In \eqref{S2} we assume that $\theta \leq 1$; this assumption can be relaxed. To this end, we have to replace in \eqref{main-eq16} and \eqref{main-eq17} the differences of first order, \begin{equation*}
		|b^{(x)}(z)-b^{(x)}(z+h)| \quad \text{and} \quad \left| \int u(y) \, \nu^{(x)}(z,dy) - \int u (y) \, \nu^{(x)}(z+h,dy) \right|,
	\end{equation*}
	by iterated differences of higher order, cf.\ \eqref{def-eq3}. This makes the proof more technical but the idea of the proof stays the same.
\end{enumerate}
\end{bem}

The proofs of the results, which we stated in this section, will be presented in Section~\ref{proofs}. 

\section{Applications} \label{ex}

In this section we apply the results from the previous section to various classes of Feller processes. We will study processes of variable order (Theorem~\ref{ex-21} and Corollary~\ref{ex-23}), random time changes of L\'evy processes (Proposition~\ref{ex-9}) and solutions to L\'evy-driven SDEs (Proposition~\ref{ex-5}). Our aim is to illustrate the range of applications, and therefore we do not strive for the greatest generality of the examples; we will, however, point the reader to possible extensions of the results which we present. We remind the reader of the notation \begin{equation*}
	\mc{C}_b^{\alpha(\cdot)+}(\mbb{R}^d) := \bigcup_{\eps>0} \mc{C}_b^{\alpha(\cdot)+\eps}(\mbb{R}^d) \qquad \quad \mc{C}_b^{\alpha(\cdot)-}(\mbb{R}^d) := \bigcap_{\eps>0} \mc{C}_b^{\max\{\alpha(\cdot)-\eps,0\}}(\mbb{R}^d)
\end{equation*}
which we introduced in Section~\ref{def}.\par \medskip
 
The first part of this section is devoted to isotropic stable-like processes, i.\,e.\ Feller processes $(X_t)_{t \geq 0}$ with symbol of the form $q(x,\xi) = |\xi|^{\alpha(x)}$. A sufficient condition for the existence of such a Feller process is that $\alpha: \mbb{R}^d \to (0,2]$ is H\"{o}lder continuous and bounded from below, cf.\ \cite[Theorem 5.2]{matters}. If $\alpha(\mbb{R}^d) \subseteq (0,2)$ then the infinitesimal generator $A$ of $(X_t)_{t \geq 0}$ satisfies \begin{equation*}
	Af(x) = c_{d,\alpha(x)} \int_{y \neq 0} \left( f(x+y)-f(x)-y \cdot \nabla f(x) \I_{(0,1)}(|y|) \right) \, \frac{1}{|y|^{d+\alpha(x)}} \,dy, \quad f \in C_c^{\infty}(\mbb{R}^d),
\end{equation*}
which means that $A$ is a fractional Laplacian of variable order, i.e. $A = - (-\Delta)^{\alpha(\cdot)/2}$. This makes $A$ -- and hence the stable-like process $(X_t)_{t \geq 0}$ -- an interesting object of study. To our knowledge there are no Schauder estimates for the Poisson equation $Af=g$ available in the existing literature. Using the results from the previous section, we are able to derive Schauder estimates for functions $f$ in the Favard space $F_1$ (and, hence in particular, for $f \in \mc{D}(A)$), cf.\ Theorem~\ref{ex-21}, as well as Schauder estimates for solutions to $Af=g$, cf.\ Corollary~\ref{ex-23} below.

\begin{thm} \label{ex-21}
	Let $(X_t)_{t \geq 0}$ be a Feller process with symbol $q(x,\xi) = |\xi|^{\alpha(x)}$ for a H\"{o}lder continuous function $\alpha: \mbb{R}^d \to (0,2)$ such that  \begin{equation*}
		0< \alpha_L := \inf_{x \in \mbb{R}^d} \alpha(x) \leq \sup_{x \in \mbb{R}^d} \alpha(x)<2 .
	\end{equation*}
	The associated Favard space $F_1$ of order $1$, cf.\ \eqref{fav}, satisfies \begin{equation*}
		F_1 \subseteq \mc{C}_b^{\alpha(\cdot)-}(\mbb{R}^d).
	\end{equation*}
	For any $\eps \in (0,\alpha_L)$ there exists a finite constant $C=C(\eps,\alpha)$ such that \begin{equation}
		\|f\|_{\mc{C}_b^{\alpha(\cdot)-\eps}(\mbb{R}^d)} \leq C (\|f\|_{\infty}+\|A_e f\|_{\infty}), \qquad f \in F_1, \label{ex-eq2}
	\end{equation}
	where $A_e$ denotes the extended generator of $(X_t)_{t \geq 0}$. In particular, \eqref{ex-eq2} holds for any $f$ in the domain $\mc{D}(A)$ of the (strong) generator of $(X_t)_{t \geq 0}$, and $\mc{D}(A) \subseteq \mc{C}_b^{\alpha(\cdot)-}(\mbb{R}^d)$.
\end{thm}

\begin{bem} \label{ex-22} \begin{enumerate}[wide, labelwidth=!, labelindent=0pt]
	\item\label{ex-22-i} Theorem~\ref{ex-21} allows us to obtain information on the regularity of the transition density $p(t,x,y)$ of $(X_t)_{t \geq 0}$. Since $p(t,\cdot,y) \in \mc{D}(A)$ for each $t>0$ and $y \in \mbb{R}^d$, cf.\ \cite[Corollary 3.6]{matters}, Theorem~\ref{ex-21} shows that $p(t,\cdot,y) \in  \mc{C}_b^{\alpha(\cdot)-}(\mbb{R}^d)$; in particular, $x \mapsto p(t,x,y)$ is differentiable at any $x \in \{\alpha>1\}$.  Moreover, $(\partial_t-A_x) p(t,x,y)=0$ entails by \cite[Theorem 3.8]{matters} that \begin{equation*}
		\|p(t,\cdot,y)\|_{\mc{C}_b^{\alpha(\cdot)-\eps}(\mbb{R}^d)} \leq C t^{-1-d/\alpha_L}, \qquad t \in (0,T), \,y \in \mbb{R}^d,
	\end{equation*}
	for a finite constant $C=C(\eps,\alpha,T)$. Some related results on the regularity of the transition density were recently obtained in \cite{chen18}.
	\item\label{ex-22-ii} Theorem~\ref{ex-21} gives a necessary condition for a function $f \in C_{\infty}(\mbb{R}^d)$ to be in the domain $\mc{D}(A)$ of the infinitesimal generator; sufficient conditions were established in \cite[Example 5.5]{ihke}. Combining both results it should be possible to show that $\mc{D}(A)$ is an algebra, i.\,e.\ $f,g \in \mc{D}(A)$ implies $f \cdot g \in \mc{D}(A)$, and that \begin{equation*}
		A(f \cdot g)= f Ag + g Af + \Gamma(f,g), \qquad f,g \in \mc{D}(A),
	\end{equation*}
	see \cite[Proof of Theorem 4.3(iii)]{reg-levy} for the idea of the proof; here \begin{equation*}
		\Gamma(f,g)(x) := c_{d,\alpha(x)} \int_{y \neq 0} \left( f(x+y)-f(x) \right) \left( g(x+y)-g(x) \right) \frac{1}{|y|^{d+\alpha(x)}} \, dy
	\end{equation*}
	is the so-called Carr\'e du Champ operator, cf.\ \cite{bouleau91,meyer4}, and $\nu(x,dy) = c_{d,\alpha(x)} |y|^{-d-\alpha(x)} \, dy$ is the family of L\'evy measures associated with the symbol $|\xi|^{\alpha(x)}$ via the L\'evy--Khintchine representation.
	\item Theorem~\ref{ex-21} can be generalized to a larger class of ``stable-like'' Feller processes, e.\,g.\ relativistic stable-like processes and tempered stable-like processes, cf.\ \cite[Section 5.1]{matters} or \cite[Example 4.7]{parametrix} for the existence of such processes. In order to apply the results from Section~\ref{main} we need two key ingredients: general existence results  -- which ensure the existence of a ``nice'' Feller process $(Y_t)_{t \geq 0}$ whose symbol is ``truncated'' in a suitable way, cf.\ Step 1 in the proof of Theorem~\ref{ex-21} -- and certain heat kernel estimate which are needed to establish H\"older estimates for the semigroup; in \cite{matters} both ingredients were established for a wide class of stable-like processes.
\end{enumerate} \end{bem}

As a corollary of Theorem~\ref{ex-21} and Theorem~\ref{feller-12} we will establish the following Schauder estimates for the elliptic equation $Af=g$ associated with the infinitesimal generator $A$ of the isotropic stable-like process.

\begin{kor} \label{ex-23}
	Let $(X_t)_{t \geq 0}$ be a Feller process with infinitesimal generator $(A,\mc{D}(A))$ and symbol $q(x,\xi) = |\xi|^{\alpha(x)}$ for a mapping $\alpha: \mbb{R}^d \to (0,2)$ which satisfies \begin{equation}
		0 < \alpha_L := \inf_{x \in \mbb{R}^d} \leq \sup_{x \in \mbb{R}^d} \alpha(x)<2 \label{ex-eq0}
	\end{equation}
	and $\alpha \in C_b^{\gamma}(\mbb{R}^d)$ for some $\gamma \in (0,1)$. If $f \in \mc{D}(A)$ is such that \begin{equation*}
		Af = g \in \mc{C}_b^{\lambda}(\mbb{R}^d)
	\end{equation*}
	for some $\lambda>0$, then $f \in \mc{C}_b^{(\alpha(\cdot)+ \min\{\lambda,\gamma\})-}(\mbb{R}^d)$. For any $\eps \in (0,\alpha_L)$ there exists a constant $C_{\eps}>0$ (not depending on $f$, $g$) such that \begin{equation}
		\|f\|_{\mc{C}_b^{\alpha(\cdot)+\min\{\lambda,\gamma\}-\eps}(\mbb{R}^d)} \leq C_{\eps}\left( \|Af\|_{\mc{C}_b^{\min\{\lambda,\gamma\}}(\mbb{R}^d)} + \|f\|_{\infty}\right).  \label{ex-eq1}
	\end{equation}
\end{kor}

It is possible to extend Corollary~\ref{ex-23} to a larger class of ``stable-like'' processes, see also Remark~\ref{ex-22}\eqref{ex-22-ii}. Let us give some remarks on the assumption that $\alpha \in C_b^{\gamma}(\mbb{R}^d)$ for $\gamma \in (0,1)$.

\begin{bem} \label{ex-24} \begin{enumerate}[wide, labelwidth=!, labelindent=0pt]
	\item\label{ex-24-i} Let $\alpha$ be Lipschitz continuous function satisfying \eqref{ex-eq0}. Since $\alpha \in \mc{C}_b^{1-\eps}(\mbb{R}^d)$ for any $\eps \in (0,1)$, the Schauder estimate \eqref{ex-eq1} holds with $\gamma=1-\eps/2$ and $\eps \rightsquigarrow \eps/2$, and this entails that \eqref{ex-eq1} holds with $\gamma=1$. This means that Corollary~\ref{ex-23} remains valid for Lipschitz continuous functions (with $\gamma=1$ in \eqref{ex-eq1}).
	\item If $\alpha \in C_b^{\gamma}(\mbb{R}^d)$ for $\gamma>1$, we can apply Corollary~\ref{ex-23} with $\gamma=1$ but this gives a weaker regularity estimate for $f$ than we would expect; this is because we lose some information on the regularity of $\alpha$.  The reason why we have to restrict ourselves to $\gamma \in (0,1)$ is that two tools which we need for the proof (Theorem~\ref{feller-12} and Proposition~\ref{app-25}) are only available for $\gamma \in (0,1)$. However, we believe that both results are valid for $\gamma>0$, and that, hence, that the assumption $\gamma \in (0,1)$ in Corollary~\ref{ex-23} can be dropped. 
\end{enumerate} \end{bem}

Since the proofs of Theorem~\ref{ex-21} and Corollary~\ref{ex-23} are quite technical, we defer them to Section~\ref{iso}. The idea is to apply  Theorem~\ref{feller-9} and Theorem~\ref{feller-12}. As ``localizing'' process $(Y_t^{(x)})_{t \geq 0}$ we will use a Feller process with symbol \begin{equation*}
	p^{(x)}(z,\xi) := |\xi|^{\alpha^{(x)}(z)}, \qquad z,\xi \in \mbb{R}^d
\end{equation*}
where \begin{equation*}
	\alpha^{(x)}(z) := (\alpha(x)-\eps) \vee \alpha(z) \wedge (\alpha(x)+\eps), \qquad z \in \mbb{R}^d,
\end{equation*}
for fixed $x \in \mbb{R}^d$ and small $\eps>0$. In order to apply the results from the previous section,  we need suitable regularity estimates for the semigroup $(P_t)_{t \geq 0}$ associated with an isotropic stable-like process $(Y_t)_{t \geq 0}$. We will study the regularity of $x \mapsto P_t u(x)$ using the parametrix construction of (the transition density of) $(Y_t)_{t \geq 0}$ in \cite{matters}; the results are of independent interest, we refer the reader to Subsection~\ref{iso-reg}.  \par \medskip

Next we study Feller processes with symbols of the particular form $q(x,\xi) = m(x) |\xi|^{\alpha}$. They can be constructed as random time changes of isotropic $\alpha$-stable L\'evy processes, see e.\,g.\ \cite[Section 4.1]{ltp} and \cite{perpetual} for further details. This class of Feller processes includes, in particular, solutions to SDEs \begin{equation*}
	dX_t = \sigma(X_{t-}) \, dL_t, \qquad X_0 = x
\end{equation*}
driven by a one-dimensional isotropic $\alpha$-stable L\'evy process $(L_t)_{t \geq 0}$, $\alpha \in (0,2]$; for instance, if $\sigma>0$ is continuous and at most of linear growth, then there exists a unique weak solution to the SDE,  and the solution is a Feller process with symbol $q(x,\xi) = |\sigma(x)|^{\alpha} |\xi|^{\alpha}$,  cf.\ \cite[Example 5.4]{mp}.

\begin{prop} \label{ex-9}
	Let $(X_t)_{t \geq 0}$ be a Feller process with symbol $q(x,\xi)= m(x) |\xi|^{\alpha}$ for $\alpha \in (0,2)$ and a H\"{o}lder continuous function $m: \mbb{R}^d \to (0,\infty)$ such that \begin{equation*}
		0 < \inf_{x \in \mbb{R}^d} m(x) \leq \sup_{x \in \mbb{R}^d} m(x) < \infty.
	\end{equation*}
	\begin{enumerate}
		\item\label{ex-9-i} The infinitesimal generator $(A,\mc{D}(A))$ and the Favard space $F_1$ of order $1$ satisfy
		\begin{equation*}
			\mc{C}_{\infty}^{\alpha+}(\mbb{R}^d)  \subseteq \mc{D}(A) \subseteq F_1 \subseteq \mc{C}_b^{\alpha-}(\mbb{R}^d), 
		\end{equation*}
		where \begin{equation}
			\mc{C}_{\infty}^{\alpha+}(\mbb{R}^d) := \mc{C}_b^{\alpha+}(\mbb{R}^d) \cap C_{\infty}^{\floor{\alpha}}(\mbb{R}^d) = \begin{cases} \mc{C}_b^{\alpha}(\mbb{R}^d) \cap C_{\infty}(\mbb{R}^d), & \alpha \in (0,1) \\ \mc{C}_b^{\alpha}(\mbb{R}^d) \cap C_{\infty}^1(\mbb{R}^d), & \alpha \in [1,2). \end{cases} \label{ex-eq30}
		\end{equation}
		For any $\kappa \in (0,\alpha)$ there exists a finite constant $C_1>0$ such that \begin{equation}
			\|f\|_{\mc{C}_b^{\kappa}(\mbb{R}^d)} \leq C_1 (\|f\|_{\infty}+\|A_e f\|_{\infty}) \fa f \in F_1; \label{ex-eq31}
		\end{equation}
		here $A_e$ denotes the extended infinitesimal generator.
		\item\label{ex-9-ii} Let $\theta \in (0,1]$ be such that $m \in C_b^{\theta}(\mbb{R}^d)$. If $f \in \mc{D}(A)$ is such that $Af=g \in \mc{C}_b^{\lambda}(\mbb{R}^d)$ for some $\lambda>0$, then $f \in \mc{C}_b^{(\alpha+\min\{\lambda,\theta\})-}(\mbb{R}^d)$ and for any $\kappa \in (0,\alpha)$ there exists a constant $C_2>0$ (not depending on $f$, $Af$) such that \begin{equation*}
			\|f\|_{\mc{C}_b^{\kappa+\min\{\lambda,\theta\}}(\mbb{R}^d)} \leq C_2 \left( \|f\|_{\infty} + \|Af\|_{\mc{C}_b^{\min\{\lambda,\theta\}}(\mbb{R}^d)} \right).
		\end{equation*}
	\end{enumerate}
\end{prop} 

\begin{proof}
	It follows from \cite[Theorem 3.3]{matters} that there exists a unique Feller process $(X_t)_{t \geq 0}$ with symbol $q(x,\xi)=m(x) |\xi|^{\alpha}$, $x,\xi \in \mbb{R}^d$. Using a very similar reasoning as in the proof of Proposition~\ref{app-1} and Proposition~\ref{app-25}, it follows from the parametrix construction of the transition density $p$ in \cite{matters} that the semigroup $(P_t)_{t \geq 0}$ satisfies \begin{equation*}
		\|P_t u\|_{\mc{C}_b^{\kappa}(\mbb{R}^d)} \leq c_{1,\kappa} t^{-\kappa/\alpha} \|u\|_{\infty}, \qquad u \in \mc{B}_b(\mbb{R}^d),\, t \in (0,1),
	\end{equation*}
	and \begin{equation*}
		\|P_t u\|_{\mc{C}_b^{\kappa + \lambda}(\mbb{R}^d)} \leq c_{2,\kappa} t^{-\kappa/\alpha} \|u\|_{\mc{C}_b^{\lambda}(\mbb{R}^d)}, \quad u \in \mc{C}_b^{\lambda}(\mbb{R}^d),\,t \in (0,1),
	\end{equation*}
	for any $\kappa \in (0,\alpha)$ and $\lambda \in [0,\theta]$; for the particular case $\alpha \in (0,1]$ the first inequality follows from \cite{wang18-2}. Applying Proposition~\ref{feller-7} we get \eqref{ex-eq31}; in particular $F_1 \subseteq \mc{C}_b^{\alpha-}(\mbb{R}^d)$. The inclusion $\mc{C}_{\infty}^{\alpha+}(\mbb{R}^d) \subseteq \mc{D}(A)$ is a direct consequence of \cite[Example 5.4]{ihke}. The Schauder estimate in \eqref{ex-9-ii} follows Theorem~\ref{feller-12} applied with $Y_t^{(x)} := X_t$ for all $x \in \mbb{R}^d$ (using the regularity estimates for $(P_t)_{t \geq 0}$ from above).
\end{proof}

\begin{bem}[{Possible extensions of Proposition~\ref{ex-9}}] \label{ex-11}  \begin{enumerate}[wide, labelwidth=!, labelindent=0pt]
	\item\label{ex-11-i} Proposition~\ref{ex-9} can be extended to symbols $q(x,\xi) = m(x) \psi(\xi)$ for ``nice'' continuous negative definite functions $\psi$, e.\,g.\ the characteristic exponent of a relativistic stable or tempered stable L\'evy process, cf.\ \cite[Table 5.2]{matters} for further examples.
	\item\label{ex-11-ii} The family of L\'evy kernels associated with $(X_t)_{t \geq 0}$ is of the form $\nu(x,dy) = m(x) |y|^{-d-\alpha} \, dy$. More generally, it is possible to consider Feller processes with L\'evy kernels $\nu(x,dy) = m(x,y) \, \nu(dy)$, for instance \cite{bogdan17,wang18-2,szcz18} establish existence results as well as H\"{o}lder estimates under suitable assumptions on $m$ and $\nu$. Combining the results with Proposition~\ref{feller-7} we can obtain Schauder estimates for functions in the domain of the infinitesimal generator of $(X_t)_{t \geq 0}$. Let us mention that for $\nu(x,y) = m(x,y) |y|^{-d-\alpha} \, dy$ Schauder estimates were studied in \cite{bass08}.
\end{enumerate} \end{bem}

We close this section with some results on solutions to L\'evy-driven SDEs.

\begin{prop} \label{ex-5}
	Let $(L_t)_{t \geq 0}$ be a $1$-dimensional isotropic $\alpha$-stable L\'evy process, $\alpha \in (0,2)$. Consider the SDE \begin{equation}
		dX_t = b(X_{t-}) \, dt + \sigma(X_{t-}) \, dL_t, \qquad X_0 = x, \label{ex-eq3}
	\end{equation}
	for a bounded $\beta$-H\"{o}lder continuous mapping $b: \mbb{R} \to \mbb{R}$ and a bounded Lipschitz continuous mapping $\sigma: \mbb{R} \to (0,\infty)$. If \begin{equation}
		\beta+\alpha>1 \quad \text{and} \quad \sigma_L :=\inf_{x \in \mbb{R}} \sigma(x)>0, \label{ex-eq4}
	\end{equation}
	then there exists a unique weak solution $(X_t)_{t \geq 0}$ to \eqref{ex-eq3}, and it gives rise to a Feller process with infinitesimal generator $(A,\mc{D}(A))$. The associated Favard space $F_1$ of order $1$ satisfies \begin{equation*}
		\mc{D}(A) \subseteq F_1 \subseteq \bigcap_{k \in \mbb{N}} \mc{C}_b^{\min\{1,\alpha-1/k\}}(\mbb{R}),
	\end{equation*}
	and there exists for any $k \in \mbb{N}$ a finite constant $C>0$ such that 
	\begin{equation}
		\|f\|_{\mc{C}_b^{\min\{\alpha-1/k,1\}}(\mbb{R})} \leq C (\|f\|_{\infty}+\|A_e f\|_{\infty}) \fa f \in F_1 \label{ex-eq7}
	\end{equation}
	where $A_e$ denotes the extended generator.  In particular, \eqref{ex-eq7} holds for any $f \in \mc{D}(A)$ with $A_e f=Af$.
\end{prop}

\begin{proof}
	It follows from \eqref{ex-eq4} that the SDE \eqref{ex-eq3} has a unique weak solution $(X_t)_{t \geq 0}$ for any $x \in \mbb{R}$, cf.\ \cite{kulik15}. By \cite{schnurr}, see also \cite{sde}, $(X_t)_{t \geq 0}$ is a Feller process. Moreover, \cite{wang18} shows that for any $\kappa<\alpha$ there exists a constant $c>0$ such that the semigroup $(P_t)_{t \geq 0}$ satisfies \begin{equation*}
		\|P_t u\|_{\mc{C}_b^{\kappa \wedge 1}(\mbb{R})}\leq c \|u\|_{\infty} t^{-\kappa/\alpha}
	\end{equation*}
	for all $t \in (0,1)$ and $u \in \mc{B}_b(\mbb{R})$. Applying Proposition~\ref{feller-7} proves the assertion.
\end{proof}

Before giving some remarks on possible extensions of Proposition~\ref{ex-5}, let us mention that sufficient conditions for a function $f$ to be in the domain $\mc{D}(A)$ were studied in \cite{ihke}; for instance if the SDE has no drift part, i.\,e.\ $b=0$, then it follows from Proposition~\ref{ex-5} and \cite[Example 5.6]{ihke} that \begin{equation}
	\mc{C}^{\alpha+}_{\infty}(\mbb{R}) \subseteq \mc{D}(A) \subseteq \mc{C}_b^{\alpha-}(\mbb{R}) \quad \text{if $\alpha \in (0,1]$} \label{ex-eq11}
\end{equation}
and \begin{equation}
	\mc{C}^{\alpha+}_{\infty}(\mbb{R}) \subseteq \mc{D}(A) \subseteq \mc{C}_b^{1}(\mbb{R}) \quad \text{if $\alpha \in (1,2)$;}  \label{ex-eq13}
\end{equation}
see \eqref{ex-eq30} for the definition of $\mc{C}^{\alpha+}_{\infty}(\mbb{R})$.  Intuitively one would expect that \eqref{ex-eq11} holds for $\alpha \in (0,2)$. If we knew that the semigroup $(P_t)_{t \geq 0}$ of the solution to \eqref{ex-eq3} satisfies \begin{equation}
	 \|P_t u\|_{\mc{C}_b^{\kappa}(\mbb{R})} \leq c t^{-\kappa/ \alpha} \|u\|_{\infty}, \qquad u \in \mc{B}_b(\mbb{R}), \,t \in (0,1), \,\kappa \in (0,\alpha), \label{ex-eq15}
\end{equation}
for some constant $c=c(\kappa)>0$, this would immediately follow from Proposition~\ref{feller-7}.  We could not find \eqref{ex-eq15} in the literature but we strongly believe that the parametrix construction of the transition density in \cite{kulik15} can be used to establish such an estimate; this is also indicated by the proof of Theorem~\ref{ex-21} (see in particular the proof of Proposition~\ref{app-1}). In fact, we are positive that the parametrix construction in \cite{kulik15} entails estimates of the form
 \begin{equation*}
	 \|P_t u\|_{\mc{C}_b^{\kappa+\min\{\lambda,\beta\}}(\mbb{R})} \leq c t^{-\kappa/ \alpha} \|u\|_{\mc{C}_b^{\min\{\lambda,\beta\}}(\mbb{R})}, \qquad u \in \mc{C}_b^{\lambda}(\mbb{R}), \,t \in (0,1), \,\kappa \in (0,\alpha), \lambda>0
\end{equation*}
(recall that $\beta$ is the H\"older exponent of the drift $b$) which would then allow us to establish Schauder estimates to the equation $Af=g$ for $g \in \mc{C}_b^{\lambda}(\mbb{R})$ using Theorem~\ref{feller-12}. 

\begin{bem}[{Possible extensions of Proposition~\ref{ex-5}}] \label{ex-7} \begin{enumerate}[wide, labelwidth=!, labelindent=0pt]
		\item\label{ex-7-i} The gradient estimates in \cite{wang18} were obtained under more general conditions, and (the proof of) Proposition~\ref{ex-5} extends naturally to this more general framework. Firstly, Proposition~\ref{ex-5} can be extended to higher dimensions; the assumption $\sigma_L>0$ in \eqref{ex-eq4} is then replaced by the assumption that $\sigma$ is uniformly non-degenerate in the sense that \begin{equation*}
			M^{-1} |\xi| 
			\leq \inf_{x \in \mbb{R}^d} \min\{|\sigma(x) \xi|, |\sigma(x)^{-1} \xi|\} 
			\leq \sup_{x \in \mbb{R}^d} \max\{|\sigma(x) \xi|, |\sigma(x)^{-1} \xi|\}
			\leq M |\xi|
		\end{equation*}
		for some absolute constant $M>0$ which does not depend on $\xi \in \mbb{R}^d$. Secondly, Proposition~\ref{ex-5} holds for a larger class of driving L\'evy processes; it suffices to assume that the L\'evy measure $\nu$ satisfies $\nu(dz) \geq c |z|^{-d-\alpha} \I_{\{|z| \leq \eta\}}$ for some $c,\eta>0$ and that the SDE \eqref{ex-eq3} has a unique weak solution. Under the stronger balance condition $\beta+\alpha/2>1$ this is automatically satisfied  for a large class of L\'evy processes, e.g.\ if $(L_t)_{t \geq 0}$ is an relativistic stable or a tempered stable L\'evy process, cf.\ \cite{chen17}.
		\item\label{ex-7-iii} Recently, Kulczycki et al.\ \cite{kulc18} established H\"{o}lder estimates for the semigroup associated with the solution to the SDE \begin{equation*}
			dX_t = \sigma(X_{t-}) \, dL_t
		\end{equation*}
		driven by a $d$-dimensional L\'evy process $(L_t)_{t \geq 0}$, $d \geq 2$, whose components are independent $\alpha$-stable L\'evy processes, $\alpha \in (0,1)$, under the assumption that $\sigma: \mbb{R}^d \to \mbb{R}^{d \times d}$ is bounded, Lipschitz continuous and satisfies $\inf_x \det(\sigma(x))>0$. Combining the estimates with Proposition~\ref{feller-7} we find that the assertion of Proposition~\ref{ex-5} remains valid in this framework, i.e.\ the Favard space $F_1$ associated with the unique solution $(X_t)_{t \geq 0}$ satisfies $F_1 \subseteq \mc{C}_b^{\alpha-}(\mbb{R}^d)$ and \begin{equation*}
			\|f\|_{\mc{C}_b^{\alpha-1/k}(\mbb{R}^d)} \leq C_k (\|f\|_{\infty}+\|A_e f\|_{\infty}), \qquad f \in F_1.
		\end{equation*}
		\item\label{ex-7-iv} Using coupling methods, Liang et. al \cite{wang18-3} recently studied the regularity of semigroups associated with solutions to SDEs with additive noise \begin{equation*}
			dX_t = b(X_{t-}) \, dt+ dL_t
		\end{equation*}
		for a large class of driving L\'evy processes $(L_t)_{t \geq 0}$. The results from \cite{wang18-3} and Section~\ref{main} can be used to obtain Schauder estimates for functions in the domain of the infinitesimal generator of $(X_t)_{t \geq 0}$.
		
\end{enumerate} \end{bem} 

\section{Proofs of results from Section~\ref{main}} \label{proofs}

For the proof of Proposition~\ref{feller-7} we use the following lemma which shows how H\"{o}lder estimates for a Feller semigroup translate to regularity properties of the $\lambda$-potential operator \begin{equation*}
	R_{\lambda} u := \int_{(0,\infty)} e^{-\lambda t} P_t u \, dt, \qquad u \in \mc{B}_b(\mbb{R}^d), \, \lambda>0.
\end{equation*}

\begin{lem} \label{feller-3}
	Let $(X_t)_{t \geq 0}$ be a Feller process with semigroup $(P_t)_{t \geq 0}$ and $\lambda$-potential operators $(R_{\lambda})_{\lambda>0}$. \begin{enumerate}
		\item\label{feller-3-i} If there exist $T>0$, $M \geq 0$, $\kappa \geq 0$ and $\beta \geq 0$ such that \begin{equation*}
		\|P_t u\|_{\mc{C}_b^{\kappa}(\mbb{R}^d)} \leq M t^{-\beta} \|u\|_{\infty}
		\end{equation*}
		for all $t \in (0,T)$ and $u \in \mc{B}_b(\mbb{R}^d)$, then \begin{equation}
		\|P_t u\|_{\mc{C}^{\kappa}_b(\mbb{R}^d)} \leq M e^{mt} t^{-\beta} \|u\|_{\infty} \label{proofs-eq3}
		\end{equation} 
		for all $t>0$ and $u \in \mc{B}_b(\mbb{R}^d)$ where $m := \log(2) \beta/T$.
		\item\label{feller-3-ii}  If $u \in \mc{B}_b(\mbb{R}^d)$ is such that \eqref{proofs-eq3} holds for some $\beta \in [0,1)$, then $R_{\lambda} u \in \mc{C}_b^{\kappa}(\mbb{R}^d)$ for any $\lambda>m$ and \begin{equation*}
		\|R_{\lambda} u\|_{\mc{C}_b^{\kappa}(\mbb{R}^d)} \leq \|u\|_{\infty}  \left( \frac{1}{\lambda-m} + \frac{1}{1-\beta} \right) (M+1).
		\end{equation*} \end{enumerate} 
\end{lem}

\begin{proof}
	\firstpara{\eqref{feller-3-i}} By the contraction property of $(P_t)_{t \geq 0}$, we have $\|P_{t} u\|_{\mc{C}_b^{\kappa}(\mbb{R}^d)} \leq \|P_{t/2} u\|_{\mc{C}_b^{\kappa}(\mbb{R}^d)}$ for all $t \geq 0$, and so \begin{equation*}
		\|P_{t} u\|_{\mc{C}_b^{\kappa}(\mbb{R}^d)} \leq M \left( \frac{t}{2} \right)^{-\beta} = M 2^{\beta} t^{-\beta} \fa t \in (0,2T).
	\end{equation*}
	Iterating the procedure, it follows easily that \eqref{proofs-eq3} holds. 
	
	\para{\eqref{feller-3-ii}} Let $u \in \mc{B}_b(\mbb{R}^d)$ be such that \eqref{proofs-eq3} holds for some $\beta <1$. If we choose $K>\kappa$, then \eqref{proofs-eq3} gives that the iterated difference operator $\Delta_h^K$, cf.\ \eqref{def-eq3}, satisfies \begin{equation*}
	|\Delta_h^K P_t u(x)| \leq M e^{mt} t^{-\beta} \|u\|_{\infty} |h|^{\kappa}
	\end{equation*}
	for any $x \in \mbb{R}^d$ and $|h| \leq 1$. Since, by the linearity of the integral, \begin{equation*}
	\Delta_h^K R_{\lambda} u(x) = \int_{(0,\infty)} e^{-\lambda t} \Delta_h^K P_t u(x) \, dt
	\end{equation*}
	we find that \begin{equation*}
	|\Delta_h^K R_{\lambda} u(x)| \leq M |h|^{\kappa} \|u\|_{\infty} \int_{(0,\infty)} e^{-t(\lambda-m)} t^{-\beta} \, dt.
	\end{equation*}
	On the other hand, we have $\|R_{\lambda} u\|_{\infty} \leq \lambda^{-1} \|u\|_{\infty}$, and therefore we get for all $\lambda>m$ \begin{equation*} 
	\|R_{\lambda} u\|_{\mc{C}_b^{\kappa}(\mbb{R}^d)}
	\leq \lambda^{-1} \|u\|_{\infty} + M \|u\|_{\infty} \left( \int_0^1 t^{-\beta} \, dt + \int_1^{\infty} e^{-t(\lambda-m)} \, dt \right)
	\end{equation*}
	which proves the assertion. 
\end{proof}

We are now ready to prove Proposition~\ref{feller-7}.

\begin{proof}[Proof of Proposition~\ref{feller-7}]
	It follows from Lemma~\ref{feller-3}\eqref{feller-3-i} that \eqref{proofs-eq3} holds with $m := \log(2) \beta/T$ for any $u \in \mc{B}_b(\mbb{R}^d)$. If we set $\lambda := 2m$ and $u := \lambda f-A_ef$ for $f \in F_1$, then $f = R_{\lambda} u$. Applying Lemma~\ref{feller-3}\eqref{feller-3-ii} we find that \begin{equation*}
	\|f\|_{\mc{C}_b^{\kappa}(\mbb{R}^d)}
	= \|R_{\lambda} u\|_{\mc{C}_b^{\kappa}(\mbb{R}^d)}
	\leq K \|u\|_{\infty}
	\leq \lambda K  \|f\|_{\infty} + K \|A_ef\|_{\infty}
	\end{equation*}
	for $K := 2m^{-1} + (1-\beta)^{-1}$.
\end{proof}

For the proof of Theorem~\ref{feller-9} we need two auxiliary results.

\begin{lem} \label{feller-13}
	Let $(X_t)_{t \geq 0}$ and $(Y_t)_{t \geq 0}$ be Feller processes with infinitesimal generator $(A,\mc{D}(A))$ and $(L,\mc{D}(L))$, respectively, such that \begin{equation*}
	Af(z) = -q(z,D) f(z) \quad \text{and} \quad Lf(z) = -p(z,D) f(z) \fa f \in C_c^{\infty}(\mbb{R}^d), \,z \in \mbb{R}^d,
	\end{equation*}
	cf.\ \eqref{pseudo}, and assume that the $(A,C_c^{\infty}(\mbb{R}^d))$-martingale problem is well-posed. Let $U \subseteq \mbb{R}^d$ be an open set such that \begin{equation*}
	p(z,\xi) = q(z,\xi) \fa z \in U, \,\xi \in \mbb{R}^d.
	\end{equation*}
	If $x \in U$ and $r>0$ are such that $\overline{B(x,r)} \subseteq U$, then for the stopping times \begin{equation}
	\tau^X := \inf\{t>0; |X_t-x|>r \} \qquad
	\tau^Y := \inf\{t>0; |Y_t-x|>r\}  \label{feller-eq25}
	\end{equation}
	the random variables $X_{t \wedge \tau^X}$ and $Y_{t \wedge \tau^Y}$ are equal in distribution with respect to $\mbb{P}^x$ for any $t \geq 0$.\footnote{Here and below we are a bit sloppy in our notation. The Feller processes $(X_t)_{t \geq 0}$ and $(Y_t)_{t \geq 0}$ each come with a family of probability measures, i.e.\,their semigroups are of the form $\int f(X_t) \, \mbb{P}^x(dy)$ and $\int f(Y_t) \, \tilde{\mbb{P}}^x(dy)$, respectively, for families of probability measures $(\mbb{P}^x)_{x \in \mbb{R}^d}$ and $(\tilde{\mbb{P}}^x)_{x \in \mbb{R}^d}$. To keep the notation simple, we will not distinguish these two families. Formally written, the assertion of Lemma~\ref{feller-12} reads $\mbb{P}^x(X_{t \wedge \tau^X} \in \cdot) = \tilde{\mbb{P}}^x(Y_{t \wedge \tau^Y} \in \cdot)$.}
\end{lem}

\begin{proof}
	Set
	\begin{equation*}
	\sigma^X := \inf\{t>0; X_{t} \notin U \, \, \text{or} \, \, X_{t-} \notin U\} \qquad
	\sigma^Y := \inf\{t>0; Y_{t} \notin U \, \, \text{or} \, \, Y_{t-} \notin U\}
	\end{equation*}
	It follows from the well-posedness of the $(A,C_c^{\infty}(\mbb{R}^d))$-martingale problem that the local martingale problem for $U$ is well-posed, cf.\ \cite[Theorem 4.6.1]{ethier} or \cite{hoh} for details. On the other hand, Dynkin's formula shows that both $(X_{t \wedge \sigma^X})_{t \geq 0}$ and $(Y_{t \wedge \sigma^Y})_{t \geq 0}$ are solutions to the local martingale problem, and therefore $(X_{t \wedge \sigma^X})_{t \geq 0}$ equals in distribution $(Y_{t \wedge \sigma^Y})_{t \geq 0}$ with respect to $\mbb{P}^x$ for any $x \in U$. If $x \in U$ and $r>0$ are such that $\overline{B(x,r)} \subseteq U$, then it follows from the definition of $\tau^X$ and $\tau^Y$ that $\tau^X \leq \sigma^X$ and $\tau^Y \leq \sigma_U^Y$; in particular, \begin{equation*}
	X_{t \wedge \tau^X} = X_{t \wedge \tau^X \wedge \sigma^X} \quad \text{and} \quad Y_{t \wedge \tau^Y} = Y_{t \wedge \tau^Y \wedge \sigma^Y}.
	\end{equation*}
	Approximating $\tau^X$ and $\tau^Y$ from above by sequences of discrete-valued stopping times, we conclude from $(X_{t \wedge \sigma^X})_{t \geq 0} \stackrel{d}{=} (Y_{t \wedge \sigma^Y})_{t \geq 0}$ that $X_{t \wedge \tau^X} \stackrel{d}{=} Y_{t \wedge \tau^Y}$. 
\end{proof}

\begin{lem} \label{feller-15}
	Let $(Y_t)_{t \geq 0}$ be a Feller process with infinitesimal generator $(A,\mc{D}(A))$ and symbol \begin{equation*}
	p(x,\xi) = -i b(x) \cdot \xi + \int_{y \neq 0} \left( 1- e^{iy \cdot \xi} + i y \cdot \xi \I_{(0,1)}(|y|) \right) \, \nu(x,dy), \qquad x,\xi \in \mbb{R}^d.
	\end{equation*}
	If $\alpha>1$ and $U  \in \mc{B}(\mbb{R}^d)$ are such that \begin{equation*}
	\sup_{z \in U} \left( |b(z)| + \int_{y \neq 0} \min\{1,|y|^{\alpha}\} \, \nu(z,dy) \right)<\infty,
	\end{equation*}
	then there exists an absolute constant $c>0$ such that the stopped process $(Y_{t \wedge \tau_U})_{t \geq 0}$, \begin{equation*}
	\tau_U := \inf\{t \geq 0; Y_t \notin U\},
	\end{equation*}
	satisfies \begin{align}
	\mbb{E}^x(|Y_{t \wedge \tau_U}-x|^{\alpha} \wedge 1) \leq c t \sup_{z \in U} \left( |b(z)| + \int_{y \neq 0} \min\{1,|y|^{\alpha}\} \, \nu(z,dy) \right) \label{feller-eq28}
	\end{align}
	for all $x \in U$, $t \geq 0$.
\end{lem}

Note that \eqref{feller-eq28} implies, by Jensen's inequality,  that the moment estimate \begin{equation}
	\mbb{E}^x(|Y_{t \wedge \tau_U}-x|^{\beta} \wedge 1) \leq c' t^{\beta/\alpha} \sup_{z \in U} \left( |b(z)| + \int_{y \neq 0} \min\{1,|y|^{\alpha}\} \, \nu(z,dy) \right)^{\beta/\alpha} 
\end{equation}
holds for any $\beta \in [0,\alpha]$, $x \in U$ and $t \geq 0$. If $(Y_t)_{t \geq 0}$ has a compensated drift, in the sense that $b(z) = \int_{|y|<1} y \, \nu(z,dy)$ for all $z \in U$, then Lemma~\ref{feller-15} holds also for $\alpha \in (0,1]$. Let us mention that estimates for fractional moments of Feller processes were studied in \cite{moments}; it is, however, not immediate how Lemma~\ref{feller-15} can be derived from the results in \cite{moments}.

\begin{proof}[Proof of Lemma~\ref{feller-15}]
	Let $(f_k)_{k \in \mbb{N}} \subseteq \mc{C}_b^{\alpha}(\mbb{R}^d) \cap C_c(\mbb{R}^d)$ be such that $f_k \geq 0$, $f_k(z)=\min\{1,|z|^{\alpha}\}$ for $|z| \leq k$ and $M:=\sup_k \|f_k\|_{\mc{C}_b^{\alpha}}<\infty$. Pick a function $\chi \in C_c^{\infty}(\mbb{R}^d)$, $\chi \geq 0$ such that $\int_{\mbb{R}^d} \chi(x) \, dx=1$ and set $\chi_{\eps}(z) := \eps^{-1} \chi(\eps^{-1}z)$.  If we define for fixed $x \in U$ \begin{equation*}
	f_{k,\eps}(z) := (f_k(\cdot -x) \ast \chi_{\eps})(z) := \int_{\mbb{R}^d} f_k(z-x-y) \chi_{\eps}(y) \, dy, \qquad z \in \mbb{R}^d,
	\end{equation*}
	then $f_{k,\eps} \to f_k(\cdot-x)$ uniformly as $\eps \to 0$ and $\|f_{k,\eps}\|_{\mc{C}_b^{\alpha}(\mbb{R}^d)} \leq M$. As $f_{k,\eps} \in C_c^{\infty}(\mbb{R}^d) \subseteq \mc{D}(A)$ an application of Dynkin's formula shows that \begin{equation*}
	\mbb{E}^x f_{k,\eps}(Y_{t \wedge \tau_U})-f_{k,\eps}(x) = \mbb{E}^x \left( \int_{(0,t \wedge \tau_U)} Af_{k,\eps}(Y_s) \, ds \right)
	\end{equation*}
	for all $t \geq 0$. Since $\alpha>1$ there exists an absolute constant $C>0$ such that \begin{equation*}
		|\nabla f_{k,\eps}(z)| \leq C \|f_{k,\eps}\|_{\mc{C}_b^{\alpha}(\mbb{R}^d)} \leq C M
	\end{equation*} and \begin{align*}
	\left| f_{k,\eps}(z+y)-f_{k,\eps}(z)-y \cdot \nabla f_{k,\eps}(z) \I_{(0,1)}(|y|)\right| \leq C \|f_{k,\eps}\|_{\mc{C}_b^{\alpha}(\mbb{R}^d)} \min\{1,|y|^{\alpha}\}
	\end{align*}
	for all $z \in \mbb{R}^d$. This implies \begin{align*}
	|Af_{k,\eps}(z)| 
	&\leq |b(z)| \, |\nabla f_{k,\eps}(z)| + \int_{y \neq 0} \left| f_{k,\eps}(z+y)-f_{k,\eps}(z)-y \cdot \nabla f_{k,\eps}(z) \I_{(0,1)}(|y|) \right| \, \nu(z,dy) \\
	&\leq C M \left( |b(z)| + \int_{y \neq 0} \min\{1,|y|^{\alpha}\} \, \nu(z,dy) \right)
	\end{align*}
	for any $z \in U$. Hence, \begin{equation*}
	\mbb{E}^x f_{k,\eps}(Y_{t \wedge \tau_U}) \leq f_{k,\eps}(x)+ 2C Mt \sup_{z \in U} \left( |b(z)| + \int_{y \neq 0} \min\{1,|y|^{\alpha}\} \, \nu(z,dy) \right)
	\end{equation*}
	for $x \in U$. Applying Fatou's lemma twice we conclude that \begin{align*}
	\mbb{E}^x \min\{1,|Y_{t \wedge \tau_U}-x|^{\alpha}\} 
	&\leq \liminf_{k \to \infty} \liminf_{\eps \to 0} \mbb{E}^x f_{k,\eps}(Y_{t \wedge \tau_U})  \\
	&\leq 2C Mt \sup_{z \in U} \left( |b(z)| + \int_{y \neq 0} \min\{1,|y|^{\alpha}\} \, \nu(z,dy) \right). \qedhere
	\end{align*}
\end{proof}

We are now ready to prove Theorem~\ref{feller-9}.

\begin{proof}[Proof of Theorem~\ref{feller-9}]
	Since $x \in \mbb{R}^d$ is fixed throughout this proof, we will omit the superscript $x$ in the notation which we used in the statement of Theorem~\ref{feller-9}, e.g.\, we will write $(Y_t)_{t \geq 0}$ instead of $(Y_t^{(x)})_{t \geq 0}$, $L$ instead of $L^{(x)}$ etc. \par
	Denote by $(L_e,\mc{D}(L_e))$ the extended generator of $(Y_t)_{t \geq 0}$, and fix a truncation function $\chi \in C_c^{\infty}(\mbb{R}^d)$ such that $\I_{\overline{B(x,\delta)}} \leq \chi \leq \I_{\overline{B(x,2\delta)}}$ and $\|\chi\|_{C_b^2(\mbb{R}^d)} \leq 10 \delta^{-2}$. To prove the assertion it suffices by \eqref{C4} and Proposition~\ref{feller-7} to show that $v:=f \cdot \chi \in \mc{D}(L_e)$ and \begin{equation}
	\|L_e v\|_{\infty} 
	\leq C \left(\|A_e f\|_{\infty}+\|f\|_{\infty}+\|f\|_{\mc{C}_b^{\varrho(x)}(\overline{B(x,4\delta)})}\right) \label{feller-eq30}
	\end{equation}
	for a suitable constant $C>0$. The first -- and main-- step is to estimate \begin{equation}
	\sup_{t \in (0,1)} \frac{1}{t} \sup_{z \in \mbb{R}^d} |\mbb{E}^z v(Y_{t \wedge \tau_{\delta}^z})-v(z)| \label{feller-eq31}
	\end{equation}
	for the stopping time \begin{equation*}
	\tau_{\delta}^z := \inf\{t>0; |Y_t-z|> \delta\}.
	\end{equation*}
	We consider separately the cases  $z \in B(x,3\delta)$ and $z \in \mbb{R}^d \backslash B(x,3\delta)$. 	For $z \in \mbb{R}^d \backslash B(x,3\delta)$ it follows from $\spt \chi \subseteq \overline{B(x,2\delta)}$ that $v=0$ on $\overline{B(z,\delta)}$, and so \begin{equation*}
	v(Y_{t \wedge \tau_{\delta}^z}(\omega))-v(z)=0 \fa \omega \in \{\tau_{\delta}^z > t\}.
	\end{equation*}
	Hence, \begin{equation*}
	|\mbb{E}^z v(Y_{t \wedge \tau_{\delta}^z})-v(z)| 
	\leq 2\|v\|_{\infty} \mbb{P}^z(\tau_{\delta}^z \leq t).
	\end{equation*}
	Applying the maximal inequality \eqref{max} for Feller processes we find that there exists an absolute constant $c_1>0$ such that \begin{align*}
	|\mbb{E}^z v(Y_{t \wedge \tau_{\delta}^z})-v(z)| 
	&\leq c_1 t \|f\|_{\infty}  \sup_{y \in \mbb{R}^d} \sup_{|\xi| \leq \delta^{-1}} |p(y,\xi)|
	\end{align*}
	for all $z \in \mbb{R}^d \backslash B(x,3\delta)$; the right-hand side is finite since $p$ has, by assumption, bounded coefficients. \par
	For $z \in B(x,3\delta)$ we write \begin{equation*}
	|\mbb{E}^z v(Y_{t \wedge \tau_{\delta}^z})-v(z)| \leq I_1+I_2+I_3
	\end{equation*}
	for \begin{align*}
	I_1 &:= |\chi(z) \mbb{E}^z(f(Y_{t \wedge \tau_{\delta}^z})-f(z))| \\
	I_2 &:= |f(z) \mbb{E}^z(\chi(Y_{t \wedge \tau_{\delta}^z})-\chi(z))| \\
	I_3 &:= \left| \mbb{E}^z \left[ \big( f(Y_{t \wedge \tau_{\delta}^z})-f(z) \big)  \big( \chi(Y_{t \wedge \tau_{\delta}^z})-\chi(z) \big) \right] \right|.
	\end{align*}
	We estimate the terms separately. By \eqref{feller-eq21} and \eqref{C2}, it follows from Lemma~\ref{feller-13} that\begin{equation*}
	\mbb{E}^z f(X_{t \wedge \tau_{\delta}^z(X)}) = \mbb{E}^z f(Y_{t \wedge \tau_{\delta}^z}) \fa t \geq 0
	\end{equation*}
	where $\tau_{\delta}^z(X)$ is the exit time of $(X_t)_{t \geq 0}$ from $\overline{B(z,\delta)}$.  As $0 \leq \chi \leq 1$ we thus find \begin{align*}
	I_1
	\leq |\mbb{E}^z(f(X_{t \wedge \tau_{\delta}^z(X)})-f(z))| .
	\end{align*}
	Since $f \in F_1^X$ an application of Dynkin's formula \eqref{gen-eq6} shows that \begin{equation*}
	I_1 \leq
	\|A_e f\|_{\infty} \mbb{E}^z(t \wedge \tau_{\delta}^z(X))
	\leq \|A_e f\|_{\infty} t.
	\end{equation*}
	We turn to $I_2$. As $\chi \in C_c^{\infty}(\mbb{R}^d) \subseteq \mc{D}(L)$ we find from the (classical) Dynkin formula that \begin{align*}
	|I_2| 
	\leq \|f\|_{\infty} |\mbb{E}^z(\chi(Y_{t \wedge \tau_{\delta}^z})-\chi(z))|
	&= \|f\|_{\infty} \left| \mbb{E}^z \left( \int_{(0,t \wedge \tau_{\delta}^z)} L\chi(Y_s) \, ds \right) \right| 
	\leq t \|f\|_{\infty} \sup_{|z-x| \leq 4 \delta} |L\chi(z)|
	\end{align*}
	A straight-forward application of Taylor's formula shows that \begin{equation*}
	|L\chi(z)| \leq 2 \|\chi\|_{C_b^2(\mbb{R}^d)}  \left( |b(z)| + \int_{y \neq 0} \min\{1,|y|^2\} \, \nu(z,dy) \right). 
	\end{equation*}
	Since $0 \leq \varrho(x) \leq 1$ and $\chi$ is chosen such that $\|\chi\|_{C_b^2(\mbb{R}^d)} \leq 10 \delta^{-2}$  we thus get \begin{equation*}
	I_2 \leq 20 \delta^{-2} t \|f\|_{\infty}  \sup_{|z-x| \leq 4 \delta} \left( |b(z)| + \int_{y \neq 0} \min\{1,|y|^{1+\varrho(x)}\} \, \nu(z,dy) \right).
	\end{equation*}
	It remains to estimate $I_3$. Because of the assumptions on the H\"{o}lder regularity of $f$ on $\overline{B(x,4\delta)}$, we have \begin{equation*}
	I_3 \leq 16 \delta^{-2} (\|f\|_{\mc{C}_b^{\varrho(x)}(\overline{B(x,4\delta)})}+\|f\|_{\infty}) \|\chi\|_{C_b^1(\mbb{R}^d)} \mbb{E}^z (|Y_{t \wedge \tau_{\delta}^z}-z|^{1+\varrho(x)} \wedge 1).
	\end{equation*}
	It follows from Lemma~\ref{feller-15} that there exists an absolute constant $c_2>0$ such that \begin{align*}
	I_3
	&\leq c_2 \delta^{-4} t (\|f\|_{\mc{C}_b^{\varrho(x)}(\overline{B(x,4\delta)})}+\|f\|_{\infty}) \sup_{|z-x| \leq 4 \delta} \left( |b(z)|+ \int_{y \neq 0} \min\{|y|^{\varrho(x)+1},1\} \, \nu(z,dy) \right).
	\end{align*}
	Combining the estimates and applying Corollary~\ref{gen-5} we find that $v=\chi \cdot f \in \mc{D}(L_e)$ and \begin{align*}
	\|L_e v\|_{\infty} \leq C' \left(\|A_{\eps} f\|_{\infty} + \|f\|_{\infty} + \|f\|_{\mc{C}_b^{\varrho(x)}(\overline{B(x,4\delta)})} \right) 
	\end{align*}
	where \begin{align*}
	C'
	&:= c_3 \sup_{z \in \mbb{R}^d} \sup_{|\xi| \leq \delta^{-1}} |p(z,\xi)| +c_3 \delta^{-4} \sup_{|z-x| \leq 4 \delta} \left( |b(z)| + \int_{y \neq 0} \min\{|y|^{1+\varrho(x)},1\} \, \nu(z,dy) \right)
	\end{align*}
	for some absolute constant $c_3>0$. Since there exists an absolute constant $c_4>0$ such that \begin{equation*}
	\sup_{z \in \mbb{R}^d} \sup_{|\xi| \leq \delta^{-1}} |p(z,\delta)|
	\leq c_4 \sup_{z \in \mbb{R}^d} \left( |b(z)| + \int_{y \neq 0} \min\{1,|y|^2\} \, \nu(z,dy) \right) \delta^{-2}
	\end{equation*}
	for $\delta \in (0,1)$, cf.\ \cite[Lemma 6.2]{schnurr} and \cite[Theorem 2.31]{ltp}, we obtain, in particular, that 
	\begin{align*}
	\|L_e v\|_{\infty} \leq C'' \left(\|A_{\eps} f\|_{\infty} + \|f\|_{\infty} + \|f\|_{\mc{C}_b^{\varrho(x)}(\overline{B(x,4\delta)})}\right) 
	\end{align*}
	for \begin{align*}
	C''
	&:= c_5 \delta^{-4} \sup_{z \in \mbb{R}^d} \left( |b(z)| +\int_{y \neq 0} \min\{1,|y|^2\} \, \nu(z,dy) \right) 
	+ c_5 \delta^{-4} \sup_{|z-x| \leq 4 \delta}  \int_{|y| \leq 1} \min\{|y|^{1+\varrho(x)},1\} \, \nu(z,dy).
	\end{align*}
	This finishes the proof of \eqref{feller-eq30}. The continuous dependence of the constant $C>0$ in \eqref{feller-eq23} on the parameters $\beta(x) \in [0,1)$, $M(x) \in [0,\infty)$, $K(x) \in [0,\infty)$ follows from the fact that each of the constants in this proof depends continuously on these parameters, see also Lemma~\ref{feller-3}.
\end{proof}

The remaining part of this section is devoted to the proof of Theorem~\ref{feller-12}. We need the following auxiliary result. 

\begin{lem} \label{feller-21}
	Let $(Y_t)_{t \geq 0}$ be a Feller process with infinitesimal generator $(L,\mc{D}(L))$, symbol $p$ and characteristics $(b(x),Q(x),\nu(x,dy))$. For $x \in \mbb{R}^d$ and $r>0$ denote by \begin{equation*}
	\tau_r^x = \inf\{t>0; |Y_t-x|>r\}
	\end{equation*}
	the exit time from the closed ball $\overline{B(x,r)}$. For any fixed $x \in \mbb{R}^d$ and $r>0$ the family of measures \begin{equation*}
	\mu_t(x,B) := \frac{1}{t} \mbb{P}^x(Y_{t \wedge \tau_r^x}-x \in B), \qquad t>0,\, B \in \mc{B}(\mbb{R}^d \backslash \{0\}),
	\end{equation*}
	converges vaguely to $\nu(x,dy)$, i.\,e.\ \begin{equation*}
	\lim_{t \to 0} \frac{1}{t} \mbb{E}^x f(Y_{t \wedge \tau_r^x}-x) = \int_{y \neq 0} f(y) \, \nu(x,dy) \fa f \in C_c(\mbb{R}^d \backslash \{0\}). 
	\end{equation*}
\end{lem}

The main ingredient for the proof of Lemma~\ref{feller-21} is \cite[Theorem 4.2]{ihke} which states that the family of measures $p_t(x,B) :=t^{-1}  \mbb{P}^x(Y_t - x \in B)$, $t>0$, converges vaguely to $\nu(x,dy)$ as $t \to 0$.

\begin{proof}[Proof of Lemma~\ref{feller-21}]
	By the Portmanteau theorem, it suffices to show that \begin{equation}
	\limsup_{t \to 0} \mu_t(x,K) \leq \nu(x,K) \label{proofs-eq21}
	\end{equation}
	for any compact set $K \subseteq \mbb{R}^d \backslash \{0\}$. For given $K \subseteq \mbb{R}^d \backslash \{0\}$ compact there exists by Urysohn's lemma a sequence $(\chi_n)_{n \in \mbb{N}} \subseteq C_c^{\infty}(\mbb{R}^d)$ and a constant $\delta>0$ such that $\spt \chi_n \subseteq B(0,\delta)^c$ for all $n \in \mbb{N}$ and $\I_K = \inf_{n \in \mbb{N}} \chi_n$. It follows from \cite[Theorem 4.2]{ihke} that \begin{equation*}
	\lim_{t \to 0} \frac{\mbb{E}^x \chi_n(Y_t-x)}{t} = \int_{y \neq 0} \chi_n(y) \, \nu(x,dy)
	\end{equation*}
	for all $n \in \mathbb{N}$. On the other hand, an application of Dynkin's formula yields that \begin{align*}
	|\mbb{E}^x \chi_n(Y_{t \wedge \tau_r^x}-x)-\mbb{E}^x \chi_n(Y_t-x)|
	&\leq \|L\chi_n\|_{\infty} \mbb{E}^x(t-\min\{t,\tau_r^x\}) 
	\leq t \|L\chi_n\|_{\infty} \mbb{P}^x(\tau_r^x \leq t).
	\end{align*}
	Since $(Y_t)_{t \geq 0}$ has right-continuous sample paths, we have $\mbb{P}^x(\tau_r^x \leq t) \to 0$ as $t \to 0$, and therefore we obtain that \begin{align*}
	\lim_{t \to 0} \frac{\mbb{E}^x \chi_n(Y_{t \wedge \tau_r^x}-x)}{t} = \int_{y \neq 0} \chi_n(y) \, \nu(x,dy).
	\end{align*}
	Hence, \begin{align*}
	\limsup_{t \to 0} \mu_t(x,K)
	\leq \limsup_{t \to 0} \frac{1}{t} \mbb{E}^x \chi_n(Y_{t \wedge \tau_r^x}-x)
	= \int_{y \neq 0} \chi_n(y) \, \nu(x,dy).
	\end{align*}
	As $\I_K = \inf_{n \in \mbb{N}} \chi_n$, the monotone convergence theorem gives \eqref{proofs-eq21}.
\end{proof}

\begin{proof}[Proof of Theorem~\ref{feller-12}]
	For fixed $x \in \mbb{R}^d$ let $(Y_t^{(x)})_{t \geq 0}$ be the Feller process from Theorem~\ref{feller-12}. Let $\chi_0 \in C_c^{\infty}(\mbb{R}^d)$ be a truncation function such that $\I_{\overline{B(0,\delta)}} \leq \chi_0 \leq \I_{\overline{B(0,2\delta)}}$, and set $\chi^{(x)}(z) := \chi_0(z-x)$, $z \in \mbb{R}^d$. Since $x \in \mbb{R}^d$ is fixed throughout Step 1-3 of this proof, we will often omit the superscript $x$ in our notation, i.e.\ we will write $(Y_t)_{t \geq 0}$ instead of $(Y_t^{(x)})_{t \geq 0}$, $\chi(z)$ instead of $\chi^{(x)}(z)$, etc. \par
	\textbf{Step 1:} Show that $v := \chi \cdot f$ is in the domain $\mc{D}(L_e)$ of the extended generator of $(Y_t)_{t \geq 0}$ and determine $L_e(v)$. \par
	First of all, we note that $(X_t)_{t \geq 0}$, $(Y_t)_{t \geq 0}$ and $f$ satisfy the assumptions of Theorem~\ref{feller-9}. Since we have seen in the proof of Theorem~\ref{feller-9} that $v=\chi \cdot f$ is in the Favard space $F_1^Y$ of order $1$ associated with $(Y_t)_{t \geq 0}$, it follows that $v \in \mc{D}(L_e)$ and $\|L_e(v)\|_{\infty}<\infty$. Applying Corollary~\ref{gen-5} we find that \begin{equation*}
	L_e v(z) = \lim_{t \to 0} \frac{\mbb{E}^z v(Y_{t \wedge \tau_{\delta}^z})-v(z)}{t} 
	\end{equation*}
	(up to a set of potential zero) where \begin{equation*}
	\tau_{\delta}^z := \inf\{t>0; |Y_t-z|>\delta\}.
	\end{equation*}
	On the other hand, the proof of Theorem~\ref{feller-9} shows  that \begin{align*}
	\frac{\mbb{E}^z v(Y_{t \wedge \tau_{\delta}^z})-v(z)}{t}
	= I_1(t)+I_2(t)+I_3(t)
	\end{align*}
	where \begin{align*}
	I_1(t) &:=t^{-1} f(z) (\mbb{E}^z \chi(Y_{t \wedge \tau_{\delta}^z})-\chi(z))  \\
	I_2(t) &:= t^{-1} \chi(z) (\mbb{E}^z f(X_{t \wedge \tau_{\delta}^z(X)})-f(z)) \\
	I_3(t) &:= t^{-1} \mbb{E}^z \big[ (f(Y_{t \wedge \tau_{\delta}^z})-f(z))( \chi(Y_{t \wedge \tau_{\delta}^z})-\chi(z)) \big];
	\end{align*}
	here $\tau_{\delta}^z(X)$ denotes the exit time of $(X_t)_{t \geq 0}$ from $\overline{B(z,\delta)}$.  Since $\chi \in C_c^{\infty}(\mbb{R}^d)$ is in the domain of the (strong) infinitesimal generator $L$ of $(Y_t)_{t \geq 0}$ and $f$ is the Favard space $F_1^X$ associated with $(X_t)_{t \geq 0}$, another application of Corollary~\ref{gen-5} shows that \begin{equation*}
	\lim_{t \to 0} I_1(t) = f(z) L\chi(z) \quad \text{and} \quad \lim_{t \to 0} I_2(t) = \chi(z) A_e f(z)
	\end{equation*}
	for all $z \in \mbb{R}^d$. We claim that \begin{equation}
	\lim_{t \to 0} I_3(t) = \Gamma(f,\chi)(z) := \int_{y \neq 0} (f(z+y)-f(z))(\chi(z+y)-\chi(z)) \, \nu(z,dy) \label{proofs-eq27}
	\end{equation}
	for all $z \in \mbb{R}^d$ where $\nu(z,dy)= \nu^{(x)}(z,dy)$ denotes the family of L\'evy measures associated with $(Y_t)_{t \geq 0} = (Y_t^{(x)})_{t \geq 0}$, cf.\ \eqref{main-eq15}. Once we have shown this, it follows that \begin{equation}
	L_e v = f L \chi + \chi A_e f + \Gamma(f,\chi). \label{proofs-eq29}
	\end{equation}
	To prove \eqref{proofs-eq27} we fix a truncation function $\varphi \in C_c^{\infty}(\mbb{R}^d)$ such that $\I_{B(0,1)} \leq \varphi \leq \I_{B(0,2)}$ and set $\varphi_{\eps}(y) := \varphi(\eps^{-1} y)$ for $\eps>0$, $y \in \mbb{R}^d$. Since $y \mapsto (1-\varphi_{\eps}(y))$ is zero in a neighbourhood of $0$, we find from Lemma~\ref{feller-21} that \begin{align*}
	&\frac{\mbb{E}^z\big[ (1-\varphi_{\eps}(Y_{t \wedge \tau_{\delta}^z}-z))  (f(Y_{t \wedge \tau_{\delta}^z})-f(z))( \chi(Y_{t \wedge \tau_{\delta}^z})-\chi(z)) \big]}{t} \\
	&\xrightarrow[]{t \to 0} \int_{y \neq 0} (1-\varphi_{\eps}(y)) (f(y+z)-f(z)) (\chi(z+y)-\chi(z)) \, \nu(z,dy).
	\end{align*}
	If $z \in \mbb{R}^d \backslash B(x,3\delta)$ then $\chi=0$ on $B(z,\delta)$, and therefore the integrand on the right hand side equals zero for $|y|< \delta$. Applying the dominated convergence theorem we thus find that the right-hand side converges to $\Gamma(f,\chi)(z)$, defined in \eqref{proofs-eq27}, as $\eps \to 0$. For $z \in B(x,3\delta)$ we note that $\chi \in C_b^1(\mbb{R}^d)$ and $f \in \mc{C}_b^{\varrho(\cdot)}(\mbb{R}^d)$ for $\varrho$ satisfying \eqref{main-eq21}; it now follows from \eqref{S1} and the dominated convergence theorem that the right-hand side converges to $\Gamma(f,\chi)(z)$ as $\eps \to 0$. To prove \eqref{proofs-eq27} it remains to show that \begin{equation*}
	J(\eps,t,z) := \left|\mbb{E}^z\big[ \varphi_{\eps}(Y_{t \wedge \tau_{\delta}^z}-z)  (f(Y_{t \wedge \tau_{\delta}^z})-f(z))( \chi(Y_{t \wedge \tau_{\delta}^z})-\chi(z)) \big] \right|
	\end{equation*} 
	satisfies \begin{equation*}
	\limsup_{\eps \to 0} \limsup_{t \to 0} \frac{1}{t} J(\eps,t,z)=0 \fa z \in \mbb{R}^d.
	\end{equation*}
	By \eqref{main-eq21} and \eqref{S1}, there exists some constant $\gamma>0$ such that \begin{equation}
	1+\min\{\varrho(z),1\} \geq \alpha(z)+2 \gamma \fa z \in \overline{B(x,3\delta)}. \label{proofs-eq31} 
	\end{equation} 
	Indeed: On $\{\varrho \geq 1\}$ this inequality holds since $\alpha$ is bounded away from $2$, cf.\ \eqref{S1}, and on $\{\varrho<1\}$ this is a direct consequence of \eqref{main-eq21}. Now fix some $z \in \overline{B(x,3\delta)}$. As $\spt \varphi_{\eps} \subseteq \overline{B(0,2\eps)}$ it follows from $f \in \mc{C}_b^{\varrho(\cdot)}(\mbb{R}^d)$ and $\chi \in C_b^1(\mbb{R}^d)$ that
	\begin{align*}
	J(\eps,t,z)
	&\leq c_1 \eps^{\gamma} \|f\|_{\mc{C}_b^{\varrho(\cdot)}(\mbb{R}^d)} \|\chi\|_{C_b^1(\mbb{R}^d)} \mbb{E}^z \min\{|Y_{t \wedge \tau_{\delta}^z}-z|^{\alpha(z)+\gamma},1\}
	\end{align*}
	with $\gamma$ from \eqref{proofs-eq31} and some constant $c_1>0$ (not depending on $f$, $x$, $z$). An application of Lemma~\ref{feller-15} now yields \begin{align*}
	J(\eps,t,z) \leq c_2 \eps^{\gamma} t \sup_{|z-x| \leq 4 \delta} \left( |b(z)| + \int_{y \neq 0} \min\{|y|^{\alpha(z)+\gamma},1\} \nu(z,dy) \right)
	\end{align*}
	which is finite because of \eqref{S1} and \eqref{S4}. Hence, \begin{equation*}
	\limsup_{t \to 0} \limsup_{\eps \to 0} \frac{1}{t} J(\eps,t,z)=0 \fa |z-x| \leq 3 \delta.
	\end{equation*}
	If $z \in \mbb{R}^d \backslash B(x,3\delta)$ then it follows from $\chi|_{B(z,\delta)}=0$ and $\spt \varphi \subseteq B(0,2\eps)$ that \begin{equation*}
	J(\eps,t,z) \leq 4 \eps \|f\|_{\infty} \|\chi\|_{C_b^1(\mbb{R}^d)} \mbb{P}^z(\tau_{\delta}^z \leq t).
	\end{equation*}
	Applying the maximal inequality \eqref{max} for Feller processes we conclude that \begin{equation*}
	\limsup_{\eps \to 0} \limsup_{t \to 0} t^{-1} J(\eps,t,z) = 0 \fa z \in \mbb{R}^d \backslash B(x,3\delta).
	\end{equation*}
	\textbf{Step 2:} If $\varrho: \mbb{R}^d \to [0,2]$ is a uniformly continuous function satisfying \eqref{main-eq21} and $\varrho_0 := \inf_z \varrho(z)>0$, then \begin{equation*}
	f \in F_1^X \cap \mc{C}_b^{\varrho(\cdot)}(\mbb{R}^d), A_e f = g \in  \mc{C}_b^{\lambda}(\mbb{R}^d) \implies \forall \eps>0: \, \, L_e(f \chi) \in \mc{C}_b^{(\varrho_0 \wedge \lambda \wedge \theta \wedge \sigma)-\eps}(\mbb{R}^d)
	\end{equation*}
	for any $\lambda \in [0,\Lambda]$ where $\chi=\chi^{(x)}$ is the truncation function chosen at the beginning of the proof; see \eqref{S2}, \eqref{S3} and \eqref{main-eq21} for  the definition of $\theta$, $\Lambda$ and $\sigma$. \par
	\emph{Indeed:} We know from Step 1 that \begin{equation*}
	L_e (f \chi) = f L \chi + \chi A_e f + \Gamma(f,\chi) =: I_1+I_2+I_3.
	\end{equation*}
	As $\theta \leq 1$ we have $\varrho_0 \wedge \lambda \wedge \theta \wedge \sigma \leq 1$, and therefore it suffices to estimate \begin{equation*}
	\sup_{z \in \mbb{R}^d} |I_k(z)| + \sup_{z,h \in \mbb{R}^d} |I_k(z+h)-I_k(z)|
	\end{equation*}
	for $k=1,2,3$. \par
	\textbf{Estimate of $I_1=f L \chi$:} First we estimate the H\"{o}lder norm of $L\chi$. As $\chi \in C_c^{\infty}(\mbb{R}^d)$ a straight-forward application of Taylor's formula shows that \begin{equation*}
	\|L \chi\|_{\infty} \leq 2 \|\chi\|_{C_b^2(\mbb{R}^d)} \sup_{z \in \mbb{R}^d} \left( |b(z)| + \int_{y \neq 0} \min\{1,|y|^2\} \, \nu(z,dy) \right).
	\end{equation*}
	If we set $D_y \chi(z) := \chi(z+y)-\chi(z)-\chi'(z) y \I_{(0,1)}(|y|)$, then \begin{align*}
	|L\chi(z)-L \chi(z+h)| &\leq |b(z)| \, |\nabla \chi(z+h)-\nabla \chi(z)| + |b(z+h)-b(z)| \, |\nabla \chi(z+h)| \\ &+ \int_{y \neq 0} |D_y \chi(z+h)-D_y \chi(z)| \, \nu(z,dy) + \left| \int_{y \neq 0} D_y \chi(z+h) \, (\nu(z+h,dy)-\nu(z,dy)) \right|.
	\end{align*}
	for all $z,h \in \mbb{R}^d$. To estimate the first two terms on the right-hand side we use the H\"{o}lder continuity of $b$, cf.\ \eqref{S2}, and the fact that $\chi \in C_b^2(\mbb{R}^d)$. For the third term we use \begin{equation*}
	|D_y \chi(z+h)-D_y \chi(z)| \leq \|\chi\|_{C_b^3(\mbb{R}^d)} |h| \min\{|y|^2,1\},
	\end{equation*}
	cf.\ \cite[Theorem 5.1]{bass08} for details, and noting that \begin{equation*}
		|D_y \chi(z+h)| \leq 2 \|\chi\|_{C_b^2(\mbb{R}^d)} \min\{1,|y|^2\}
	\end{equation*}
	we can estimate the fourth term for small $h$ by applying \eqref{S2}. Hence,	\begin{equation*}
	|L\chi(z)-L \chi(z+h)| \leq |h| \|\chi\|_{C_b^3(\mbb{R}^d)} \left(|b(z)|+ \int_{y \neq 0} \min\{1,|y|^2\} \, \nu(z,dy)\right) + 2C |h|^{\theta} \|\chi\|_{C_b^2(\mbb{R}^d)}
	\end{equation*}
	for small $h>0$. Hence, \begin{equation*}
	\|L \chi\|_{\mc{C}_b^{\theta}(\mbb{R}^d)} \leq c_1 \|\chi\|_{C_b^3(\mbb{R}^d)} \sup_{z \in \mbb{R}^d} \left(1+ |b(z)| + \int_{y \neq 0} \min\{1,|y|^2\} \, \nu(z,dy) \right)
	\end{equation*}
	for some absolute constant $c_1>0$. Since $f \in \mc{C}_b^{\varrho(\cdot)}(\mbb{R}^d) \subseteq \mc{C}_b^{\varrho_0}(\mbb{R}^d)$, this entails that \begin{equation*}
	\|f L \chi\|_{\mc{C}_b^{\theta \wedge \varrho_0}(\mbb{R}^d)} \leq  c_1' \|f\|_{\mc{C}_b^{\varrho_0}(\mbb{R}^d)} \|\chi\|_{C_b^3(\mbb{R}^d)} \sup_{z \in \mbb{R}^d} \left(1+ |b(z)| + \int_{y \neq 0} \min\{1,|y|^2\} \, \nu(z,dy) \right).
	\end{equation*}
	\textbf{Estimate of $I_2=\chi A_e f$:} By assumption, $A_e f = g \in \mc{C}_b^{\lambda}(\mbb{R}^d)$ and $\chi \in C_c^{\infty}(\mbb{R}^d)$. Thus, \begin{equation*}
	\|\chi A_e f\|_{\mc{C}_b^{\lambda}} \leq 2 \|\chi\|_{C_b^{\lambda}} \|A_e f\|_{\mc{C}_b^{\lambda}} < \infty.
	\end{equation*}
	\textbf{Estimate of $I_3 = \Gamma(f,\chi)$:}  As $f \in C_b^{\varrho(\cdot)}(\mbb{R}^d)$ and $\chi \in C_b^1(\mbb{R}^d)$, it follows from the definition of $\Gamma(f,\chi)$, cf.\ \eqref{proofs-eq27}, that \begin{equation*}
	|\Gamma(f,\chi)(z)| 
	\leq 4\|f\|_{\mc{C}_b^{\varrho(\cdot)}(\mbb{R}^d)} \|\chi\|_{C_b^1(\mbb{R}^d)} \int_{y \neq 0} \min\{|y|^{1+\min\{1,\varrho(z)\}} \wedge 1,1\} \, \nu(z,dy)<\infty
	\end{equation*}
	for all $|z-x| \leq 3\delta$. If $z \in \mbb{R}^d \backslash B(x,3\delta)$, then $\Delta_y \chi(z)=0$ for all $|y| \leq \delta$, and so \begin{equation*}
		|\Gamma(f,\chi)(z)| \leq 4 \|f\|_{\infty} \int_{|y|>\delta/2} \, \nu(z,dy)
	\end{equation*}
	for all $z \in \mbb{R}^d \backslash B(x,3\delta)$. Combining both estimates and using \eqref{main-eq21}, \eqref{S1} and \eqref{S4}, we get \begin{equation*}
	\|\Gamma(f,\chi)\|_{\infty} \leq c_2 \|f\|_{\mc{C}_b^{\varrho(\cdot)}(\mbb{R}^d)}
	\end{equation*}
	for some constant $c_2>0$ not depending on $x$, $z$ and $f$. To study the regularity of $\Gamma(f,\chi)$ we consider separately the cases $\|\varrho\|_{\infty} \leq 1$ and $\|\varrho\|_{\infty}>1$. We start with the case $\|\varrho\|_{\infty} \leq 1$, see the end of this step for the other case. To estimate $\Delta_h \Gamma(f,\chi)$ we note that \begin{equation}
	|\Delta_h \Gamma(f,\chi)(z)| = |\Gamma(f,\chi)(z+h)-\Gamma(f,\chi)(z)|
	\leq J_1+J_2+J_3 \label{proofs-eq39}
	\end{equation}
	where \begin{align*}
	J_1(z) &:= \int_{y \neq 0} |\Delta_y f(z+h) - \Delta_y f(z)| \, |\Delta_y \chi(z+h)| \, \nu(z,dy) \\
	J_2(z) &:= \int_{y \neq 0} |\Delta_y f(z)| \, |\Delta_y \chi(z+h)-\Delta_y \chi(z)| \, \nu(z,dy) \\
	J_3(z) &:= \left| \int_{y \neq 0} \Delta_y f(z+h) \Delta_y \chi(z+h) (\nu(z,dy)-\nu(z+h,dy)) \right|.
	\end{align*}
	We estimate the terms separately and start with $J_1$. Fix $\eps \in (0,\min\{\varrho_0,\sigma\}/2)$, cf.\ \eqref{main-eq21} for the definition of $\sigma$. Since $\varrho$ is uniformly continuous there exists $r \in (0,1)$ such that \begin{equation*}
	|\varrho(z)-\varrho(z+h)| \leq \eps \fa z \in \mbb{R}^d, \,|h| \leq r.
	\end{equation*}
	For $|h| \leq r$  and $|y| \leq r$ it then follows from $f \in \mc{C}_b^{\varrho(\cdot)}(\mbb{R}^d)$ that \begin{align*}
	|\Delta_y f(z+h)-\Delta_y f(z)| 
	&\leq 2 \|f\|_{\mc{C}_b^{\varrho(\cdot)}(\mbb{R}^d)} \min\{|y|^{\varrho(z) \wedge \varrho(z+h)},|h|^{\varrho(y+z) \wedge \varrho(z)}\} \\
	&\leq 2 \|f\|_{\mc{C}_b^{\varrho(\cdot)}(\mbb{R}^d)} \min\{|y|^{\varrho(z) -\eps},|h|^{\varrho(z)-\eps}\}.
	\end{align*}
	(Here we use $\|\varrho\|_{\infty} \leq 1$; otherwise we would need to replace $\varrho(z)$ by $\varrho(z) \wedge 1$ etc.) On the other hand, we also have \begin{equation}
	|\Delta_y f(z+h)-\Delta_y f(z)| \leq 2 \|f\|_{\mc{C}_b^{\varrho(\cdot)}(\mbb{R}^d)} |h|^{\varrho_0}  \label{proofs-eq41}
	\end{equation}
	for all $y \in \mbb{R}^d$. Combining both estimates yields  \begin{align*}
	J_1(z)
	\leq 2 \|f\|_{\mc{C}_b^{\varrho(\cdot)}(\mbb{R}^d)} \|\chi\|_{C_b^1(\mbb{R}^d)} \left( \int_{|y| \leq r} \min\{|y|^{\varrho(z)- \eps},|h|^{\varrho(z)-\eps}\} |y| \, \nu(z,dy) +|h|^{\varrho_0} \int_{|y|>r} \, \nu(z,dy) \right)
	\end{align*}
	for $|h| \leq r$. It is now not difficult to see from \eqref{S1} and \eqref{S4} that there exists a constant $c_3>0$ (not depending on $x$, $z$, $f$) such that \begin{equation*}
	J_1(z) \leq c_3  \|f\|_{\mc{C}_b^{\varrho(\cdot)}(\mbb{R}^d)} (|h|^{\varrho_0} + |h|^{\varrho(z)+1-\alpha(z)-\eps}) \fa |h| \leq r, z \in B(x,3\delta).
	\end{equation*}
	By the very definition of $\sigma$, cf.\ \eqref{main-eq21}, this implies that \begin{equation*}
	\sup_{z \in B(x,3\delta)} J_1(z) \leq c_3  \|f\|_{\mc{C}_b^{\varrho(\cdot)}(\mbb{R}^d)} |h|^{\min\{\varrho_0,\sigma\}-\eps} \fa |h| \leq r.
	\end{equation*}
	If $z \in \mbb{R}^d \backslash B(x,3\delta)$ then $\Delta_y \chi(z+h)=0$ for $|h| \leq \delta/2$ and $|y| \leq \delta/2$. Using \eqref{proofs-eq41} we get \begin{align*}
	J_1(z) \leq 2 |h|^{\varrho_0} \|f\|_{\mc{C}_b^{\varrho(\cdot)}(\mbb{R}^d)}  \int_{|y| \geq \delta/2} \, \nu(z,dy) \fa |h| \leq \delta/2.
	\end{align*}
	Invoking once more \eqref{S1} and \eqref{S4} we obtain that \begin{equation*}
	\sup_{z \in \mbb{R}^d \backslash B(x,3\delta)} J_1(z) \leq c_4 |h|^{\varrho_0}\|f\|_{\mc{C}_b^{\varrho(\cdot)}(\mbb{R}^d)}, \qquad |h| \leq \delta/2,
	\end{equation*}
	for some constant $c_4$ not depending on $x$, $z$ and $f$. In summary, we have shown that \begin{equation*} 
	\sup_{z \in \mbb{R}^d} J_1(z) \leq c_5 |h|^{\min\{\varrho_0, \sigma\}-\eps}\|f\|_{\mc{C}_b^{\varrho(\cdot)}(\mbb{R}^d)}.
	\end{equation*}
	To estimate $J_2$ we consider again separately the cases $z \in B(x,3\delta)$ and $z \in \mbb{R}^d \backslash B(x,3\delta)$. If $z \in \mbb{R}^d \backslash B(x,3\delta)$ then $\Delta_y \chi(z+h) = 0 = \Delta_y \chi(z)$ for all $|y| \leq \delta/2$ and $|h| \leq \delta/2$. Since we also have \begin{equation}
	|\Delta_y \chi(z+h)-\Delta_y \chi(z)| \leq 2 \|\chi\|_{C_b^2(\mbb{R}^d)} \min\{|y|,|h|\}  \label{proofs-eq43}
	\end{equation}
	we find that \begin{equation*}
	J_2(z) \leq 4 \|f\|_{\infty} \|\chi\|_{C_b^2(\mbb{R}^d)} |h| \int_{|y| \geq \delta/2} \, \nu(z,dy)
	\end{equation*}
	for $|h| \leq \delta/2$. Because of \eqref{S1} and \eqref{S4} this gives the existence of a constant $c_6>0$ (not depending on $f$, $x$ and $z$) such that \begin{equation*}
	\sup_{z \in \mbb{R}^d \backslash B(x,3\delta)} J_2(z) \leq c_6 \|f\|_{\infty} |h|.
	\end{equation*}
	For $z \in B(x,3\delta)$ we combine \begin{equation*}
	|\Delta_y f(z)| \leq 2\|f\|_{\mc{C}_b^{\varrho(\cdot)}(\mbb{R}^d)} \min\{|y|^{\varrho(z)},1\}
	\end{equation*}
	with \eqref{proofs-eq43} to get \begin{equation*}
	J_2(z) 
	\leq 4 \|f\|_{\mc{C}_b^{\varrho(\cdot)}(\mbb{R}^d)} \|\chi\|_{C_b^2(\mbb{R}^d)} \int_{y \neq 0} \min\{|y|^{\varrho(z)},1\} \min\{|y|,|h|\} \, \nu(z,dy)
	\end{equation*}
	which implies, by \eqref{S1}, \eqref{S4} and \eqref{main-eq21}, that \begin{equation*}
	\sup_{z \in B(x,3\delta)} J_2 \leq c_7 \|f\|_{\mc{C}_b^{\varrho(\cdot)}(\mbb{R}^d)} |h|^{\sigma \wedge 1}.
	\end{equation*}
	We conclude that \begin{equation*}
	\sup_{z \in \mbb{R}^d} J_2(z) \leq c_8  |h|^{\sigma \wedge 1} \|f\|_{\mc{C}_b^{\varrho(\cdot)}(\mbb{R}^d)}.
	\end{equation*}
	It remains to estimate $J_3$. By the uniform continuity of $\varrho$ there exists $r \in(0,1)$ such that $|\Delta_h \varrho(z)| \leq \sigma/2$ for all $|h| \leq r$. Since $f \in \mc{C}_b^{\varrho(\cdot)}(\mbb{R}^d)$ we have \begin{align*}
	|\Delta_y f(z+h) \Delta_y \chi(z+h)| 
	&\leq 4\|f\|_{\mc{C}_b^{\varrho(\cdot)}(\mbb{R}^d)} \|\chi\|_{C_b^1(\mbb{R}^d)} \min\{|y|^{\varrho(z+h)+1},1\} 
	\end{align*}
	and thus, by \eqref{main-eq21} and our choice of $r \in (0,1)$, \begin{align*}
	|\Delta_y f(z+h) \Delta_y \chi(z+h)| 
	&\leq 4\|f\|_{\mc{C}_b^{\varrho(\cdot)}(\mbb{R}^d)} \|\chi\|_{C_b^1(\mbb{R}^d)} \min\{|y|^{\varrho(z)+1-\sigma/2},1\}  \\
	&\leq 4\|f\|_{\mc{C}_b^{\varrho(\cdot)}(\mbb{R}^d)} \|\chi\|_{C_b^1(\mbb{R}^d)} \min\{|y|^{\sigma/2+\alpha(z)},1\}
	\end{align*}
	for all $|z-x| \leq 3\delta$ and $|h| \leq r$. On the other hand, if $z \in \mbb{R}^d \backslash B(x,3\delta)$, then $\chi=0$ on $\overline{B(z,\delta)}$ and so \begin{equation*}
	|\Delta_y f(z+h) \Delta_y \chi(z+h)| = 0 \fa |h| \leq \delta/2, \,|y| \leq \delta/2.
	\end{equation*}
	Consequently, there exists a constant $c_9 = c_9(\delta,r)>0$ such that \begin{equation*}
	|\Delta_y f(z+h) \Delta_y \chi(z+h)|  \leq c_9 \|f\|_{\mc{C}_b^{\varrho(\cdot)}(\mbb{R}^d)} \|\chi\|_{C_b^1(\mbb{R}^d)} \min\{|y|^{\sigma+\alpha(z)},1\}
	\end{equation*}
	for all $z \in \mbb{R}^d$, $y \in \mbb{R}^d$ and $|h| \leq \min\{r,\delta\}/2$. Applying \eqref{S2} we thus find \begin{equation*}
	\sup_{z \in \mbb{R}^d} J_3(z) \leq c_{10} |h|^{\theta} \|f\|_{\mc{C}_b^{\varrho(\cdot)}(\mbb{R}^d)} .
	\end{equation*}
	Combining the above estimates we conclude that \begin{equation*}
	\|\Gamma(f,\chi)\|_{\mc{C}_b^{\varrho_0 \wedge \theta \wedge \sigma - \eps}(\mbb{R}^d)} \leq c_{11} \|f\|_{\mc{C}_b^{\varrho(\cdot)}(\mbb{R}^d)}
	\end{equation*}
	provided that $\|\varrho\|_{\infty} \leq 1$. In the other case, i.\,e.\ if $\varrho$ takes values strictly larger than one, then we need to consider second differences $\Delta_h^2 \Gamma(f,\chi)(z)$ in order to capture the full information on the regularity of $f$. The calculations are very similar to the above ones but quite lengthy (it is necessary to consider nine terms separately) and therefore we do not present the details here. \par
	\textbf{Conclusion of Step 2:} For any small $\eps>0$ there exists a finite constant $K_{1,\eps}>0$ such that \begin{equation}
	\|L_e (f \chi)\|_{\mc{C}_b^{\min\{\varrho_0,\lambda,\theta,\sigma\}-\eps}(\mbb{R}^d)} \leq K_{1,\eps} \left( \|A_e f\|_{\mc{C}_b^{\lambda}(\mbb{R}^d)} + \|f\|_{\mc{C}_b^{\varrho(\cdot)}(\mbb{R}^d)} \right). \label{proofs-eq59}
	\end{equation}
	The constant $K_{1,\eps}$ does not depend on $x$, $z$ and $f$.  \par
	\textbf{Step 3:}  If $u \in \mc{D}(L_e)$ is such that $u \in \mc{C}_b^{\lambda}(\mbb{R}^d)$ and $L_e u \in \mc{C}_b^{\lambda}(\mbb{R}^d)$ for some $\lambda \leq \Lambda$ (cf.\ \eqref{S3}), then \begin{equation*}
	\|u\|_{\mc{C}_b^{\kappa(x)+\lambda}(\mbb{R}^d)} \leq K_2 (\|u\|_{\mc{C}_b^{\lambda}(\mbb{R}^d)} + \|L_e u\|_{\mc{C}_b^{\lambda}(\mbb{R}^d)})
	\end{equation*}
	for some constant $K_2>0$ which does not depend on $x$, $z$ and $f$. (Recall that $L_{e}=L_e^{(x)}$ is the extended generator of the Feller process $(Y_t)_{t \geq 0} = (Y_t^{(x)})_{t \geq 0}$; this explains the $x$-dependence of the regularity on the left-hand side of the inequality.) \par
	\emph{Indeed:} The $\mu$-potential operators $(R_{\mu})_{\mu>0}$ associated with $(Y_t)_{t \geq 0} = (Y_t^{(x)})_{t \geq 0}$ satisfies \begin{equation}
	\|R_{\mu} v\|_{\mc{C}_b^{\kappa(x)+\lambda}(\mbb{R}^d)} \leq K \|v\|_{\mc{C}_b^{\lambda}(\mbb{R}^d)}, \qquad v \in \mc{C}_b^{\lambda}(\mbb{R}^d), \,\lambda \leq \Lambda \label{proofs-eq61}
	\end{equation}
	for $\mu$ sufficiently large and some constant $K=K(\mu)>0$. This is a direct consequence of \eqref{S3} and Lemma~\ref{feller-3}. Now if $u \in \mc{D}(L_e)$ is such that $u \in \mc{C}_b^{\lambda}(\mbb{R}^d)$ and $L_e u \in \mc{C}_b^{\lambda}(\mbb{R}^d)$, then we have $u = R_{\mu}v$ for $v:=\mu u- L_e u \in \mc{C}_b^{\lambda}(\mbb{R}^d)$.  Applying \eqref{proofs-eq61} proves the desired estimate.  \par
	\textbf{Conclusion of the proof:} Let $f \in \mc{C}_b^{\varrho(\cdot)}(\mbb{R}^d) \cap F_1^X$ for $\varrho$ satisfying \eqref{main-eq21} be such that $A_e f \in \mc{C}_b^{\lambda}(\mbb{R}^d)$ for some $\lambda \leq \Lambda$. Without loss of generality we may assume that $\varrho_0 := \inf_x \varrho(x)>0$. \emph{Indeed}: It follows from Corollary~\ref{feller-11} that $f \in \mc{C}_b^{\kappa(\cdot)-\eps}(\mbb{R}^d)$ for $\eps := \kappa_0/2 := \inf_x \kappa(x)/2>0$, and therefore we may replace $\varrho$ by $\tilde{\varrho}(z) := \max\{\varrho(z),\kappa(z)-\eps\}$ which is clearly bounded away from zero and satisfies the assumptions of Theorem~\ref{feller-12}. \par
	For fixed $x \in \mbb{R}^d$ denote by $\chi= \chi^{(x)}$ the truncation function chosen at the beginning of the proof, and fix $\eps \in (0,\min\{\varrho_0,\kappa_0\}/2)$. It follows from Step 2 and Step 3 that there exists a constant $c_1>0$ such that \begin{equation*}
	\|f \chi^{(x)}\|_{\mc{C}_b^{\kappa(x)+\min\{\varrho_0,\sigma,\theta,\lambda\}-\eps}(\mbb{R}^d)} \leq c_1 \left(\|A_e f\|_{\mc{C}_b^{\lambda}(\mbb{R}^d)} + \|f\|_{\mc{C}_b^{\varrho(\cdot)}(\mbb{R}^d)} \right)
	\end{equation*}
	for all $x \in \mbb{R}^d$. As $\chi^{(x)}=1$ on $B(x,\delta)$ we obtain that \begin{equation*}
	\|f\|_{\mc{C}_b^{\kappa(\cdot)+\min\{\varrho_0,\sigma,\theta,\lambda\}-\eps}(\mbb{R}^d)} \leq c_1'  \left(\|A_e f\|_{\mc{C}_b^{\lambda}(\mbb{R}^d)} + \|f\|_{\mc{C}_b^{\varrho(\cdot)}(\mbb{R}^d)}\right).
	\end{equation*}
	Since, by assumption, $f \in \mc{C}_b^{\varrho(\cdot)}(\mbb{R}^d)$, this implies $f \in \mc{C}_b^{\varrho^1(\cdot)}(\mbb{R}^d)$ for \begin{equation*}
		\varrho^1(x) := \max\{\varrho(x),\kappa(x)-\eps+\min\{\varrho_0,\sigma,\theta,\lambda\}\}, \qquad x \in \mbb{R}^d,
	\end{equation*}
	and we have \begin{equation*}
	\|f\|_{\mc{C}_b^{\varrho^{1}(\cdot)}(\mbb{R}^d)} 
	\leq (c_1'+1)  \left(\|A_e f\|_{\mc{C}_b^{\lambda}(\mbb{R}^d)} + \|f\|_{\mc{C}_b^{\varrho(\cdot)}(\mbb{R}^d)}\right).
	\end{equation*}
	As $\varrho^{1}$ satisfies \eqref{main-eq21} (with $\varrho$ replaced by $\varrho^{1}$) we may apply Step 2 with $\varrho$ replaced by $\varrho^{1}$ to obtain that 
	\begin{equation*}
	\|f \chi^{(x)}\|_{\mc{C}_b^{\kappa(x)+\min\{\varrho_0^1,\sigma,\theta,\lambda\}-\eps}(\mbb{R}^d)} 
	\leq c_2 \left(\|A_e f\|_{\mc{C}_b^{\lambda}(\mbb{R}^d)} + \|f\|_{\mc{C}_b^{\varrho(\cdot)}(\mbb{R}^d)}\right)
	\end{equation*}
	where $\varrho_0^1 := \inf_{x \in \mbb{R}^d} \varrho^1(x)$. Repeating the argumentation from above, i.\,e.\ using that $\chi^{(x)}=1$ on $B(x,\delta)$, we obtain $f \in \mc{C}_b^{\varrho^2(\cdot)}(\mbb{R}^d)$ for $\varrho^2(x) := \max\{\varrho(x),\kappa(x)-\eps+\min\{\varrho_0^1,\sigma,\theta,\lambda\}\}$ and \begin{equation*}
	\|f\|_{\mc{C}_b^{\varrho^2(\cdot)}(\mbb{R}^d)} \leq c_2' \left(\|A_e f\|_{\mc{C}_b^{\lambda}(\mbb{R}^d)} + \|f\|_{\mc{C}_b^{\varrho(\cdot)}(\mbb{R}^d)}\right).
	\end{equation*}
	We proceed by iteration, i.\,e.\ we define $\varrho^n(x) := \max\{\varrho(x),\kappa(x)-\eps+\min\{\varrho_0^{n-1},\sigma,\theta,\lambda\}\}$, $n \geq 2$, where $\varrho_0^{n-1} := \inf_x \varrho^{n-1}(x)$. By Step 2 and 3, we then have \begin{equation}
	\|f\|_{\mc{C}_b^{\varrho^n(\cdot)}(\mbb{R}^d)} \leq c_n \left(\|A_e f\|_{\mc{C}_b^{\lambda}(\mbb{R}^d)} + \|f\|_{\mc{C}_b^{\varrho(\cdot)}(\mbb{R}^d)}\right) \label{proofs-eq63}
	\end{equation}
	for some constant $c_n>0$. Since $\kappa_0 = \inf_x \kappa(x)>0$ and $\eps<\kappa_0/2$ it is not difficult to see that we can choose $n \in \mathbb{N}$ sufficiently large such that $\varrho_0^n \geq \min\{\sigma,\theta,\lambda\}$ and so 
	\begin{equation*}
	\varrho^{n+1}(x) \geq \kappa(x)-\eps + \min\{\sigma,\theta,\lambda\}.
	\end{equation*}
	Using \eqref{proofs-eq63} (with $n$ replaced by $n+1$) we conclude that \begin{equation*}
	\|f\|_{\mc{C}_b^{\kappa(\cdot)+\min\{\sigma,\theta,\lambda\}-\eps}(\mbb{R}^d)} \leq c_{n+1} \left(\|A_e f\|_{\mc{C}_b^{\lambda}(\mbb{R}^d)} + \|f\|_{\mc{C}_b^{\varrho(\cdot)}(\mbb{R}^d)}\right) 
	\end{equation*}
	which proves the assertion. 	
\end{proof}

\section{Proof of Schauder estimates for isotropic stable-like processes} \label{iso}

In this section we present the proof of the Schauder estimates for isotropic stable-like processes which we stated in Theorem~\ref{ex-21} and Corollary~\ref{ex-23}. Throughout this section, $(X_t)_{t \geq 0}$ is an isotropic stable-like process, i.\,e.\ a Feller process with symbol of the form $q(x,\xi) = |\xi|^{\alpha(x)}$, $x,\xi \in \mbb{R}^d$, for a mapping $\alpha: \mbb{R}^d \to (0,2]$. We remind the reader that such a Feller process exists if $\alpha$ is H\"{o}lder continuous and bounded away from zero. \par
We will apply the results from Section~\ref{main} to establish the Schauder estimates. To this end, we need regularity estimates for the semigroup $(P_t)_{t \geq 0}$ associated with $(X_t)_{t \geq 0}$. The results, which we obtain, are of independent interest and we present them in Subsection~\ref{iso-reg} below. Once we have established another auxiliary statement in Subsection~\ref{iso-aux}, we will present the proof of Theorem~\ref{ex-21} and Corollary~\ref{ex-23} in Subsection~\ref{iso-proofs}.

\subsection{Regularity estimates for the semigroup of stable-like processes} \label{iso-reg}

Let $(P_t)_{t \geq 0}$ be the semigroup of an isotropic stable-like process $(X_t)_{t \geq 0}$ with symbol $q(x,\xi) = |\xi|^{\alpha(x)}$. In this subsection we study the regularity of the mapping $x \mapsto P_t u(x)$. We will see that there are several parameters which influence the regularity of $P_t u$: \begin{itemize}
	\item the regularity of $x \mapsto u(x)$,
	\item the regularity of $x \mapsto \alpha(x)$,
	\item $\alpha_L := \inf_{x \in \mbb{R}^d} \alpha(x)$;
\end{itemize}
the larger these quantities are, the higher the regularity of $P_t u$. The regularity estimates, which we present, rely on the parametrix construction of (the transition density of) $(X_t)_{t \geq 0}$ in \cite{matters}. Let us mention that there are other approaches to obtain regularity estimates for the semigroup. Using coupling methods, Luo \& Wang \cite{wang14} showed that for any $\kappa \in (0,\alpha_L)$ there exists $c>0$ such that \begin{equation*}
	\|P_t u\|_{\mc{C}_b^{\kappa \wedge 1}(\mbb{R}^d)} \leq c \|u\|_{\infty} t^{-(\kappa \wedge 1)/\alpha_L} \fa u \in \mc{B}_b(\mbb{R}^d), \, t \in (0,T].
\end{equation*}
For $\alpha_L>1$ this estimate is not good enough for our purpose, we need a higher regularity of $P_t u$.

\begin{prop} \label{app-1}
	Let $(X_t)_{t \geq 0}$ be a Feller  process with symbol $q(x,\xi) = |\xi|^{\alpha(x)}$, $x,\xi \in \mbb{R}^d$, for a mapping $\alpha: \mbb{R}^d \to (0,2)$ which is bounded away from zero, i.e.\ $\alpha_L := \inf_{x \in \mbb{R}^d} \alpha(x)>0$, and $\gamma$-H\"{o}lder continuous for $\gamma \in (0,1)$. For any $T>0$ and $\kappa \in (0,\alpha_L)$ there exists a constant  $C>0$ such that the semigroup $(P_t)_{t \geq 0}$ satisfies \begin{equation}
		\|P_t u\|_{\mc{C}_b^{\kappa}(\mbb{R}^d)} \leq C \|u\|_{\infty}  t^{-\kappa/\alpha_L} \fa u \in \mc{B}_b(\mbb{R}^d), \, t \in (0,T]. \label{app-eq1}
	\end{equation}
	In particular, $(P_t)_{t \geq 0}$ has the strong Feller property. The constant $C>0$ depends continuously on $\alpha_L \in (0,2)$, $\alpha_L-\kappa \in (0,\alpha_L)$, $\|\alpha\|_{\mc{C}_b^{\gamma}(\mbb{R}^d)} \in [0,\infty)$ and $T \in [0,\infty)$.
\end{prop}

For the proof of Proposition~\ref{app-1} we use a representation for the transition density $p$ which was obtained in \cite{matters} using a parametrix construction, see also \cite{parametrix}. For $\varrho \in (0,2)$ denote by $p^{\varrho}(t,x)$ the transition density of an isotropic $\varrho$-stable L\'evy process and set \begin{equation*}
p_0(t,x,y) := p^{\alpha(y)}(t,x-y), \qquad t>0, \,x,y \in \mbb{R}^d.
\end{equation*}
The transition density $p$ of $(X_t)_{t \geq 0}$ has the representation \begin{equation}
p(t,x,y) = p_0(t,x,y)+ (p_0 \circledast \Phi)(t,x,y), \qquad t>0,\,x,y \in \mbb{R}^d \label{app-eq3}
\end{equation}
where $\circledast$ is the time-space convolution and $\Phi$ is a suitable function satisfying
\begin{equation}
	\sup_{x \in \mbb{R}^d} \int_{\mbb{R}^d} |\Phi(t,x,y)| \, dy \leq C_1 t^{-1+\lambda}, \qquad t \in (0,T) \label{app-eq5},
\end{equation}
for some constant $\lambda>0$ and $C_1=C_1(T)>0$. For further details we refer the reader to Appendix~\ref{pix} where we collect the material from \cite{matters} which we need in this article.

\begin{proof}[Proof of Proposition~\ref{app-1}]
	Fix $T>0$, $u \in \mc{B}_b(\mbb{R}^d)$ and $\kappa \in (0,\alpha_L)$. Since $\|P_t u\|_{\infty} \leq \|u\|_{\infty}$ it suffices to show that the iterated differences of order $2$, cf.\ \eqref{def-eq3}, satisfy \begin{equation*}
		\sup_{x \in \mbb{R}^d} |\Delta_h^2 P_t u(x)| \leq C t^{-\kappa/\alpha_L} \|u\|_{\infty} \fa t \in (0,T],\, |h| \leq 1.
	\end{equation*}
	Because of the representation \eqref{app-eq3} we have \begin{align*}
	|\Delta_h^2 P_t u(x)| 
	&\leq |\Delta_h^2 P_t^{(0)} u(x)| + |\Delta_h^2 P_t^{(1)} u(x)|
	\end{align*}
	for any $x,h \in \mbb{R}^d$ and $t \in (0,T]$ where \begin{align*}
	P_t^{(0)} u(z) := \int_{\mbb{R}^d} u(y) p_0(t,z,y) \, dy \quad \text{and} \quad
	P_t^{(1)} u(z) := \int_{\mbb{R}^d} u(y) (p_0 \circledast \Phi)(t,z,y) \, dy.
	\end{align*}
	We estimate the terms separately; we start with $P^{(0)}$. The transition density $p^{\varrho}(t,x)$ of an isotropic $\varrho$-stable L\'evy process is twice differentiable, and by  there exists a constant $c_1>0$ such that \begin{equation}
		|p^{\varrho}(t,x)| \leq c_1 S(x,\varrho,t) \quad |\partial_{x_i} p^{\varrho}(t,x)| \leq c_1 t^{-1/\varrho} S(x,\varrho,t) \quad |\partial_{x_i} \partial_{x_j} p^{\varrho}(t,x)| \leq c_1 t^{-2/\varrho} S(x,\varrho,t) \label{app-eq55}
	\end{equation}
	where \begin{equation}
		S(x,\varrho,t) := \min \left\{ t^{-d/\alpha}, \frac{t}{|x|^{d+\alpha}} \right\}, \label{app-eq56}
	\end{equation}
	and $\varrho \in [\alpha_L,\|\alpha\|_{\infty}]$, $t \in (0,T)$, $x \in \mbb{R}^d$ and $i,j \in \{1,\ldots,d\}$, cf.\ Lemma~\ref{pix-3}. For the parametrix $p_0(t,x,y) = p^{\alpha(y)}(t,x-y)$ this implies, by Taylor's formula, that there exists is $c_2>0$ such that \begin{equation*}
	|p_0(t,x+2h,y)-2p_0(t,x+h,y)+p_0(t,x,y)| \leq  c_2 t^{-2/\alpha(y)} |h|^2 S(\eta(x,h)-y,\alpha(y),t), \quad x,h \in \mbb{R}^d
	\end{equation*}
	for some intermediate value $\eta(x,h) \in B(x,2h)$. As $t \leq T$ we find that \begin{equation*}
	|p_0(t,x+2h,y)-2p_0(t,x+h,y)+p_0(t,x,y)| \leq  c_3 t^{-2/\alpha_L} |h|^2 S(\eta(x,h)-y,\alpha(y),t), \quad x,h \in \mbb{R}^d 
	\end{equation*}
	for a suitable constant $c_3 = c_3(T,\alpha_L,\|\alpha\|_{\infty})$. On the other hand, \eqref{app-eq55} gives \begin{align*}
	|p_0(t,x+2h,y)-&2p_0(t,x+h,y)+p_0(t,x,y)|\\
	&\leq c_1 (S(x+2h-y,\alpha(y),t)+2S(x+h-y,\alpha(y),t)+ S(x-y,\alpha(y),t)).
	\end{align*}
	Combining both estimates we obtain that there exists a constant $c_4 = c_4(T,\alpha_L,\|\alpha\|_{\infty})$ such that
	\begin{equation}
	|p_0(t,x+2h,y)-2p_0(t,x+h,y)+p_0(t,x,y)| \leq c_4 |h|^{\kappa} t^{-\kappa/\alpha_L} U(t,x,y,h) \label{app-eq6}
	\end{equation}
	for \begin{equation*}
	U(t,x,y,h) := S(\eta(x,h)-y,\alpha(y),t) + S(x+h-y,\alpha(y),t)+S(x-h-y,\alpha(y),t)+S(x-y,\alpha(y),t),
	\end{equation*}
	cf.\ Lemma~\ref{ineq-1} with $r:=t^{1/\alpha_L}$. Hence, \begin{align*}
	|P_t^{(0)}u(x+2h)-2P_t^{(0)}u(x+h)+P_t^{(0)}u(x)| 
	\leq c_4 \|u\|_{\infty} t^{-\kappa/\alpha_L} |h|^\kappa \int_{\mbb{R}^d} U(t,x,y,h) \, dy
	\end{align*}
	for any $x,h \in \mbb{R}^d$ and $t \in (0,T)$. Since  \begin{equation}
	c_T := \sup_{t \in (0,T)}\sup_{z \in \mbb{R}^d} \int_{\mbb{R}^d} S(z-y,\alpha(y),t) \, dy <\infty, \label{app-eq7}
	\end{equation}
	cf.\ Appndix~\ref{pix}, we have \begin{equation}
	\sup_{t \in (0,T)} \sup_{z \in \mbb{R}^d} \int_{\mbb{R}^d} U(t,z,y,h) \, dy \leq 4 c_T < \infty, \label{app-eq75}
	\end{equation}
	and therefore we conclude that \begin{equation*}
	|P_t^{(0)}u(x+2h)-2P_t^{(0)}u(x+h)+P_t^{(0)}u(x)|
	\leq 4c_4 c_T \|u\|_{\infty} t^{-\kappa/\alpha_L} |h|^\kappa.
	\end{equation*}
	It remains to establish the H\"{o}lder estimate for $P^{(1)}_t$. By \eqref{app-eq6}, we have \begin{align*}
	&|(p_0 \circledast \Phi)(t,x+2h,y)-2(p_0 \circledast \Phi)(t,x+h,y)+ (p_0 \circledast \Phi)(t,x,y)| \\
	&\leq c_4 |h|^{\kappa} \int_0^t \!\! \int_{\mbb{R}^d} (t-s)^{-\kappa/\alpha_L} U(t-s,x,z,h) |\Phi(s,z,y)| \, dz \, ds.
	\end{align*}
	Integrating with respect to $y \in \mbb{R}^d$, it follows from \eqref{app-eq5} and \eqref{app-eq75} that 
	\begin{align*}
	|P_t^{(1)} u(x+2h)-2P_t^{(1)}u(x+h)+P_t^{(1)}u(x)|
	&\leq c_6 |h|^{\kappa} \|u\|_{\infty} \int_0^t (t-s)^{-\kappa/\alpha_L} s^{-1+\lambda} \, ds \\
	&\leq c_7 |h|^{\kappa} t^{-\kappa/\alpha_L} \|u\|_{\infty}
	\end{align*}
	for suitable constants $c_6$ and $c_7$. 
	Combining the estimates we find that \eqref{app-eq1} holds for some finite constant $C>0$. The continous dependence of $C$ on the parameters $\alpha_L-\kappa \in (0,\alpha_L)$, $\alpha_L \in (0,2)$, $\|\alpha\|_{C_b^{\gamma}}>0$ and $T>0$ follows from the fact that each of the constants in this proof depends continuously on these parameters.
\end{proof}

In Proposition~\ref{app-1} we studied the regularity of $x \mapsto P_t u(x)$ for measurable functions $u$.  The next result is concerned with the regularity of $P_t u(\cdot)$ for H\"{o}lder continuous functions $u$. It is natural to expect that $P_t u$ ``inherits'' some regularity from $u$.

\begin{prop} \label{app-25}
	Let $(X_t)_{t \geq 0}$ be a Feller process with symbol $q(x,\xi) = |\xi|^{\alpha(x)}$, $x,\xi \in \mbb{R}^d$, for a mapping $\alpha: \mbb{R}^d \to (0,2)$ such that $\alpha_L := \inf_{x \in \mbb{R}^d} \alpha(x)>0$ and $\alpha \in C_b^{\gamma}(\mbb{R}^d)$ for some $\gamma \in (0,1)$ satisfying  \begin{equation*}
		\gamma> \gamma_0 :=\|\alpha\|_{\infty}-\alpha_L.
	\end{equation*}
	For any $T>0$, $\kappa \in (0,\alpha_L)$ and $\eps \in (\gamma_0,\min\{\gamma,\alpha_L\})$ there exists a constant $C>0$ such that the semigroup $(P_t)_{t \geq 0}$ of $(X_t)_{t \geq 0}$ satisfies \begin{equation}
		\|P_t u\|_{\mc{C}_b^{\kappa + \min\{\delta,\gamma\}-\eps}(\mbb{R}^d)} \leq C (1+|\log t|) t^{-\kappa/\alpha_L}  \|u\|_{\mc{C}_b^{\min\{\delta,\gamma\}}(\mbb{R}^d)}, \qquad u \in \mc{C}_b^{\delta}(\mbb{R}^d), \label{app-eq31}
	\end{equation} 
	for all $\delta>0$ and $t \in (0,T]$. The constant $C>0$ depends continuously on $\alpha_L \in (0,2)$, $\kappa-\alpha_L \in (0,2)$, $(\eps-\|\alpha\|_{\infty})/\alpha_L \in (1,\infty)$, $\|\alpha\|_{C_b^{\gamma}(\mbb{R}^d)} \in [0,\infty)$ and $T \in [0,\infty)$.
\end{prop}

For the proof of the Schauder estimates, Corollary~\ref{ex-23}, we will apply Proposition~\ref{app-25} for an isotropic stable-like process $(X_t)_{t \geq 0}$ with symbol $q(x,\xi)=|\xi|^{\alpha(x)}$ for a ``truncated'' function $\alpha$ of the form \begin{equation*}
	\alpha(x) := (\varrho(x_0)-\delta) \vee \varrho(x) \wedge (\varrho(x_0)+\delta), \qquad x \in \mbb{R}^d
\end{equation*}
where $x_0 \in \mbb{R}^d$ is fixed and $\delta>0$ is a constant which we can choose as small as we like; in particular $\gamma_0 :=\|\alpha\|_{\infty}-\alpha_L \leq 2 \delta$ is small and therefore the assumptions $\eps>\gamma_0$ and $\gamma>\gamma_0$ in Proposition~\ref{app-25} are not a restriction. Let us mention that both assumptions, i.\,e.\ $\eps>\gamma_0$ and $\gamma>\gamma_0$, come into play when estimating one particular term in the proof of Proposition~\ref{app-25}, see \eqref{app-eq61} below; a more careful analysis of this term would probably allow us to relax these two conditions. 

\begin{proof}[Proof of Proposition~\ref{app-25}]
	 Fix $\eps \in (\gamma_0, \gamma \wedge \alpha_L)$, $\kappa \in (0,\alpha_L)$ and $T>0$. First of all, we note that it clearly suffices to show \eqref{app-eq31} for $u \in \mc{C}_b^{\delta}(\mbb{R}^d)$ with $\delta \leq \gamma \leq 1$. Throughout the first part of this proof, we will assume that \begin{equation}
		\kappa \leq 1. \label{app-eq33}
	\end{equation}
	Under \eqref{app-eq33} the assertion follows if we can show that \begin{equation*}
		|\Delta_h^2 P_t u(x)| \leq C \|u\|_{\mc{C}_b^{\delta}(\mbb{R}^d)} (1+|\log(t)|) t^{-\kappa/\alpha_L} |h|^{\kappa+\delta-\eps}, \qquad x \in \mbb{R}^d, \, |h| \leq 1, \, t \in (0,T]
	\end{equation*}
	where $\Delta_h^2$ denotes as usual the iterated difference operator, cf.\ \eqref{def-eq3}. For the proof of this inequality we use again the parametrix construction of the transition density $p$ of $(X_t)_{t \geq 0}$, \begin{equation}
		p(t,x,y) = p_0(t,x,y) + (p_0 \circledast \Phi)(t,x,y), \qquad t>0, \, x,y \in \mbb{R}^d \label{app-eq35}
	\end{equation}
	where \begin{equation}
		p_0(t,x,y) = p^{\alpha(y)}(t,x-y),\qquad t>0, \, x,y \in \mbb{R}^d, \label{app-eq36}
	\end{equation}
	see Appendix~\ref{pix} for details. Since \begin{align*}
		\Delta_h P_t u(x) 
		&=  \int_{\mbb{R}^d} \Delta_h u(y) p(t,x,y) \, dy - \int_{\mbb{R}^d} \left( u(y+h) p(t,x,y) - u(y) p(t,x+h,y) \right) \, dy \\
		&= \int_{\mbb{R}^d} \Delta_h u(y) p(t,x,y) \, dy - \int_{\mbb{R}^d} u(y+h) (p(t,x,y)-p(t,x+h,y+h)) \, dy
	\end{align*}
	we find that $\Delta_h^2 P_t f(x) = J_1-J_2$ where \begin{align} \label{app-eq40} \begin{aligned} 
		J_1 &:= \int_{\mbb{R}^d} \Delta_h u(y) \left( p(t,x+h,y)-p(t,x,y) \right) \, dy \\
		J_2 &:= \int_{\mbb{R}^d} u(y+h) q(t,x,y) \, dy.\end{aligned}
	\end{align}
	with \begin{equation*}
		q(t,x,y) := p(t,x+h,y)-p(t,x+2h,y+h)-p(t,x,y)+p(t,x+h,y+h)
	\end{equation*}
	for fixed $h$.
	We estimate the terms separately. For fixed $h \in \mbb{R}^d$, $|h| \leq 1$, define an auxiliary function $v$ by $v(y) := \Delta_h u(y)$. Proposition~\ref{app-1} gives \begin{equation*}
		|J_1| 
		\leq |h|^{\kappa} \|P_t v \|_{\mc{C}_b^{\kappa}(\mbb{R}^d)} 
		\leq C_1 |h|^{\kappa} \|v\|_{\infty} t^{-\kappa/\alpha_L}, \qquad t \in (0,T],
	\end{equation*}
	and so, by the definition of $v$ and the H\"older continuity of $u$, \begin{equation*}
		|J_1| \leq C_1 |h|^{\kappa+\delta} \|u\|_{\mc{C}_b^{\delta}(\mbb{R}^d)} t^{-\kappa/\alpha_L}, \qquad t \in (0,T].
	\end{equation*}
	It remains to establish the corresponding estimate for $J_2$, and to this end we use the representation \eqref{app-eq35} for the transition density $p$.  \par
	\textbf{Step 1:} There exists a constant $c_1>0$ such that 
	\begin{equation}
		q_0(t,x,y):= p_0(t,x+h,y)-p_0(t,x+2h,y+h)-p_0(t,x,y) + p_0(t,x+h,y+h) \label{app-eq47}
	\end{equation}
	satisfies \begin{equation*}
		\int_{\mbb{R}^d} |q_0(t,x,y)| \,dy \leq c_1 |h|^{\kappa+\gamma} (1+|\log(t)|) t^{-\kappa/\alpha_L} \fa x,h \in \mbb{R}^d, \, t \in (0,T].
	\end{equation*}
	\emph{Indeed:} If we denote by $p^{\varrho}$ the transition density of the $d$-dimensional isotropic $\varrho$-stable L\'evy process, $\varrho \in (0,2)$, then there is a constant $c_2>0$ such that \begin{equation}
		\int_{\mbb{R}^d} \left| \frac{\partial}{\partial \varrho} p^{\varrho}(t,x) \right| \, dx \leq c_2 (1+|\log(t)|) \quad \int_{\mbb{R}^d} \left| \frac{\partial}{\partial x_j} \frac{\partial}{\partial \varrho} p^{\varrho}(t,x) \right| \, dx \leq c_2 (1+|\log(t)|) t^{-1/\alpha_L} \label{app-eq44}
	\end{equation}
	for all $t \in (0,T]$, $j \in \{1,\ldots,d\}$ and $\varrho \in [\alpha_L,\|\alpha\|_{\infty}] \subseteq (0,2]$, cf.\ Lemma~\ref{pix-3}. To shorten the notation, we fix $x,h \in \mbb{R}^d$ and $t \in (0,T]$, and write $q_0(y)$ for the function defined in \eqref{app-eq47}. By the very definition of $p_0$, cf.\ \eqref{app-eq36}, we have \begin{equation*}
		|q_0(y)| = |p^{\alpha(y)}(t,x+h-y) - p^{\alpha(y+h)}(t,x+h-y)-p^{\alpha(y)}(t,x-y) + p^{\alpha(y+h)}(t.x-y)|,
	\end{equation*}
	and so, by the fundamental theorem of calculus and the mean value theorem, \begin{align}
		|q_0(y)| 
		&= \left| \int_{\alpha(y)}^{\alpha(y+h)} \left(\partial_{\varrho} p^{\varrho}(t,x+h-y) - \partial_{\varrho} p^{\varrho}(t,x-y) \right) \, d\varrho \right| \label{app-eq46} \\
		&\leq |h| \int_{\alpha(y)}^{\alpha(y+h)} \left| \nabla_x \partial_{\varrho} p^{\varrho}(t,\eta_{\varrho}(x,h)-y) \right| \, d\varrho \notag
	\end{align}
	for some intermediate value $\eta_{\varrho}(x,h) \in B(x,h)$. Integrating with respect to $y$ and using \eqref{app-eq44} we obtain that \begin{align}
		\int_{\mbb{R}^d} |q_0(y)| \, dy
		&\leq c_3 (1+|\log(t)|) t^{-1/\alpha_L} |h| \sup_{y \in \mbb{R}^d} \int_{\alpha(y)}^{\alpha(y+h)} d\varrho\\
	&	\leq c_3 \|\alpha\|_{\mc{C}_b^{\gamma}(\mbb{R}^d)} (1+|\log(t)|) t^{-1/\alpha_L} |h|^{1+\gamma}. \label{app-eq48}
	\end{align}
	On the other hand, it follows from \eqref{app-eq46} and the H\"older continuity of $\alpha$ that \begin{align*}
		\int_{\mbb{R}^d} |q_0(y)| \, dy
		\leq |h|^{\gamma} \|\alpha\|_{\mc{C}_b^{\gamma}(\mbb{R}^d)} \sup_{\varrho \in [\alpha_L,\|\alpha\|_{\infty}]} \sup_{\eta \in \mbb{R}^d} \int_{\mbb{R}^d} |\partial_{\varrho} p^{\varrho}(t,\eta-y)| \, dy.
	\end{align*}
	Hence, by \eqref{app-eq44}, \begin{align}
		\int_{\mbb{R}^d} |q_0(y)| \, dy
		\leq c_4 |h|^{\gamma} (1+|\log(t)|). \label{app-eq50}
	\end{align}
	Combining \eqref{app-eq48} and \eqref{app-eq50} we find that \begin{equation*}
		\int_{\mbb{R}^d} |q_0(y)| \, dy \leq c_5 |h|^{\kappa+\gamma} (1+|\log(t)|) t^{-\kappa/\alpha_L}, \qquad \kappa \in [0,\alpha_L];
	\end{equation*}
	the reasoning is very similar to the proof of Lemma~\ref{ineq-1}, alternatively we can use an interpolation theorem. \par
	\textbf{Step 2:} There exists a constant $c>0$ such that \begin{equation*}
		|J_2| \leq c |h|^{\kappa+\delta-\eps} \|u\|_{\mc{C}_b^{\delta}(\mbb{R}^d)} (1+|\log(t)|) t^{-\kappa/\alpha_L} \fa t\in (0,T], \, |h| \leq 1, \, x \in \mbb{R}^d;
	\end{equation*}
	recall that $\eps \in (\gamma_0,\alpha_L \wedge \gamma)$ has been fixed at the beginning of the proof.\par
	\emph{Indeed:} Because of the decomposition \eqref{app-eq35}, we have $J_2 = J_{2,1} + J_{2,2}$ for \begin{align} \label{app-eq52} \begin{aligned}
		J_{2,1} &:= \int_{\mbb{R}^d} u(y+h) q_0(t,x,y) \, dy \\
		J_{2,2} &:= \int_{\mbb{R}^d} u(y+h) \left((p_0 \circledast \Phi)(t,x+h,y) - (p_0 \circledast \Phi)(t,x+2h,y+h)\right) \, dy \\
		&\qquad +\int_{\mbb{R}^d} u(y+h) \left((p_0 \circledast \Phi)(t,x+h,y+h) -(p_0 \circledast \Phi)(t,x,y) \right) \, dy \end{aligned}
	\end{align}
	with $q$ defined in \eqref{app-eq47}. It follows from Step 1 that \begin{equation*}
		|J_{2,1}| \leq c_1 \|u\|_{\mc{C}_b^{\delta}(\mbb{R}^d)} (1+|\log(t)|) t^{-\kappa/\alpha_L} |h|^{\kappa+\delta}, \quad t \in (0,T].
	\end{equation*}
	It remains to estimate $J_{2,2}$. By the definition of the time-space convolution, we have \begin{align*}
		&(p_0 \circledast \Phi)(t,x+h,y) - (p_0 \circledast \Phi)(t,x+2h,y+h)-(p_0 \circledast \Phi)(t,x,y) + (p_0 \circledast \Phi)(t,x+h,y+h) \\
		&= \int_0^t \!\! \int_{\mbb{R}^d} (p_0(t-s,x+h,z)-p_0(t-s,x,z)) \Phi(s,z,y) \, dz \, ds \\
		&\quad - \int_0^t \!\! \int_{\mbb{R}^d} (p_0(t-s,x+2h,z)-p_0(t-s,x+h,z)) \Phi(s,z,y+h) \, dz \, ds \\
		&= \int_0^t \!\! \int_{\mbb{R}^d}q_0(t-s,x,z) \Phi(s,z,y) \, dz \, ds \\
		&\quad - \int_0^t\!\! \int (p_0(t-s,x+2h,z+h)-p_0(t-s,x+h,z+h)) (\Phi(s,z+h,y+h)-\Phi(s,z,y)) \, dz \, ds \\
		&=: H_1(t,y)-H_2(t,y).
	\end{align*}
	Integrating with respect to $y$ and applying Tonelli's theorem, we obtain that \begin{align*}
		\left| \int_{\mbb{R}^d} u(y+h) H_1(t,y) \, dy \right|  
		&\leq \|u\|_{\infty} \int_0^t \left( \sup_{\eta \in \mbb{R}^d} \int_{\mbb{R}^d} |\Phi(s,\eta,y)| \, dy \right) \left( \int_{\mbb{R}^d} |q_0(t-s,x,z)| \, dz \right) \, ds.
	\end{align*}
	Thus, by \eqref{app-eq5} and Step 1, \begin{align}
		\left| \int_{\mbb{R}^d} u(y+h) H_1(t,y) \, dy \right|  
		\leq c_6 |h|^{\kappa+\gamma} \|u\|_{\infty} \int_0^t s^{-1+\lambda_1} (1+|\log(t-s)|) (t-s)^{-\kappa/\alpha_L} \, ds \label{app-eq60}
	\end{align}
	for a suitable constant $c_6>0$ and $\lambda_1>0$. It remains to estimate $H_2$. We claim that there exist constants $c_7>0$ and $\lambda_2>0$ such that  \begin{equation}
		\sup_{z \in \mbb{R}^d} \int_{\mbb{R}^d} |\Phi(t,z+h,y+h)-\Phi(t,z,y)| \, dy \leq c_7 |h|^{\gamma-\eps} t^{-1+\lambda_2} \label{app-eq61}
	\end{equation}
	for all $t \in (0,T]$ and $|h| \leq 1$; here $\eps \in (\gamma_0,\alpha_L \wedge \gamma)$ is the constant which we have chosen at the beginning of the proof. We postpone the proof of \eqref{app-eq61} to the end of this subsection, see Lemma~\ref{app-27} below. Using \eqref{app-eq61} and the fact that
	\begin{equation*}
		\int_{\mbb{R}^d}  |p_0(t-s,x+2h,z+h)-p_0(t-s,x+h,z+h)| \, dz \leq c_{8} |t-s|^{-\kappa/\alpha_L} |h|^{\kappa}
	\end{equation*}
	for some constant $c_{8}>0$, which follows by a similar reasoning as in the first part of the proof of Proposition~\ref{app-1}, we obtain that \begin{align*}
		\left| \int_{\mbb{R}^d} u(y+h) H_2(t,y) \, dy \right|
		\leq c_7 c_{8} \|u\|_{\infty} |h|^{\gamma+\kappa-\eps} \int_0^t s^{-1+\lambda_2} (t-s)^{-\kappa/\alpha_L} \, ds.
	\end{align*}
	Combining this estimate with \eqref{app-eq60} gives \begin{equation*}
		|J_{2,2}| \leq (c_6+c_7 c_8)  \|u\|_{\infty} |h|^{\gamma+\kappa-\eps} \int_0^t s^{-1+\lambda} (t-s)^{-\kappa/\alpha_L} (1+|\log(t-s)|) \, ds. 
	\end{equation*}
	Hence, \begin{align*}
		|J_{2,2}| 
		&\leq c_{9}  \|u\|_{\infty}|h|^{\gamma+\kappa-\eps} t^{-\kappa/\alpha_L} \int_0^1 r^{-1+\lambda} (1-r)^{-\kappa/\alpha_L} (1+|\log(1-r)|) \, dr
	\end{align*}
	for all $t \in (0,T]$ where $\lambda := \min\{\lambda_1,\lambda_2\}$. This finishes the proof of Step 2 and, hence, of Proposition~\ref{app-25} for the case $\kappa \leq 1$. 
	
	If $\kappa>1$, we need to estimate the iterated differences of third order $\Delta_h^3 P_t u(x)$. Fix $|h| \leq 1$. Since $\Delta_h^2 P_t f(x) =J_1(x) -J_2(x)$ with $J_1=J_1(x)$, $J_2=J_2(x)$ defined in \eqref{app-eq40}, we have \begin{equation*}
		\Delta_h^3 P_t u(x) = \Delta_h J_1(x) - \Delta_h J_2(x).
	\end{equation*}
	As before, we estimate the terms separately. If we define an auxilary function $v(y) := \Delta_h u(y)$, then, by \eqref{app-eq40}, \begin{align*}
		\Delta_h J_1(x)
		&= \int_{\mbb{R}^d} v(y) \left( p(t,x+2h,y)-2p(t,x+h,y) +p(t,x,y) \right) \, dy \\
		&= P_t v(x+2h)-2P_t v(x+h)+P_t v(x) = \Delta_h^2 P_t v(x).
	\end{align*}
	By Proposition~\ref{app-1}, this gives \begin{align*}
		|\Delta_h J_1(x)|
		\leq C_1 |h|^{\kappa} \|v\|_{\infty} t^{-\kappa/\alpha_L}, \quad t \in (0,T],\,x \in \mbb{R}^d,
	\end{align*}
	and so, by the definition of $v$ and the H\"{o}lder continuity of $u$, \begin{equation*}
		|\Delta_h J_1(x)|
		\leq C_1 \|u\|_{\mc{C}_b^{\delta}(\mbb{R}^d)} |h|^{\kappa+\delta} t^{-\kappa/\alpha_L}, \quad t \in (0,T], x \in \mbb{R}^d.
	\end{equation*}
	In order to estimate $\Delta_h J_2$ we use a similar procedure as in the case $\kappa \leq 1$. Denote by $q_0$ the function defined in \eqref{app-eq47}.
	
	\textbf{Step 3:} There exists a constant $c_1>0$ such that \begin{equation*}
		\int_{\mbb{R}^d} |q_0(t,x+h,y)-q_0(t,x,h)| \, dy \leq c_1 |h|^{\kappa+\gamma} (1+|\log(t)|) t^{-\kappa/\alpha_L}, \quad |h| \leq 1, \, x \in \mbb{R}^d,  \in (0,T].
	\end{equation*}
	\emph{Indeed:} By definition of $q_0$ and definition of $p_0$, cf.\ \eqref{app-eq36}, \begin{align*}
		q_0(t,x+h,y)-q_0(t,x,y)
		&= p_0(t,x,y)-p_0(t,x+3h,y+h)+2p_0(t,x+2h,y+h)\\
		&\quad -2p_0(t,x+h,y)-p_0(t,x+h,y+h) + p_0(t,x+2h,y) \\
		&= - p^{\alpha(y+h)}(t,x-y+2h) + 2p^{\alpha(y+h)}(t,x-y+h) - p^{\alpha(y+h)}(x-y) \\ 
		&\quad+ p^{\alpha(y)}(t,x-y+2h) - 2p^{\alpha(y)}(t,x-y+h) + p^{\alpha(y)}(t,x-y)
	\end{align*}
	where $p^{\varrho}$ denotes as usual the transition density of the L\'evy process with characteristic exponent $|\xi|^{\varrho}$. Hence, \begin{equation}
		q_0(t,x+h,y)-q_0(t,x,y)
		= \int_{\alpha(y)}^{\alpha(y+h)} \partial_{\varrho} \left( - p^{\varrho}(t,x-y+2h)+2p^{\varrho}(t,x-y+h) - p^{\varrho}(t,x-y) \right) \, d\varrho. \label{app-eq91}
	\end{equation}
	Applying Taylor's formula, integrating with respect to $y$ and using Lemma~\ref{pix-3}, it follows that \begin{align*}
		\int_{\mbb{R}d} |q_0(t,x+h,y)-q_0(t,x,y)| \, dy
		&\leq c_2 (1+|\log(t)|) t^{-2/\alpha_L} |h|^2 \sup_{y \in \mbb{R}^d} \int_{\alpha(y)}^{\alpha(y+h)} \, d\varrho \\
		&\leq c_2 \|\alpha\|_{\mc{C}_b^{\gamma}(\mbb{R}^d)} (1+|\log(t)|) |h|^{2+\gamma} t^{-2/\alpha_L}.
	\end{align*}
	On the other hand, \eqref{app-eq91} gives \begin{align*}
		|q_0(t,x+h,y)-q_0(t,x,y)|
		\leq 2 \sum_{j=0}^2 \int_{\alpha(y)}^{\alpha(y+h)} |\partial_{\varrho} p^{\varrho}(t,x-y+jh)| \, d\varrho.
	\end{align*}
	Another application of Lemma~\ref{pix-3} (with $k=0$) yields \begin{align*}
		\int_{\mbb{R}^d} |q_0(t,x+h,y)-q_0(t,x,y)| \, dy
		\leq c_3 \|\alpha\|_{\mc{C}_b^{\gamma}(\mbb{R}^d)} |h|^{\gamma} (1+|\log(t)|),
	\end{align*}
	and combining this with the previous estimate we get the assertion by a standard interpolation argument.
	
	\textbf{Step 4:} There exists a constant $c>0$ such that \begin{equation*}
		|J_2(x+h)-J_2(x)| \leq c |h|^{\kappa+\delta-\eps} \|u\|_{\mc{C}_b^{\delta}(\mbb{R}^d)} (1+|\log(t)|) t^{-\kappa/\alpha_L}, \quad t \in (0,T], \, |h| \leq 1, \, x \in \mbb{R}^d,
	\end{equation*}
	with $\eps$ chosen at the beginning of the proof.
	
	\emph{Indeed:} As in the first part of this proof, we write $J_2=J_{2,1}+J_{2,2}$ where \begin{align*}
		J_{2,1}(x) = \int_{\mbb{R}^d} u(y+h) q_0(t,x,y) \, dy \qquad J_{2,2}(x) = \int u(y+h) (H_1(t,x,y)-H_2(t,x,y)) \, dy,
	\end{align*}
	cf.\ Step 2. By Step 3, \begin{align*}
		|J_{2,1}(x+h)-J_{2,1}(x)|
		\leq c_1 \|u\|_{\infty} |h|^{\kappa+\delta} (1+|\log(t)|) t^{-\kappa/\alpha_L},
	\end{align*}
	and so it just remains to estimate $J_{2,2}(x+h)-J_{2,2}(x)$. It follows from the definition of $H_1$ and Fubini's theorem that \begin{align*}
		&\left| \int_{\mbb{R}^d} u(y+h) (H_1(t,x+h,y)-H_1(t,x,y)) \, dy \right|\\
		&\leq \int_0^t \left( \sup_{\eta \in \mbb{R}^d} |\Phi(s,\eta,y)| \, dy \right) \int_{\mbb{R}^d} |q_0(t-s,x+h,z)-q_0(t-s,x,z)| \, dz \, ds.
	\end{align*}
	By \eqref{app-eq5} and Step 3, there exist constants $c_4>0$ and $\lambda_1>0$ such that \begin{align*}
		\left| \int_{\mbb{R}^d} u(y+h) (H_1(t,x+h,y)-H_1(t,x,y)) \, dy \right|
		\leq c_4 |h|^{\kappa+\gamma} \int_0^t s^{-1+\lambda_1} (1+|\log(t-s)|) (t-s)^{-\kappa/\alpha_L} \, ds 
	\end{align*}
	for any $t \in (0,T]$, $x \in \mbb{R}^d$ and $|h| \leq 1$. For $H_2$ we note that \begin{align*}
		H_2(t,x+h,y)-H_2(t,x,y)
		= \int_0^t\!\!\int_{\mbb{R}^d} r(t-s,x,z) (\Phi(s,z+h,y+h)-\Phi(s,z,y)) \, dz \, ds
	\end{align*}
	where \begin{align*}
		r(t-s,x,z)
		&:= p_0(t-s,x+3h,z+h)-2p_0(t-s,x+2h,z+h)+ p_0(t-s,x+h,z+h)
	\end{align*}
	for fixed $h$. Applying Taylor's formula and using \eqref{pix-eq11}, we obtain that \begin{align*}
		\int_{\mbb{R}^d} |r(t-s,x,z)| \, dz
		\leq c_5 |h|^{\kappa} (t-s)^{-\kappa/\alpha_L},
	\end{align*}
	see the proof of Proposition~\ref{app-1} for a very similar reasoning. Combining this estimate with \eqref{app-eq61}, \begin{equation*}
		\left| \int_{\mbb{R}^d} u(y+h) (H_2(t,x+h,y)-H_2(t,x,y)) \, dy \right|
		\leq c_6 \|u\|_{\infty} |h|^{\gamma+\kappa-\eps} \int_0^t s^{-1+\lambda_2} (t-s)^{-\kappa/\alpha_L} \, ds
	\end{equation*}
	for suitable constants $\lambda_2>0$ and $c_6>0$. This gives the desired estimates for $J_{2,2}$, see the end of Step 2 for details, and hence for $J_2$.
\end{proof}

\begin{lem} \label{app-27} 
	Let $(X_t)_{t \geq 0}$ be a Feller process with symbol $q(x,\xi) = |\xi|^{\alpha(x)}$ satisfying the assumptions of Proposition~\ref{app-25}, and denote by \begin{equation*}
		p(t,x,y) =p_0(t,x,y) + (p_0 \circledast \Phi)(t,x,y)
	\end{equation*}
	the parametrix representation of the transition density $p$ of $(X_t)_{t \geq 0}$. For any $T>0$ and any $\eps \in (\gamma_0,\gamma \wedge \alpha_L)$ there exist finite constants $C>0$ and $\lambda>0$ such that \begin{equation*}
		\int_{\mbb{R}^d} |\Phi(t,x+h,y+h)-\Phi(t,x,y)| \, dy
		\leq C |h|^{\gamma-\eps} t^{-1+\lambda}
	\end{equation*}
	for all $x \in \mbb{R}^d$, $|h| \leq 1$ and $t \in (0,T]$. The constant $C>0$ depends continuously on $\alpha_L \in (0,2)$, $\kappa-\alpha_L \in (0,2)$, $(\eps-\|\alpha\|_{\infty})/\alpha_L \in (1,\infty)$, $\|\alpha\|_{C_b^{\gamma}(\mbb{R}^d)} \in [0,\infty)$ and $T \in [0,\infty)$. The constant $\lambda>0$ depends continuously on $(\eps-\|\alpha\|_{\infty})/\alpha_L \in (1,\infty)$ and $(\gamma-\|\alpha\|_{\infty})/\alpha_L \in (1,\infty)$.
\end{lem}

\begin{proof}
	Fix $\eps \in (\gamma_0,\alpha_L \wedge \gamma)$. To keep the calculations as simple as possible we consider $T:=1$. For the proof we use that $\Phi$ has the representation \begin{equation}
		\Phi(t,x,y) = \sum_{i=1}^{\infty} F^{\circledast i}(t,x,y), \qquad t>0, \, x,y \in \mbb{R}^d \label{app-eq63}
	\end{equation}
	where $F^{\circledast i} := F \circledast F^{\circledast (i-1)}$ denotes the $i$-th convolution power of \begin{equation*}
		F(t,x,y) :=  (2\pi)^{-d} \int_{\mbb{R}^d} \left( |\xi|^{\alpha(y)}-|\xi|^{\alpha(x)} \right) e^{i \xi \cdot (y-x)} e^{-t |\xi|^{\alpha(y)}} \, d\xi, \qquad t>0, \, x,y \in \mbb{R}^d,
	\end{equation*}
	cf.\ Appendix~\ref{pix}. \par
	\textbf{Step 1:} There exist constants $C>0$ and $\lambda>0$ such that \begin{equation}
		\int_{\mbb{R}^d} |F(t,x+h,y+h)-F(t,x,y)| \, dy \leq C |h|^{\gamma-\eps} t^{-1+\lambda} \quad \text{for all} \, \, x \in \mbb{R}^d,\, |h| \leq 1, \,t \in (0,1). \label{app-eq65}
	\end{equation}
	\emph{Indeed:} For fixed $|h| \leq 1$ we write \begin{equation*}
		F(t,x+h,y+h)-F(t,x,y) = (2\pi)^{-d} \left(D_1(t,x,y) + D_2(t,x,y)\right)
	\end{equation*}
	where \begin{align*}
		D_1(t,x,y) &:= \int_{\mbb{R}^d} \left( \left( |\xi|^{\alpha(y+h)}-|\xi|^{\alpha(y)} \right)- \left( |\xi|^{\alpha(x+h)}-|\xi|^{\alpha(x)} \right) \right) e^{i \xi \cdot (y-x)} e^{-t |\xi|^{\alpha(y)}} \, d\xi \\
		D_2(t,x,y) &:= \int_{\mbb{R}^d} \left( |\xi|^{\alpha(y)}-|\xi|^{\alpha(x)} \right) e^{i \xi \cdot (y-x)} \left( e^{-t |\xi|^{\alpha(y+h)}}-e^{-t |\xi|^{\alpha(y)}} \right) \, d\xi.
	\end{align*}
	We estimate the terms separately. As $\alpha \in \mc{C}_b^{\gamma}(\mbb{R}^d)$ it follows that $x \mapsto r^{\alpha(x)} \in \mc{C}_b^{\gamma}(\mbb{R}^d)$ for any fixed $r \geq 0$ and \begin{equation*}
		\|r^{\alpha(\cdot)}\|_{\mc{C}_b^{\gamma}(\mbb{R}^d)} \leq \left(\|\alpha\|_{\mc{C}_b^{\gamma}(\mbb{R}^d)} |\log(r)| +1 \right)  \max\{r^{\alpha_L},r^{\|\alpha\|_{\infty}}\}.
	\end{equation*}
	Applying Lemma~\ref{ineq-2} we find that there exists a constant $c_1>0$ such that \begin{align*}
		\left| \left( r^{\alpha(y+h)} - r^{\alpha(y)} \right)- \left( r^{\alpha(x+h)}- r^{\alpha(x)} \right) \right|
		&\leq c_1 |h|^{\gamma-\eps} |x-y|^{\eps} \|r^{\alpha(\cdot)}\|_{\mc{C}_b^{\gamma}(\mbb{R}^d)}  \\
		&\leq c_1' |h|^{\gamma-\eps} |x-y|^{\eps} (|\log(r)|+1) \max\{r^{\alpha_L},r^{\|\alpha\|_{\infty}}\}
	\end{align*}
	for all $r \geq 0$, $x,y \in \mbb{R}^d$ and $|h| \leq 1$. By \cite[(proof of) Theorem 4.7]{matters} this implies that there is a constant $c_2>0$ such that \begin{equation*}
		|D_1(t,x,y)|
		\leq c_2 |h|^{\gamma-\eps} |x-y|^{\eps} \min \left\{(1+|\log(t)|) t^{-(d+\|\alpha\|_{\infty})/\alpha_L}, \frac{1+|\log(|x-y|)|}{\min\{|x-y|^{d+\alpha_L},|x-y|^{d+\|\alpha\|_{\infty}}\}} \right\}
	\end{equation*}
	for all $x,y \in \mbb{R}^d$, $t \in (0,1)$ and $|h| \leq 1$. Splitting up the domain of integration into three parts, \begin{equation*}
		\{y \in \mbb{R}^d; |x-y| < t^{1/\alpha_L}\} \quad \{y \in \mbb{R}^d; t^{1/\alpha_L} \leq |x-y| \leq 1\} \quad \{y \in \mbb{R}^d; |x-y| >1\}
	\end{equation*}
	we obtain that $\int_{\mbb{R}^d} |D_1(t,x,y)| \,dy$ is bounded by \begin{align*}
		&c_2 |h|^{\gamma-\eps} \left((1+|\log(t)|) t^{-(d+|\alpha\|_{\infty}-\eps)/\alpha_L} \int_{|z| < t^{1/\alpha_L}} \, dz +  \int_{t^{1/\alpha_L} \leq |z| \leq 1} \frac{1+|\log(|z|)|}{|z|^{d+\|\alpha\|_{\infty}-\eps}} \, dz + \int_{|z|>1} \frac{1+|\log(|z|)|}{|z|^{d+\alpha_L-\eps}} \, dz \right) \\
		&\leq c_2' |h|^{\gamma-\eps} (1+|\log(t)|) t^{-(|\alpha\|_{\infty}-\eps)/\alpha_L}.
	\end{align*}
	As $\eps> \gamma_0=\|\alpha\|_{\infty}-\alpha_L$ this means that there exists $\lambda_1>0$ such that \begin{equation*}
		\int_{\mbb{R}^d} |D_1(t,x,y)| \, dy \leq c_3 t^{-1+\lambda_1} |h|^{\gamma-\eps}, \qquad t \in (0,1), \,x \in \mbb{R}^d.
	\end{equation*}
	In order to estimate the second term we note that \begin{align*}
		D_2(t,x,y) = -t \int_{\alpha(y)}^{\alpha(y+h)} \int_{\alpha(x)}^{\alpha(y)} \int_{\mbb{R}^d} (\log(|\xi|))^2 |\xi|^u e^{i \xi \cdot (y-x)} e^{-t |\xi|^{\varrho}} \, d\xi \, du \, d\varrho.
	\end{align*}
	It follows from \cite[Theorem 4.7]{matters} and the H\"{o}lder continuity of $\alpha$ that there exists a constant $c_4>0$ such that \begin{equation*}
		|D_2(t,x,y)|
		\leq c_4 t |h|^{\gamma} |x-y|^{\gamma} \min \left\{ (1+|\log(t)|^2) t^{-(d+\|\alpha\|_{\infty})/\alpha_L}, \frac{1+|\log(|x-y|)|^2}{\min\{|x-y|^{d+\alpha_L},\|x-y|^{d+\|\alpha\|_{\infty}}\}} \right\}.
	\end{equation*}
	Now we can proceed exactly as in the first part of this step to conclude that \begin{equation*}
		\int_{\mbb{R}^d} |D_2(t,x,y)| \, dy 
		\leq c_5 |h|^{\gamma} (1+|\log(t)|^2) t^{-(\|\alpha\|_{\infty}-\gamma)/\alpha_L} 
		\leq c_5' |h|^{\gamma} t^{-1+\lambda_2}
	\end{equation*}
	for all $x \in \mbb{R}^d$, $|h| \leq 1$ and $t \in (0,1)$ and suitable constants $c_5,c_5',\lambda_2>0$; for the second estimate we used that $\gamma>\gamma_0 = \|\alpha\|_{\infty}-\alpha_L$. \par
	\textbf{Step 2:} For any $\eps \in (\gamma_0,\min\{\gamma,\alpha_L\})$ there exist constants $C>0$ and $\lambda>0$ such that \begin{equation}
		\int_{\mbb{R}^d} |F^{\circledast i}(t,x+h,y+h)-F^{\circledast i}(t,x,y)| \, dy \leq 2^{i} C^{i} \frac{\Gamma(\lambda)^{i}}{\Gamma(i \lambda)} t^{-1+i \lambda} |h|^{\gamma-\eps} \label{app-eq67}
	\end{equation}
	for all $i \in \mathbb{N}$, $x \in \mbb{R}^d$, $|h| \leq 1$ and $t \in (0,1)$. \par
	\emph{Indeed:} Fix $\epsilon \in (\gamma_0,\min\{\gamma,\alpha\})$. There exist constants $C>0$ and $\lambda>0$ such that \begin{align}
		\int_{\mbb{R}^d} |F^{\circledast i}(t,x,y)| \, dy \leq C^{i} \frac{\Gamma(\lambda)^{i}}{\Gamma(i \lambda)} t^{-1+i \lambda} \label{app-eq69}
	\end{align}
	for all $x \in \mbb{R}^d$, $i \geq 1$ and $t \in (0,1)$, cf.\ Appendix~\ref{pix}. Without loss of generality, we may assume that $C>0$ and $\lambda>0$ are such that \eqref{app-eq65} holds (otherwise we enlarge $C>0$ and choose $\lambda>0$ smaller). We claim that \eqref{app-eq67} holds for this choice of $C>0$ and $\lambda>0$ and prove this by induction. For $i=1$ the estimate is a direct consequence of \eqref{app-eq65}. Now assume that \eqref{app-eq67} holds for some $i \geq 1$. By the very definition of the time-space convolution, we have \begin{align*}
		(F \circledast F^{\circledast i})(t,x+h,y+h)
		&= \int_0^t \int_{\mbb{R}^d}  F(t-s,x+h,z) F^{\circledast i}(s,z,y+h) \, dz \, ds \\
		&= \int_0^t \int_{\mbb{R}^d}  F(t-s,x+h,z+h) F^{\circledast i}(s,z+h,y+h) \, dz \, ds
	\end{align*}
	and so \begin{equation*}
		|(F \circledast F^{\circledast i})(t,x+h,y+h)-(F \circledast F^{\circledast i})(t,x,y)| \leq I_1(t,x,y) + I_2(t,x,y)
	\end{equation*}
	for \begin{align*}
		I_1(t,x,y) &:= \int_0^t\int_{\mbb{R}^d} \left| (F(t-s,x+h,z+h)-F(t-s,x,z)) F^{\circledast i}(s,z+h,y+h) \right| \, dz \, ds \\
		I_2(t,x,y) &:= \int_0^t\int_{\mbb{R}^d} \left| (F^{\circledast i}(s,z+h,y+h)-F^{\circledast i}(s,z,y)) F(t-s,x,z) \right| \, dz \, ds.
	\end{align*}
	Using first \eqref{app-eq69} and then \eqref{app-eq65} we obtain \begin{align*}
		\int_{\mbb{R}^d} |I_1(t,x,y)| \, dy
		\leq C^{i+1} \frac{\Gamma(\lambda)^{i}}{\Gamma(i \lambda)} |h|^{\gamma-\eps} \int_0^t (t-s)^{-1+\lambda} s^{-1+i\lambda} \, ds
	\end{align*}
	for all $x \in \mbb{R}^d$, $|h| \leq 1$ and $t \in (0,1)$. In order to estimate the second term, we use \eqref{app-eq69} with $i=1$ and our induction hypothesis to find that \begin{equation*}
		\int_{\mbb{R}^d} |I_2(t,x,y)| \, dy
		\leq 2^i C^{i+1} \frac{\Gamma(\lambda)^{i}}{\Gamma(i \lambda)} |h|^{\gamma-\eps} \int_0^t (t-s)^{-1+\lambda} s^{-1+i\lambda} \, ds
	\end{equation*}
	for all $x \in \mbb{R}^d$, $|h| \leq 1$ and $t \in (0,1)$. Combining both estimates gives that $F^{\circledast (i+1)} = F \circledast F^{\circledast i}$ satisfies \begin{align*}
		\int_{\mbb{R}^d} |F^{\circledast (i+1)}(t,x+h,y+h)&-F^{\circledast (i+1)}(t,x,y)| \, dy \\
		&\leq (2C)^{i+1} \frac{\Gamma(\lambda)^{i}}{\Gamma(i \lambda)} |h|^{\gamma-\eps} \int_0^t (t-s)^{-1+\lambda} s^{-1+i\lambda} \, ds.
	\end{align*}
	Performing a change of variables, $s \rightsquigarrow tr$, and using the product formula for the Beta function, $B(u,v) = \Gamma(u) \Gamma(v)/\Gamma(u+v)$, we get \begin{align*}
		\int_0^t (t-s)^{-1+\lambda} s^{-1+i\lambda} \, ds
		= t^{-1+(i+1) \lambda} B(\lambda,i \lambda)
		= t^{-1+(i+1) \lambda} \frac{\Gamma(i) \Gamma(i \lambda)}{\Gamma((i+1) \lambda)}.
	\end{align*}
	Plugging this identity in the previous estimate shows that \eqref{app-eq67} holds for $i+1$, and this finishes the proof of Step 2. \par
	\textbf{Conclusion of the proof:} Fix $\eps \in (\gamma_0,\gamma \wedge \alpha_L)$. Since, by \eqref{app-eq63}, \begin{equation*}
		|\Phi(t,x+h,y+h)-\Phi(t,x,y)| \leq \sum_{i=1}^{\infty} |F^{\circledast i}(t,x+h,y+h)-F^{\circledast i}(t,x,y)| 
	\end{equation*}
	it follows from the monotone convergence theorem that \begin{equation*}
		\int_{\mbb{R}^d} |\Phi(t,x+h,y+h)-\Phi(t,x,y)| \, dy
		\leq  \sum_{i=1}^{\infty} \int_{\mbb{R}^d} |F^{\circledast i}(t,x+h,y+h)-F^{\circledast i}(t,x,y)|  \, dy,
	\end{equation*}
	and so, by Step 2, \begin{equation*}
		\int_{\mbb{R}^d} |\Phi(t,x+h,y+h)-\Phi(t,x,y)| \, dy
		\leq |h|^{\gamma-\eps} t^{-1+\lambda} \sum_{i \geq 1} 2^i C^i \frac{\Gamma(\lambda)^{i}}{\Gamma(i \lambda)}
	\end{equation*}
	for all $x \in \mbb{R}^d$, $|h| \leq 1$ and $t \in (0,1)$ and suitable constants $C>0$ and $\lambda>0$ (not depending on $x$, $h$, $t$). 	It is not difficult to see that the series on the right-hand side converges, see \cite[Lemma A.6]{matters} for details, and consequently we have proved the desired estimate.
\end{proof}

\subsection{Auxiliary result for the proof of Theorem~\ref{ex-21}} \label{iso-aux}

Let $(X_t)_{t \geq 0}$ be an isotropic stable-like process with symbol $q(x,\xi) = |\xi|^{\alpha(x)}$ for a H\"{o}lder continuous mapping $\alpha: \mbb{R}^d \to (0,2)$ with $\alpha_L := \inf_x \alpha(x)>0$. From Proposition~\ref{app-1} and Proposition~\ref{feller-7} we obtain immediately that any function $f$ in the Favard space $F_1$ associated with $(X_t)_{t \geq 0}$ satisfies the a-priori estimate \begin{equation}
\|f\|_{\mc{C}_b^{\kappa}(\mbb{R}^d)} \leq c (\|A_e f\|_{\infty}+\|f\|_{\infty}) \label{ex-eq21}
\end{equation}
for $\kappa \in (0,\alpha_L)$; in particular, $F_1 \subseteq \mc{C}_b^{\alpha_L-}(\mbb{R}^d)$. For the proof of Theorem~\ref{ex-21} we need the following auxiliary result which will allow us to derive an improved a priori estimate once we have shown that $f \in F_1$ is sufficiently regular on $\{x \in \mbb{R}^d; \alpha(x) \leq 1\}$. 

\begin{lem} \label{app-3}
	Let $(X_t)_{t \geq 0}$ be a Feller process with extended infinitesimal generator $(A_e,\mc{D}(A_e))$, Favard space $F_1$ and symbol $q(x,\xi)=|\xi|^{\alpha(x)}$ for a H\"{o}lder continuous mapping $\alpha: \mbb{R}^d \to (0,2)$ such that \begin{equation*}
	0 < \alpha_L := \inf_{x \in \mbb{R}^d} \alpha(x) \leq \sup_{x \in \mbb{R}^d} \alpha(x) < 2.
	\end{equation*}
	Let $f \in F_1$ be such that for any $\eps \in (0,\alpha_L)$ there exists a constant $M(\eps)>0$ such that  \begin{equation}
	|\Delta_h f(x)|=|f(x+h)-f(x)| \leq M(\eps) |h|^{\alpha(x)-\eps}, \qquad |h| \leq 1, \label{app-eq9}
	\end{equation}
	for any $x \in \{\alpha \leq 1\}$. Then there exists for any $\theta \in (0,1)$ a  constant $C=C(\alpha,\theta)$ such that \begin{equation*}
	|\Delta_h^2 f(x)| \leq C |h|^{1-\theta} (\|A_e f\|_{\infty}+\|f\|_{\infty}+M(\theta/12)), \qquad |h| \leq 1,
	\end{equation*}
	for any $x \in \{\alpha \geq 1\}$.
\end{lem}

\begin{proof}
	The idea of the proof is similar to the proof of Theorem~\ref{feller-9}. For fixed $0<\theta<\min\{\alpha_L,1/4\}$ define $\tilde{\alpha}(x) := \max\{1-3\theta,\alpha(x)\}$. By \cite[Theorem 5.2]{matters} there exists a Feller process $(Y_t)_{t \geq 0}$ with symbol $p(x,\xi) := |\xi|^{\tilde{\alpha}(x)}$ and the $(L,C_c^{\infty}(\mbb{R}^d))$-martingale problem for the generator $L$ of $(Y_t)_{t \geq 0}$ is well-posed. Since $\alpha$ is H\"{o}lder continuous, there exists $\delta>0$ such that \begin{equation}
	|x-z| \leq 2\delta \implies |\alpha(x)-\alpha(z)| \leq \theta. \label{app-eq11} 
	\end{equation}
	As usual, we denote by \begin{equation*}
	\tau_{\delta}^x := \inf\{t>0; |Y_t-x|>\delta\}
	\end{equation*}
	the exit time from the closed ball $\overline{B(x,\delta)}$. Pick $\kappa \in C_b^{\infty}(\mbb{R}^d)$, $0 \leq \kappa \leq 1$, such that $\kappa(x)=0$ for any $x \in \{\alpha \leq 1-2\theta\}$ and $\kappa(x)=1$ for $x \in \{\alpha \geq 1-\theta\}$, see Lemma~\ref{app-5} for the existence of such a mapping. \par 
	
	\textbf{Step 1:} We are going to show that for any $f \in F_1$ the product $v:=f \cdot \kappa$ is in the domain $\mc{D}(L_e)$ of the extended generator of $(Y_t)_{t \geq 0}$; we will use a similar reasoning as in the proof of Theorem~\ref{feller-9}, i.\,e.\ we will estimate \begin{equation*}
	\frac{1}{t} \sup_{x \in \mbb{R}^d} |\mbb{E}^x v(Y_{t \wedge \tau_{\delta}^x})-v(x)|. 
	\end{equation*}
	Clearly, \begin{align*}
	|\mbb{E}^x v(Y_{t \wedge \tau_{\delta}^x})-v(x)| \leq I_1(x)+I_2(x)+I_3(x)
	\end{align*}
	where \begin{align*}
	I_1(x) &:= |\kappa(x) \mbb{E}^x(f(Y_{t \wedge \tau_{\delta}^x})-f(x))| \\
	I_2(x) &:= |f(x) \mbb{E}^x(\kappa(Y_{t \wedge \tau_{\delta}^x})-\kappa(x))| \\
	I_3(x) &:= \left|\mbb{E}^x \big( (f(Y_{t \wedge \tau_{\delta}^x})-f(x))(\kappa(Y_{t \wedge \tau_{\delta}^x})-\kappa(x)) \big) \right|.
	\end{align*}
	We are going to estimate the terms separately; we start with $I_1$. If $x \in \{\alpha \geq 1-2\theta\}$, then it follows from \eqref{app-eq11} that $B(x,2\delta) \subseteq \{\alpha \geq 1-3\theta\}$ and therefore \begin{equation}
		q(z,\xi) = |\xi|^{\alpha(z)} = |\xi|^{\tilde{\alpha}(z)} = p(z,\xi) \fa z \in B(x,2\delta), \,\xi \in \mbb{R}^d. \label{app-eq13}
	\end{equation}
	Applying Lemma~\ref{feller-13} we find that \begin{equation*}
	I_1(x) = |\kappa(x) \mbb{E}^x(f(X_{t \wedge \tau_{\delta}^x(X)})-f(x))|
	\end{equation*}
	where $\tau_{\delta}^x(X)$ is the exit time of $(X_t)_{t \geq 0}$ from $\overline{B(x,\delta)}$.  As $f \in F_1$ an application of Dynkin's formula \eqref{gen-eq6} gives \begin{equation*}
	I_1(x) \leq t \|A_e f\|_{\infty}.
	\end{equation*}
	If $x \in \{\alpha<1-2\theta\}$, then $\kappa(x)=0$ by the very definition of $\kappa$, and so $I_1(x)=0$. Hence, \begin{equation*}
	\sup_{x \in \mbb{R}^d} I_1(x) \leq t \|A_e f\|_{\infty}.
	\end{equation*}
	For $I_2$ we note that $\kappa \in C_b^{\infty}(\mbb{R}^d) \subseteq \mc{D}(L)$, and therefore an application of the (classical) Dynkin formula gives \begin{equation*}
	\sup_{x \in \mbb{R}^d} I_2(x) 
	\leq t \|f\|_{\infty} \|L\kappa\|_{\infty}.
	\end{equation*}
	To estimate $I_3$ we consider two cases separately. If $x \in \{\alpha \leq 1\}$, then it follows from our assumption on the regularity of $f$, cf.\ \eqref{app-eq9}, and the Lipschitz continuity of $\kappa$ that \begin{equation*}
	|f(Y_{t \wedge \tau_{\delta}^x})-f(x)| \cdot |\kappa(Y_{t \wedge \tau_{\delta}^x})-\kappa(x)|
	\leq 4 (\|f\|_{\infty}+M(\theta/3)) \|\kappa\|_{C_b^1(\mbb{R}^d)} \min\{|Y_{t \wedge \tau_{\delta}^x}-x|^{\alpha(x)-\theta/3+1},1\}.
	\end{equation*}
	Applying Lemma~\ref{feller-15} we find that there exists a constant $c_2=c_2(\alpha_L,\|\alpha\|_{\infty})>0$ such that \begin{align}
	I_3(x)
	\leq c_2 (\|f\|_{\infty}+M(\theta/3))  \|\kappa\|_{C_b^1(\mbb{R}^d)} \sup_{|z-x| \leq \delta} \int_{y \neq 0} \min\{1,|y|^{\alpha(x)-\theta/3+1}\} \frac{1}{|y|^{d+\tilde{\alpha}(z)}} \, dy. \label{app-eq15} \tag{$\star$}
	\end{align}
	For $x \in \mbb{R}^d$ with $\alpha(x) \leq 1-2\theta$ we note that it follows from the definition of $\tilde{\alpha}$ that $\tilde{\alpha}(z) \geq 1-3\theta$ for all $z \in \mbb{R}^d$, and so \begin{align*}
	\sup_{x \in \{\alpha \leq 1-2\theta\}} I_3(x)
	\leq c_2 (\|f\|_{\infty}+M(\theta/3)) \left( \int_{|y| \leq 1} |y|^{-d+2\theta/3} \, dy  + \int_{|y|>1} |y|^{-d-1+3\theta} \, dy \right)<\infty.
	\end{align*}
	If $1-2\theta \leq \alpha(x) \leq 1$, then $\alpha(z) = \tilde{\alpha}(z)$ for all $|z-x| \leq \delta$; using \eqref{app-eq11} we find from \eqref{app-eq15} that \begin{equation*}
	\sup_{x \in \{1-2\theta \leq \alpha \leq 1\}} I_3(x)
	\leq c_2 (\|f\|_{\infty}+M(\theta/3)) \left( \int_{|y| \leq 1} |y|^{-d+1-4\theta/3} \, dy + \int_{|y|>1} |y|^{-d-\alpha_L} \, dy \right)<\infty.
	\end{equation*}
	Finally, if $x \in \{\alpha>1\}$, then $\overline{B(x,\delta)} \subseteq \{\alpha \geq 1-\theta\}$, and therefore $\kappa(z)=1$ for any $|z-x| \leq \delta$; hence, \begin{equation*}
	|f(Y_{t \wedge \tau_{\delta}^x})-f(x)| \cdot |\kappa(Y_{t \wedge \tau_{\delta}^x})-\kappa(x)|
	\leq 2\|f\|_{\infty} \I_{\{\tau_{\delta}^x \leq t\}}
	\end{equation*} 
	which implies \begin{equation*}
	I_3(x) \leq 2 \|f\|_{\infty} \mbb{P}^x(\tau_{\delta}^x \leq t).
	\end{equation*}
	Applying the maximal inequality \eqref{max} we get \begin{align*}
	I_3(x) \leq c_3 \|f\|_{\infty} t \sup_{|z-x| \leq \delta} \sup_{|\xi| \leq \delta^{-1}} |p(z,\xi)|
	\end{align*}
	for some absolute constant $c_3>0$. As $|p(z,\xi)| \leq |\xi|^2$ for all $\xi \in \mbb{R}^d$ this shows that \begin{align*}
	\sup_{x \in \{\alpha>1\}} I_3(x) 
	&\leq c_3 \|f\|_{\infty} t \delta^{-2}.
	\end{align*}
	Combining the estimates we conclude that \begin{equation*}
	\sup_{t>0} \frac{1}{t} \sup_{x \in \mbb{R}^d} |\mbb{E}^x v(Y_{t \wedge \tau_{\delta}^x})-v(x)|
	\leq c_4 (\|f\|_{\infty}+\|A_e f\|_{\infty}+M(\theta/3)) 
	\end{equation*}
	for some constant $c_4 = c_4(\theta,\delta,\alpha_L,\|\alpha\|_{\infty},\|L\kappa\|_{\infty})$. \par
	\textbf{Step 2:} Applying Corollary~\ref{gen-5} we find that $v=f \cdot \kappa$ is in the Favard space $F_1^Y$ of order $1$ associated with $(Y_t)_{t \geq 0}$ and \begin{equation*}
	\|L_e(f \cdot \kappa)\|_{\infty} \leq c_5 (\|f\|_{\infty}+\|A_e f\|_{\infty}+M(\theta/3)). 
	\end{equation*} 
	Since Proposition~\ref{app-1} shows that the semigroup $(T_t)_{t \geq 0}$ associated with $(Y_t)_{t \geq 0}$ satisfies the H\"older estimate \begin{equation*}
	 \|T_t u\|_{\mc{C}_b^{1-4\theta}(\mbb{R}^d)} \leq c_6 \|u\|_{\infty} t^{-(1-4\theta)/(1-3\theta)}, \qquad t \in (0,1], \,u \in \mc{B}_b(\mbb{R}^d)
	\end{equation*}
	for $c_6 = c_6(\alpha,\theta)>0$, it follows from Proposition~\ref{feller-7} that \begin{equation*}
	\|f \cdot \kappa\|_{\mc{C}_b^{1-4\theta}(\mbb{R}^d)} \leq c_7 (\|f\|_{\infty}+\|A_e f\|_{\infty}+M(\theta/3)) 
	\end{equation*}
	for some constant $c_7>0$ which does not depend on $f$. Finally, we note that for any $x \in \{\alpha \geq 1\}$ we have $\kappa(z)=1$ for $z \in \overline{B(x,\delta)}$, and therefore it follows for all $|h| \leq \delta/2$ that \begin{align*}
	|f(x+2h)-2f(x+h)+f(x)| 
	&= |\kappa(x+2h) f(x+2h) -2 \kappa(x+h) f(x+h)+ \kappa(x) +f(x)| \\
	&\leq c_7 |h|^{1-4\theta} (\|f\|_{\infty}+\|A_e f\|_{\infty}+M(\theta/3)). \qedhere
	\end{align*}
\end{proof}

\subsection{Proof of Theorem~\ref{ex-21} and Corollary~\ref{ex-23} }
\label{iso-proofs}

\begin{proof}[Proof of Theorem~\ref{ex-21}] 
	Fix $\eps \in (0,\alpha_L)$. Since $\alpha$ is H\"{o}lder continuous there exists $\delta>0$ such that \begin{equation}
	|\alpha(x)-\alpha(y)| \leq \frac{\eps}{2} \fa |x-y| \leq  4 \delta. \label{ex-st3} \tag{$\star$}
	\end{equation}
	Moreover, $\|\alpha\|_{\infty}<2$ implies that we can choose $\theta \in (0,\alpha_L)$ such that $\alpha(x)<2-\theta$ for all $x \in \mbb{R}^d$; without loss of generality, we may assume that $\eps \leq \theta$. We divide the proof in two steps. In the first part, we will establish the H\"{o}lder regularity of functions $f \in F_1$ at points $x \in \mbb{R}^d$ such that $\alpha(x)\leq 1+\alpha_L-\theta$. In the second part, we will consider the remaining points.  \par \medskip
	
	\textbf{Step 1:} There exists a constant $C_1>0$ such that \begin{equation}
	|\Delta_h^2 f(x)| \leq C_1 |h|^{\alpha(x)-\eps} (\|A_e f\|_{\infty}+\|f\|_{\infty}) \fa f \in F_1, \, |h| \leq \delta, \, x \in \{\alpha \leq \alpha_L+ 1-\theta\}.   \label{ex-eq23}
	\end{equation}
	\emph{Indeed:} Fix $x \in \mbb{R}^d$ such that $\alpha(x) \leq \alpha_L+1-\theta$ and define \begin{equation*}
		\alpha^x(z) := \max\{\alpha(z),\alpha(x)-\eps/2\}, \qquad z \in \mbb{R}^d.
	\end{equation*} 
	It is not difficult to see that $\|\alpha^x\|_{\mc{C}_b^{\gamma}(\mbb{R}^d)} \leq \|\alpha\|_{\mc{C}_b^{\gamma}(\mbb{R}^d)}$ and, moreover, \begin{equation*}
		\alpha^x_L := \inf_{z \in \mbb{R}^d} \alpha^x(z) \geq \alpha(x)-\frac{\eps}{2}>0.
	\end{equation*}
	It follows from \cite[Theorem 5.2]{matters} that there exists a Feller process with symbol $p(z,\xi) := |\xi|^{\alpha^x(z)}$ and that the $(L,C_c^{\infty}(\mbb{R}^d))$-martingale problem for the generator $L$ of $(Y_t)_{t \geq 0}$ is well-posed. Note that, by \eqref{ex-st3}, $\alpha^x(z)=\alpha(z)$ for $|z-x| \leq 4\delta$ and therefore \begin{equation*}
		q(z,\xi) = |\xi|^{\alpha(z)} = |\xi|^{\alpha^x(z)} = p^{(x)}(z,\xi)  \fa \xi \in \mbb{R}^d, \,|z-x| \leq 4 \delta.
	\end{equation*}
	Moreover, an application of Lemma~\ref{app-3} shows that there exists a constant $c_1=c_1(\eps,\alpha)$ such that the semigroup $(T_t)_{t \geq 0}$ associated with $(Y_t)_{t \geq 0}$ satisfies \begin{equation}
		\|T_t u\|_{\mc{C}_b^{\alpha(x)-\eps}(\mbb{R}^d)} \leq c_1 \|u\|_{\infty}   t^{-(\alpha(x)-\eps)/(\alpha(x)-\eps/2)} \label{ex-eq25}
	\end{equation}
	for any $u \in \mc{B}_b(\mbb{R}^d)$ and $t \in (0,1]$. This shows that the conditions \eqref{C1}-\eqref{C4} in Theorem~\ref{feller-9} are satisfied. By \eqref{ex-eq21} it follows from Theorem~\ref{feller-9} (with $\varrho(x) := \alpha_L-\theta/4$) that there exists a constant $c_2=c_2(\eps,\alpha)$ such that \begin{equation*}
		|\Delta_h^2 f(x)| \leq c_2 K(x) |h|^{\alpha(x)-\eps} (\|A_e f\|_{\infty}+\|f\|_{\infty}), \qquad f \in F_1, \,|h| \leq \delta
	\end{equation*}
	where \begin{equation*}
		K(x) := \sup_{z \in \mbb{R}^d} \int_{y \neq 0} \min\{1,|y|^2\} \frac{1}{|y|^{d+\alpha^x(z)}} dy +  \sup_{|z-x| \leq 4\delta} \int_{y \neq 0} \min\{1,|y|^{1+\alpha_L-\theta/4}\} \frac{1}{|y|^{d+\alpha^x(z)}} \, dy;
	\end{equation*}
	if we can show that $K:=\sup_{x \in \{\alpha \leq \alpha_L+1-\theta\}} K(x) < \infty$, this gives \eqref{ex-eq23}. To this end, we note that $\eps \leq \theta$ and \eqref{ex-st3} imply \begin{equation*}
		\alpha^x(z) = \alpha(z) \leq \alpha(x)+\frac{\eps}{2} \leq (\alpha_L+1-\theta) + \frac{\theta}{2} = \alpha_L+1-\frac{\theta}{2} \fa |z-x| \leq 4 \delta
	\end{equation*}
	and so \begin{equation*}
		K \leq  \sup_{\beta \in [\alpha_L,\|\alpha\|_{\infty}]} \int_{y \neq 0} \min\{1,|y|^2\} \frac{1}{|y|^{d+\beta}} \, dy + \sup_{\beta \in [\alpha_L,\alpha_L+1-\theta/2]} \int_{y \neq 0}  \min\{1,|y|^{1+\alpha_L-\theta/4}\} \frac{1}{|y|^{d+\beta}} \, dy < \infty.
	\end{equation*}
	\textbf{Step 2:} There exists $C_2>0$ such that \begin{equation*}
		|\Delta_h^2 f(x)| \leq C_2 |h|^{\alpha(x)-\eps} (\|A_e f\|_{\infty}+\|f\|_{\infty}) \fa f \in F_1, \, |h| \leq \delta, \, x \in \{\alpha \geq \alpha_L+1-\theta\}. 
	\end{equation*}
	\emph{Indeed:} It follows from Lemma~\ref{app-3} and Step 1 that there exists a constant $c_3>0$ such that \begin{equation}
	|\Delta_h^2 f(x)| \leq c_3  |h|^{1-\theta/2} (\|A_e f\|_{\infty}+\|f\|_{\infty}), \qquad |h| \leq 1, \label{ex-eq27}
	\end{equation}
	for any $f \in F_1$ and $x \in \{\alpha \geq 1\}$. Thanks to this improved a priori-estimate for $f \in F_1$ we can use a very similar reasoning as in the first part of the proof to deduce the desired estimate. If we set $\alpha^x(z) := \max\{\alpha(z),\alpha(x)-\eps/2\}$ for fixed $x \in \{\alpha \geq 1+\alpha_L-\theta\}$, then it follows exactly as in Step 1 that the Feller process $(Y_t)_{t \geq 0}$ with symbol $p(z,\xi) := |\xi|^{\alpha^x(z)}$ satisfies \eqref{C1}-\eqref{C4} in Theorem~\ref{feller-9}; in particular, \eqref{ex-eq25} holds for the associated semigroup $(T_t)_{t \geq 0}$. Because of \eqref{ex-eq27} we may apply Theorem~\ref{feller-9} with $\varrho(x) := 1-\theta/2$ to obtain \begin{align*}
		|\Delta_h^2 f(x)| \leq c_4 K(x) |h|^{\alpha(x)-\eps} (\|A_e f\|_{\infty}+\|f\|_{\infty}), \qquad f \in F_1
	\end{align*}
	for a constant $c_4$ (not depending on $f$ and $x$) and \begin{equation*}
		K(x) := \sup_{z \in \mbb{R}^d} \int_{y \neq 0} \min\{1,|y|^2\} \frac{1}{|y|^{d+\alpha^x(z)}} \, dy + \sup_{|z-x| \leq 4 \delta} \int_{y \neq 0} \min\{1,|y|^{2-\theta/2}\} \frac{1}{|y|^{d+\alpha^x(z)}} \, dy.
	\end{equation*}
	By our choice of $\theta$, we have $\alpha_L \leq \alpha^x(z) \leq \|\alpha\|_{\infty} < 2 - \theta$, and so \begin{equation*}
		\sup_{x \in \{\alpha \geq 1+\alpha_L-\theta\}} K(x)
		\leq 2\sup_{\beta \in [\alpha_L,\|\alpha\|_{\infty}]} \int_{y \neq 0} \min\{1,|y|^2\} \frac{1}{|y|^{d+\beta}} \, dy  + \int_{|y| \leq 1} |y|^{-d+\theta/2} \, dy < \infty. \qedhere
	\end{equation*}
\end{proof}

\begin{proof}[Proof of Corollary~\ref{ex-23}]
	We are going to apply Theorem~\ref{feller-12} to prove the assertion. To this end, we first need to construct for each $x \in \mbb{R}^d$ a Feller process $(Y_t^{(x)})_{t \geq 0}$ which satisfies \eqref{C1}-\eqref{C4} from Theorem~\ref{feller-9} as well as \eqref{S1}-\eqref{S4} from Theorem~\ref{feller-12}. Recall that $\alpha_L =\inf_x \alpha(x)>0$ and that $\gamma \in (0,1)$ is the H\"{o}lder exponent of $\alpha$. \par
	Fix $\eps \in (0,\alpha_L \wedge \gamma)$ and $x \in \mbb{R}^d$. Since $\alpha$ is H\"{o}lder continuous there exists $\delta>0$ such that \begin{equation}
		|\alpha(z+y)-\alpha(z)| \leq \frac{\eps}{4} \fa z \in \mbb{R}^d,\, |h| \leq \delta. \label{ex-eq28} \tag{$\star$}
	\end{equation}
	If we define \begin{equation*}
		\alpha^x(z) := (\alpha(x)-\eps/4) \vee \alpha(z) \wedge (\alpha(x)+\eps/4), \qquad z \in \mbb{R}^d,
	\end{equation*}
	then it follows from \cite[Theorem 5.2]{matters} that there exists a Feller process $(Y_t^{(x)})_{t \geq 0}$ with symbol $p^{(x)}(z,\xi) := |\xi|^{\alpha^x(z)}$ such that the martingale problem for its generator is well-posed. Moreover, by our choice of $\delta$, \begin{equation*}
		q(z,\xi) = |\xi|^{\alpha(z)} = |\xi|^{\alpha^x(z)} = p^{(x)}(z,\xi) \fa \xi \in \mbb{R}^d,\, |z-x| \leq 4 \delta,
	\end{equation*}
	and so \eqref{C1} and \eqref{C2} from Theorem~\ref{feller-9} hold. Applying Proposition~\ref{app-1} and Proposition~\ref{app-25}, it follows that the semigroup $(T_t^{(x)})_{t \geq 0}$ associated with $(Y_t^{(x)})_{t \geq 0}$ satisfies \begin{equation*}
		\|T_t^{(x)} u\|_{\mc{C}_b^{\kappa(x)}(\mbb{R}^d)} \leq c_1 \|u\|_{\infty} t^{-\beta(x)}, \quad u \in \mc{B}_b(\mbb{R}^d), \,t \in (0,1),
	\end{equation*}
	and \begin{equation*}
		\|T_t^{(x)} u\|_{\mc{C}_b^{\kappa(x)+\lambda}(\mbb{R}^d)} \leq c_1 \|u\|_{\mc{C}_b^{\lambda}(\mbb{R}^d)} t^{-\beta(x)}, \qquad u \in \mc{C}_b^{\lambda}(\mbb{R}^d), \,t \in (0,1),
	\end{equation*}
	for any $\lambda \leq \Lambda:=\gamma$ where $c_1>0$ is some constant (not depending on $u$, $t$, $x$) and \begin{equation*}
		\kappa(x) := \alpha(x)-\eps \qquad \quad \beta(x) := \frac{\alpha(x)-2\eps}{\alpha(x)-\eps/4}.
	\end{equation*}
	Consequently, we have established \eqref{C4} and \eqref{S3}. Since $\kappa$ is clearly uniformly continuous and bounded away from zero, we get immediately that \eqref{S5} holds. Moreover, as $\alpha$ is bounded away from zero and from two, it follows easily that \eqref{S1} and \eqref{S4} hold with $\alpha^{(x)}(z) := \alpha^x(z)$. Finally, we note that the H\"{o}lder condition \eqref{S2} on the symbol $p^{(x)}$ is a consequence of the H\"{o}lder continuity of $\alpha$, see Lemma~\ref{aux-1} below for details. \par
	We are now ready to apply Theorem~\ref{feller-12}. Let $f \in \mc{D}(A)$ be such that $Af= g \in \mc{C}_b^{\lambda}(\mbb{R}^d)$ for some $\lambda>0$. Without loss of generality, we may assume that $\lambda \leq \gamma$. Since $(X_t)_{t \geq 0}$ satisfies the assumptions of Theorem~\ref{ex-21}, it follows that $f \in \mc{C}_b^{\varrho(\cdot)}(\mbb{R}^d)$ for $\varrho(x) := \alpha(x)-\eps/4$ and, moreover, \begin{equation}
		\|f\|_{\mc{C}_b^{\varrho(\cdot)}(\mbb{R}^d)} \leq C_{\eps} (\|Af\|_{\infty} + \|f\|_{\infty}). \label{ex-eq29}
	\end{equation}
	Furthermore, by our choice of $\delta$, cf.\ \eqref{ex-eq28}, we find that \begin{equation*}
		\sigma := \inf_{x \in \mbb{R}^d} \inf_{|z-x| \leq 4\delta} (1+\varrho(x)-\alpha^x(z))
	\end{equation*}
	satisfies $\sigma \geq 1-\eps/4$. Applying Theorem~\ref{feller-12} we conclude that \begin{equation*}
		f \in \mc{C}_b^{\kappa(\cdot)+\min\{\gamma,\lambda,1-\eps/4\}-\eps/4}(\mbb{R}^d) \subseteq \mc{C}_b^{\alpha(\cdot)+\min\{\gamma,\lambda\}-2\eps}(\mbb{R}^d)
	\end{equation*}
	and  \begin{align*}
		\|f\|_{\mc{C}_b^{\alpha(\cdot)+\min\{\gamma,\lambda\}-2\eps}(\mbb{R}^d)} 
		&\leq C_{\eps}' (\|Af\|_{\mc{C}_b^{\lambda}(\mbb{R}^d)} + \|f\|_{\mc{C}_b^{\varrho(\cdot)}(\mbb{R}^d)}) 
		\leq C_{\eps}'' (\|Af\|_{\mc{C}_b^{\lambda}(\mbb{R}^d)} + \|f\|_{\infty})
	\end{align*}
	where we used \eqref{ex-eq29} for the last inequality.
\end{proof}

\begin{lem} \label{aux-1}
	For fixed $\alpha \in (0,2)$ denote by $\nu_{\alpha}$ the L\'evy measure of the isotropic $\alpha$-stable L\'evy process, i.\,e.\ \begin{equation}
		|\xi|^{\alpha} = \int_{y \neq 0} (1-\cos(y \cdot \xi)) \, \nu_{\alpha}(dy), \qquad \xi \in \mbb{R}^d. \label{aux-eq1}
	\end{equation}
	Let $\beta: \mbb{R}^d \to (0,2)$ be such that $\beta \in C_b^{\gamma}(\mbb{R}^d)$ for some $\gamma \in (0,1]$ and \begin{equation*}
		0 < \beta_L := \inf_{z \in \mbb{R}^d} \beta(z) \leq \sup_{z \in \mbb{R}^d} \beta(z)<2.
	\end{equation*}
	If $u: \mbb{R}^d \to \mbb{R}$ is a measurable mapping such that \begin{equation}
		|u(y)| \leq M \min\{|y|^{\beta(z)+r},1\}, \qquad y \in \mbb{R}^d, \label{aux-eq3}
	\end{equation}
	for some $z \in \mbb{R}^d$, $r>0$ and $M>0$, then there exist constants $K>0$ and $H>0$ (not depending on $u$ and $z$) such that \begin{equation*}
		\left| \int u(y) \, \nu_{\beta(z)}(dy) - \int u(y) \, \nu_{\beta(z+h)}(dy) \right| \leq M K |h|^{\gamma} \fa |h| \leq H.
	\end{equation*}
\end{lem}

\begin{proof}
	It is well known that $\nu_{\alpha}(dy) = c(\alpha) |y|^{-d-\alpha}$ where $c(\alpha)$ is a normalizing constant such that \eqref{aux-eq1} holds. Noting that, by the rotational invariance of $\xi \mapsto |\xi|^{\alpha}$,  \begin{equation*}
		|\xi|^{\alpha} 
		= c(\alpha) \int_{y \neq 0} (1-\cos(y_1 |\xi|)) \frac{1}{|y|^{d+\alpha}} \, dy 
		=|\xi|^{\alpha} c(\alpha) \int_{y \neq 0} (1-\cos(y_1)) \frac{1}{|y|^{d+\alpha}} \, dy
	\end{equation*}
	for all $\xi \in \mbb{R}^d$, we find $c(\alpha)=1/h(\alpha)$ for \begin{equation*}
		h(\alpha) := \int_{y \neq 0} (1-\cos(y_1)) \frac{1}{|y|^{d+\alpha}} \, dy.
	\end{equation*}
	Using that \begin{align}
		\left| \frac{1}{r^{d+\alpha}}- \frac{1}{r^{d+\tilde{\alpha}}}\right|
		= \frac{1}{r^{2d+\alpha+\beta}} |r^{d+\tilde{\alpha}}-r^{d+\alpha}| 
		\leq |\log(r)| r^{-d} \max\{r^{-\alpha},r^{-\tilde{\alpha}}\} |\alpha-\tilde{\alpha}| \label{aux-eq5}
	\end{align}
	for any $r>0$ and $\alpha,\tilde{\alpha} \in I:=[\beta_L,\|\beta\|_{\infty}] \subseteq (0,2)$, it follows easily that \begin{equation*}
		|h(\alpha)-h(\tilde{\alpha})| \leq C_1 |\alpha-\tilde{\alpha}|, \qquad \alpha,\tilde{\alpha} \in I
	\end{equation*}
	for some constant $C_1>0$. As $\inf_{\alpha \in I} h(\alpha)>0$ this implies that $c(\alpha)=1/h(\alpha)$ satisfies \begin{equation}
		|c(\alpha)-c(\tilde{\alpha})| \leq C_2  |\alpha-\tilde{\alpha}|, \qquad \alpha,\tilde{\alpha} \in I \label{aux-eq7}
	\end{equation}
	for some constant $C_2>0$. \par
	Now let $u: \mbb{R}^d \to \mbb{R}$ be a measurable mapping such that \eqref{aux-eq3} holds for some $z \in \mbb{R}^d$, $M>0$ and $r>0$. Since $\nu_{\alpha}(dy) = c(\alpha) |y|^{-d-\alpha} \, dy$ we have \begin{equation*}
		\left| \int u(y) \, \nu_{\beta(z)}(dy) - \int u(y) \, \nu_{\beta(z+h)}(dy) \right| \leq I_1+I_2
	\end{equation*}
	where \begin{align*}
		I_1 &:= |c(\beta(z))-c(\beta(z+h))| \int_{y \in \mbb{R}^d} |u(y)| \frac{1}{|y|^{d+\beta(z)}} \, dy \\
		I_2 &:= c(\beta(z+h)) \int_{y \neq 0} |u(y)| \left| \frac{1}{|y|^{d+\beta(z)}}-\frac{1}{|y|^{d+\beta(z+h)}} \right| \, dy.
	\end{align*}
	By the first part of the proof, cf.\ \eqref{aux-eq7}, and by \eqref{aux-eq3}, we find \begin{equation*}
		I_1
		\leq C_2 M |\beta(z)-\beta(z+h)| \int_{y \in \mbb{R}^d} \min\{|y|^{\beta(z)+r},1\} \frac{1}{|y|^{d+\beta(z)}} \, dy
	\end{equation*}
	and so \begin{equation*}
		I_1 \leq C_2 M |h|^{\gamma} \|\beta\|_{C_b^{\gamma}(\mbb{R}^d)} \sup_{\alpha \in I} \int_{y \neq 0} \min\{|y|^{\alpha+r},1\} |y|^{d-\alpha} \, dy
		=: C_{3} M |h|^{\gamma}
	\end{equation*}
	for all $h \in \mbb{R}^d$. To estimate $I_2$ we choose $H>0$ such that \begin{equation*}
		|\beta(x)-\beta(x+h)| \leq \frac{\min\{r,\beta_L\}}{2} \fa x \in \mbb{R}^d, |h| \leq H.
	\end{equation*}
	By \eqref{aux-eq3} and \eqref{aux-eq5}, we get \begin{align*}
		I_2
		\leq M |\beta(z)-\beta(z+h)| \sup_{\alpha \in I} c(\alpha) \int_{y \neq 0} \min\{|y|^{\beta(z)+r},1\} |\log(|y|)| \frac{\max\{|y|^{-\beta(z)},|y|^{-\beta(z+h)}\}}{|y|^{d}} \, dy
	\end{align*}
	for all $|h| \leq H$. By our choice of $H$, it holds that \begin{equation*}
		\frac{\beta(z)}{2} \leq \beta(z)-\frac{\beta_L}{2} \leq \beta(z+h) \leq \beta(z) + \frac{r}{2} \fa |h| \leq H,
	\end{equation*}
	and therefore \begin{align*}
		I_2
		&\leq M |\beta(z)-\beta(z+h)| \sup_{\alpha \in I} c(\alpha)  \left( \int_{|y| \leq 1} |y|^{-d+r/2} |\log(|y|)| \ dy + \int_{|y|>1} |y|^{-d-\beta(z)/2} \log(|y|) \,dy \right) \\
		&\leq C_4 M |h|^{\gamma} 
	\end{align*}
	for all $|h|  \leq H$ and \begin{equation*}
		C_4 := \|\beta\|_{C_b^{\gamma}(\mbb{R}^d)} \sup_{\alpha \in I} c(\alpha)  \left( \int_{|y| \leq 1} |y|^{-d+r/2} |\log(|y|)| \ dy + \sup_{\alpha \in I} \int_{|y|>1} |y|^{-d-\alpha/2} \log(|y|) \,dy \right)<\infty. \qedhere
	\end{equation*}

\end{proof}

\appendix

\section{Extended generator} \label{app-gen}

In this section, we collect some material on the extended generator of a Feller process; in particular, we present the proofs of Theorem~\ref{gen-3} and Corollary~\ref{gen-5}. The extended infinitesimal generator was originally introduced by Kunita \cite{kunita69} and was studied quite intensively in the 80s, e.\,g.\ by Airault \& F\"{o}llmer \cite{foellmer74}, Bouleau \cite{bouleau81}, Hirsch \cite{hirsch84}, Meyer \cite{meyer76} and Mokobodzki \cite{moko78}. Recall the following definition, cf.\ Section~\ref{def}.

\begin{defn} \label{app-gen-1}
	Let  $(X_t)_{t \geq 0}$ be a Feller process with $\mu$-potential operators $(R_{\lambda})_{\lambda>0}$. A function $f$ is in the domain $\mc{D}(A_e)$ of the extended generator and $g=A_e f$ if \begin{enumerate}
		\item\label{app-gen-1-i} $f \in \mc{B}_b(\mbb{R}^d)$ and $g$ is a measurable function such that $\|R_{\lambda}(|g|)\|_{\infty}< \infty$ for some (all) $\lambda>0$,
		\item\label{app-gen-1-iii} $f = R_{\lambda}(\lambda f-g)$ for all $\lambda>0$.
	\end{enumerate}
\end{defn}
Condition \ref{app-gen-1}\eqref{app-gen-1-iii} may be replaced by
\begin{enumerate}[label*=\upshape (\roman*'),ref=\upshape \roman*'] \setcounter{enumi}{1} 
	\item\label{gen-1-iv} $M_t := f(X_t)-f(X_0) - \int_0^t g(X_s) \, ds$, $t \geq 0$, is a local $\mbb{P}^x$-martingale for any $x \in \mbb{R}^d$;
\end{enumerate}
cf.\ Meyer \cite{meyer76} or Bouleau \cite{bouleau81}. Moreover, it was shown in \cite{foellmer74} that the extended generator can also be defined in terms of pointwise limits \begin{equation*}
	\lim_{t \to 0} t^{-1} (\mbb{E}^x f(X_t)-f(x)),
\end{equation*}
see also Corollary~\ref{app-gen-5} below. The domain $\mc{D}(A_e)$ is, in general, quite large; this is indicated by the fact that it is possible to show under relatively weak assumptions (e.\,g.\ $C_c^{\infty}(\mbb{R}^d) \subseteq \mc{D}(A_e)$) that $\mc{D}(A_e)$ is closed under multiplication, cf.\ \cite[pp.~144]{meyer76} or \cite[Theorem 4.3.6]{bouleau91}. There is a close connection between the extended generator and the carr\'e du champ operator, cf.\ \cite[Section 4.3]{bouleau91} or \cite{meyer4}. The following statement is essentially due to Airault \& F\"ollmer \cite{foellmer74}.

\begin{thm} \label{app-gen-3}
	Let $(X_t)_{t \geq 0}$ be a Feller process with semigroup $(P_t)_{t \geq 0}$ and extended generator $(A_e,\mc{D}(A_e))$. The associated Favard space $F_1$ of order $1$, cf.\ \eqref{fav}, satisfies \begin{equation*}
		F_1 = \{f \in \mc{D}(A_e); \|A_e f\|_{\infty}< \infty\}.
	\end{equation*}
	If $f \in F_1$ then \begin{equation*}
		K(f) := \sup_{t \in (0,1)} \frac{1}{t} \|P_t f-f\|_{\infty} = \|A_e f\|_{\infty}
	\end{equation*}
	and, moreover, Dynkin's formula \begin{equation}
		\mbb{E}^x f(X_{\tau})-f(x) = \mbb{E}^x \left( \int_0^{\tau} A_e f(X_s) \, ds \right) \label{app-gen-eq6}
	\end{equation}
	holds for any $x \in \mbb{R}^d$ and any stopping time $\tau$ such that $\mbb{E}^x \tau<\infty$.
 \end{thm}

\begin{proof}
	Denote by $(R_{\lambda})_{\lambda>0}$ the $\lambda$-potential operators of $(X_t)_{t \geq 0}$ and set \begin{equation*}
		\mc{D} := \{f \in \mc{B}_b(\mbb{R}^d); \|A_e f\|_{\infty}< \infty\}.
	\end{equation*}
	 First we prove $F_1 \subseteq \mc{D}$. Let $f \in F_1$. Airault \& F\"ollmer \cite[p.~320--322]{foellmer74} showed that the limit $g(x) = \lim_{t \to 0}t^{-1} (P_t f(x)-f(x))$ exists outside a set of potential zero and that \begin{equation*}
		M_t := f(X_t)-f(X_0) - \int_0^t g(X_s) \, ds, \qquad t \geq 0,
	\end{equation*}
	is a $\mbb{P}^x$-martingale for any $x \in \mbb{R}^d$; we set $g=0$ on the set of potential zero where the limit does not exist. Clearly, $\|g\|_{\infty} \leq K(f)<\infty$, and therefore it is obvious that $R_{\lambda}(|g|)$ is bounded for any $\lambda>0$. It remains to check \ref{app-gen-1}\eqref{app-gen-1-iii}. Since the martingale $(M_t)_{t \geq 0}$ has constant expectation, we have $P_t f = f + \int_0^t P_s g \, ds$, and so \begin{align*}
		\lambda \int_{(0,\infty)} e^{-\lambda t} P_t f(x) \, dt
		&= \lambda \int_{(0,\infty)} e^{-\lambda t} \left( f(x) + \int_0^t P_s g(x) \, ds \right) \, dt \\
		&= f(x) - \int_{(0,\infty)} \left( \frac{d}{dt} e^{-\lambda t} \right) \left( \int_0^t P_s g(x) \, ds \right) \, dt.
	\end{align*}
	Applying the integration by parts formula we find that \begin{equation*}
		\lambda \int_{(0,\infty)} e^{-\lambda t} P_t f(x) \, dt
		= f(x)+ \int_{(0,\infty)} e^{-\lambda t} P_t g(x) \, dt,
	\end{equation*}
	i.\,e.\ $\lambda R_{\lambda} f = f + R_{\lambda} g$. This proves $f \in \mc{D}(A_e)$,  $A_e f = g$ and $\|A_e f\|_{\infty} \leq K(f)$. \par
	If $f \in \mc{D}$, then the local martingale \begin{equation*}
		M_t = f(X_t)-f(X_0) - \int_0^t A_e f(X_s) \, ds
	\end{equation*}
	satisfies \begin{equation*}
		\mathbb{E}^x(M_{t \wedge \tau}^2) \leq (2\|f\|_{\infty}+\|A_e f\|_{\infty})^2 (1+t), \qquad t \geq 0, \,x \in \mbb{R}^d,
	\end{equation*} 
	for any stopping time $\tau$. It is immediate from Doob's maximal inequality that $\sup_{s \leq t} |M_s|$ is square-integrable, and this, in turn, implies that $(M_t)_{t \geq 0}$ is a martingale. In particular, $\mathbb{E}^x(M_t)=\mathbb{E}^x(M_0)$, i.\,e.\ \begin{equation*}
		\mbb{E}^x f(X_t)-f(x) = \mbb{E}^x \left( \int_0^t A_e f(X_s) \,d s \right),
	\end{equation*}
	and so $K(f) \leq \|A_e f\|_{\infty}<\infty$ and $f \in F_1$. Finally, we note that Dynkin's formula \eqref{app-gen-eq6} was shown in \cite[Corollary 5.11]{foellmer74} for any function $f \in \mc{B}_b(\mbb{R}^d)$ satisfying $K(f)<\infty$.
\end{proof}

\begin{kor} \label{app-gen-5}
	Let $(X_t)_{t \geq 0}$ be a Feller process with semigroup $(P_t)_{t \geq 0}$,  extended generator $(A_e,\mc{D}(A_e))$ and symbol $q$. Denote by \begin{equation*}
		\tau_r^x := \inf\{t>0; |X_t-x|>r\}
	\end{equation*}
	the exit time of $(X_t)_{t \geq 0}$ from the closed ball $\overline{B(x,r)}$. If the symbol $q$ has bounded coefficients, then the following statements are equivalent for any $f \in \mc{B}_b(\mbb{R}^d)$. 
 \begin{enumerate}
 		\item\label{app-gen-5-i} $f \in F_1$, i.e.\ $f \in \mc{D}(A_e)$ and $\sup_{t \in (0,1)} t^{-1} \|P_t f-f\|_{\infty} = \|A_e f\|_{\infty}< \infty$,
 		\item\label{app-gen-5-ii} There exists $r>0$ such that \begin{equation*}
 			K^{(1)}_r(f) := \sup_{t \in (0,1)} \sup_{x \in \mbb{R}^d} \frac{1}{\mbb{E}^x(t \wedge \tau_r^x)} |\mbb{E}^x f(X_{t \wedge \tau_r^x})-f(x)|<\infty. \end{equation*}
		\item\label{app-gen-5-iii} There exists $r>0$ such that \begin{equation*}
			K^{(2)}_r(f) := \sup_{t \in (0,1)} \frac{1}{t} \sup_{x \in \mbb{R}^d} |\mbb{E}^x f(X_{t \wedge \tau_r^x})-f(x)|<\infty. \end{equation*}
	\end{enumerate}
	If one (hence all) of the conditions is satisfied, then \begin{equation}
			A_e f(x) = \lim_{t \to 0} \frac{\mbb{E}^x f(X_{t \wedge \tau_r^x})-f(x)}{t} = \lim_{t \to 0} \frac{\mbb{E}^x f(X_{t \wedge \tau_r^x})-f(x)}{\mbb{E}^x (t \wedge \tau_r^x)} \label{app-gen-eq7}
	\end{equation}
	up to a set of potential zero for any $r \in (0,\infty]$. In particular, $\|A_e f\|_{\infty} \leq K^{(i)}_r(f)$ for $i \in \{1,2\}$ and $r \in (0,\infty]$.
\end{kor}

The proof of Corollary~\ref{app-gen-5} shows that the implications $\eqref{app-gen-5-i} \implies \eqref{app-gen-5-ii}$, $\eqref{app-gen-5-i} \implies \eqref{app-gen-5-iii}$ and $\eqref{app-gen-5-i} \implies \eqref{app-gen-eq7}$ remain valid if the symbol $q$ has unbounded coefficients.

\begin{proof}[Proof of Corollary~\ref{app-gen-5}]
	\eqref{app-gen-5-i} $\implies$ \eqref{app-gen-5-ii}: If $f \in F_1$ then it follows from Dynkin's formula \eqref{app-gen-eq6} that \begin{equation*}
		K_r^{(1)}(f) \leq \|A_e f\|_{\infty} < \infty \fa r>0.
	\end{equation*}
	\eqref{app-gen-5-ii} $\implies$ \eqref{app-gen-5-iii}: This is obvious because $\mbb{E}^x(t \wedge \tau_r^x) \leq t$. \par
	\eqref{app-gen-5-iii} $\implies$ \eqref{app-gen-5-i}: Fix $t \in (0,1)$. Clearly, \begin{equation*}
		|\mbb{E}^x f(X_t)-f(x)| \leq |\mbb{E}^x f(X_{t \wedge \tau_r^x})-f(x)| + |\mbb{E}^x (f(X_{t \wedge \tau_r^x})-f(X_t))|.
	\end{equation*}
	By assumption, the first term on the right-hand side is bounded by $K_r^{(2)}(f) t$. For the second term we note that \begin{equation*}
		|\mbb{E}^x (f(X_{t \wedge \tau_r^x})-f(X_t))|
		\leq 2 \|f\|_{\infty} \mbb{P}^x(\tau_r^x \leq t).
	\end{equation*}
	An application of the maximal inequality \eqref{max} for Feller processes shows that there exists an absolute constant $c>0$ such that \begin{equation*}
		|\mbb{E}^x (f(X_{t \wedge \tau_r^x})-f(X_t))|
		\leq 2ct  \|f\|_{\infty} \sup_{|y-x| \leq r} \sup_{|\xi| \leq r^{-1}} |q(y,\xi)| 
		\leq 2ct \|f\|_{\infty} \sup_{y \in \mbb{R}^d} \sup_{|\xi| \leq r^{-1}} |q(y,\xi)|;
	\end{equation*}
	note that the right-hand side is finite because $q$ has bounded coefficients. Combining both estimates gives \eqref{app-gen-5-i}.  \par
	Proof of \eqref{app-gen-eq7}: For $r=\infty$ this follows from \cite{foellmer74}, see the proof of Theorem~\ref{app-gen-3}. Fix $r \in (0,\infty)$. Applying Dynkin's formula \eqref{app-gen-eq6} we find \begin{align*}
		\left| \frac{\mbb{E}^x f(X_t)-f(x)}{t} - \frac{\mbb{E}^x f(X_{t \wedge \tau_r^x})-f(x)}{t} \right|
		\leq \frac{1}{t} \|A_e f\|_{\infty} \mbb{E}^x(t- \min\{\tau_r^x,t\})
		\leq \|A_e f\|_{\infty} \mbb{P}^x(\tau_r^x \leq t).
	\end{align*}
	The right-continuity of the sample paths of $(X_t)_{t \geq 0}$ gives $\mbb{P}^x(\tau_r^x \leq t) \to 0$ as $t \to 0$, and therefore we obtain that \begin{equation*}
		\lim_{t \to 0} \frac{\mbb{E}^x f(X_{t \wedge \tau_r^x})-f(x)}{t} = \lim_{t \to 0} \frac{\mbb{E}^x f(X_t)-f(x)}{t}.
	\end{equation*}
	Since the right-hand side equals $A_e f(x)$ up to a set of potential zero, see the proof of Theorem~\ref{app-gen-3}, this proves the first ``$=$'' in \eqref{app-gen-eq7}. Similarly, it follows from Dynkin's formula that \begin{align*}
		\left| \frac{\mbb{E}^x f(X_{t \wedge \tau_r^x})-f(x)}{t} - \frac{\mbb{E}^x f(X_{t \wedge \tau_r^x})-f(x)}{\mbb{E}^x(t \wedge \tau_r^x)} \right|
		&\leq \|A_e f\|_{\infty} \mbb{E}^x(\tau_r^x \wedge t) \left| \frac{1}{t}- \frac{1}{\mbb{E}^x(t \wedge \tau_r^x)} \right| \\
		&\leq \|A_e f\|_{\infty} \mbb{P}^x(\tau_r^x \leq t).
	\end{align*}
	As $\mbb{P}^x(\tau_r^x \leq t) \to 0$ we find that the right-hand side converges to $0$ as $t \to 0$, and this proves the second ``$=$'' in \eqref{app-gen-eq7}. 
\end{proof}

\section{Parametrix construction of the transition density} \label{pix}

Let $(X_t)_{t \geq 0}$ be an isotropic stable-like process with symbol $q(x,\xi) = |\xi|^{\alpha(x)}$ for a H\"{o}lder continuous mapping $\alpha: \mbb{R}^d \to (0,2)$ with $\alpha_L := \inf_x \alpha(x)>0$. For the proof of Proposition~\ref{app-1} the parametrix construction of the transition density of $(X_t)_{ \geq 0}$ in \cite{matters} plays a crucial role, see also \cite{parametrix}. In this section, we collect some results from \cite{matters} which are needed for our proofs. Throughout, $p^{\varrho}(t,x)$ denotes the transition density of an isotropic $\varrho$-stable L\'evy process,  $\varrho \in (0,2]$, i.e. \begin{equation}
	p^{\varrho}(t,x) = \frac{1}{(2\pi)^d} \int_{\mbb{R}^d} e^{ix \cdot \xi} e^{-t |\xi|^{\varrho}} \, d\xi \label{pix-eq1},
\end{equation}
and $\circledast$ is the time-space convolution, i.\,e. \begin{equation*}
	(f \circledast g)(t,x,y) := \int_0^t \!\! \int_{\mbb{R}^d} f(t-s,x,z) g(s,z,y) \, dz \, ds, \qquad t>0,\,x,y \in \mbb{R}^d. 
\end{equation*}
By \cite[Theorem 5.2, Theorem 4.25]{matters}, the transition density $p$ of $(X_t)_{t \geq 0}$ has the representation \begin{equation}
	p(t,x,y) = p_0(t,x,y)+ (p_0 \circledast \Phi)(t,x,y), \qquad t>0,\,x,y \in \mbb{R}^d  \label{pix-eq4}
\end{equation}
where $p_0$ is the zero-order approximation of $p$, defined by, 
\begin{equation}
	p_0(t,x,y) := p^{\alpha(y)}(t,x-y), \qquad t>0, \,x,y \in \mbb{R}^d, \label{pix-eq3}
\end{equation}
and $\Phi$ is a suitable function, see \eqref{pix-eq6} below for the precise definition. There exists for any $T>0$ a constant $C_1>0$ such that \begin{equation*}
	|p_0(t,x,y)| \leq C_1 S(x-y,\alpha(y),t), \qquad t \in (0,T), \,x,y \in \mbb{R}^d 
\end{equation*}
where \begin{equation*}
	S(x,\alpha,t) := \min \left\{ t^{-d/\alpha}, \frac{t}{|x|^{d+\alpha}} \right\}, 
\end{equation*}
cf.\ \cite[Section 4.1]{matters}. A straight-forward computation yields \begin{equation*}
	\forall 0<a<b \leq 2\::\: \sup_{t \in (0,T)}\sup_{z \in \mbb{R}^d} \sup_{\varrho \in [a,b]} \int_{\mbb{R}^d} S(z-y,\varrho,t) \, dy <\infty,
\end{equation*}
cf.\ \cite[Lemma 4.16]{matters} for details.  The function $\Phi$ in \eqref{pix-eq4} has the representation \begin{equation}
	\Phi(t,x,y) = \sum_{i=1}^{\infty} F^{\circledast i}(t,x,y), \qquad t>0, \, x,y \in \mbb{R}^d \label{pix-eq6}
\end{equation}
where $F^{\circledast i} := F \circledast F^{\circledast (i-1)}$ denotes the $i$-th convolution power of \begin{equation*}
	F(t,x,y) :=  (2\pi)^{-d} \int_{\mbb{R}^d} \left( |\xi|^{\alpha(y)}-|\xi|^{\alpha(x)} \right) e^{i \xi \cdot (y-x)} e^{-t |\xi|^{\alpha(y)}} \, d\xi, \qquad t>0, \, x,y \in \mbb{R}^d.
\end{equation*}
It is possible to show that \begin{equation*}
	\sup_{x \in \mbb{R}^d} \int_{\mbb{R}^d} |\Phi(t,x,y)| \, dy \leq C_2 t^{-1+\lambda}, \qquad t \in (0,T) 
\end{equation*}
for some constant $\lambda>0$ and $C_2=C_2(T)>0$, cf.\ \cite[Theorem 4.25(iii), Lemma A.8]{matters}. Moreover, by \cite[Lemma 4.21 \& 4.24]{matters} there exist constants $C_3=C_3(T)>0$ and $\lambda>0$ such that \begin{align*}
	\int_{\mbb{R}^d} |F^{\circledast i}(t,x,y)| \, dy \leq C_3^{i} \frac{\Gamma(\lambda)^{i}}{\Gamma(i \lambda)} t^{-1+i \lambda}, \quad x  \in \mbb{R}^d,\, t \in (0,T). 
\end{align*}
Because of the representation \eqref{pix-eq3}, the following estimates are a useful tool to derive estimates for the transition density $p$.

\begin{lem} \label{pix-3}
	Let $I=[a,b] \subset (0,2)$. For all $T>0$ and $k \in \mbb{N}_0$ there exists a constant $C>0$ such that the following estimates hold for any $\varrho \in [a,b]$, $x \in \mbb{R}^d$, $t \in (0,T)$ and any multiindex $\beta \in \mbb{N}_0^d$ with $|\beta|=k$: \begin{align}
		|\partial_x^{\beta} p^{\varrho}(t,x)| &\leq C t^{-|\beta|/\varrho} S(x,\varrho,t), \label{pix-eq11} \\
		\int_{\mbb{R}^d} \left| \frac{\partial^{\beta}}{\partial x^{\beta}} \frac{\partial}{\partial \varrho} p^{\varrho}(t,x) \right| \, dx &\leq C (1+|\log(t)|) t^{-|\beta|/\varrho}. \label{pix-eq13}
	\end{align}
\end{lem}

\begin{proof}
	We only prove \eqref{pix-eq13}; for the proof of the pointwise estimate \eqref{pix-eq11} see \cite[Theorem 4.12]{matters}. Denote by $p^{\varrho}=p^{\varrho,d}$ the transition density of the $d$-dimensional isotropic $\varrho$-stable L\'evy process, $\varrho \in (0,2)$. It follows from the Fourier representation \eqref{pix-eq1} of $p^{\varrho}$ that $\varrho \mapsto p^{\varrho,d}(t,x)$ and $x \mapsto p^{\varrho,d}(t,x)$ are infinitely many differentiable and \begin{equation*}
		\partial_{\varrho} \partial_x^{\beta} p^{\varrho,d}(t,x) = -\frac{t}{(2\pi)^d} \int_{\mbb{R}^d} (i \xi)^{\beta}e^{ix \cdot \xi} e^{-t |\xi|^{\varrho}} |\xi|^{\varrho} \log(|\xi|) \, d\xi
	\end{equation*}
	for any $\varrho \in [a,b]$, $x \in \mbb{R}^d$, $t>0$ and $\beta \in \mbb{N}_0^d$. In particular, \begin{equation}
		\frac{\partial}{\partial \varrho} p^{\varrho,d}(t,x) = -t \frac{1}{(2\pi)^d} \int_{\mbb{R}^d} e^{ix \cdot \xi} e^{-t |\xi|^{\varrho}} |\xi|^{\varrho} \log(|\xi|) \, d\xi, \qquad t>0, \, x \in \mbb{R}^d, \label{pix-eq41} 
	\end{equation}
	and, by \cite[Theorem 4.7]{matters}, this implies that there exists a constant $c_2>0$ such that \begin{equation}
		\left| \frac{\partial}{\partial \varrho} p^{\varrho,d}(t,x) \right| 
		\leq c_2 \min\left\{(1+|\log(t)|) t^{-d/\varrho},  \frac{t}{|x|^{d+\varrho}} \left(1+|\log(|x)| \right) \right\} \label{pix-eq42}
	\end{equation}
	for all $t \in (0,T]$, $x \in \mbb{R}^d$ and $\varrho \in [a,b] \subseteq (0,2]$. By \eqref{pix-eq41}, $\partial_{\varrho} p^{\varrho,d}$ is the Fourier transform of a rotationally invariant function, and therefore it follows from the dimension walk formula for the Fourier transform, see e.\,g.\ \cite[Lemma 4.13]{matters} or \cite{schoenberg} and the references therein, that \begin{equation*}
		\frac{\partial}{\partial x_j} \frac{\partial}{\partial \varrho} p^{\varrho,d}(t,x) = -2 \pi x_j \frac{\partial}{\partial \varrho} p^{\varrho,d+2}(t,x)
	\end{equation*}
	for $j=1,\ldots,d$, $t>0$, $x \in \mbb{R}^d$ and $\varrho \in (0,2)$. Using \eqref{pix-eq42} for dimension $d+2$ we obtain that there is a constant $c_3>0$ such that \begin{equation}
		\int_{\mbb{R}^d} \left| \frac{\partial}{\partial x_j} \frac{\partial}{\partial \varrho} p^{\varrho,d}(t,x) \right| \, dx \leq c_3 (1+|\log(t)|) t^{-1/\alpha_L} \label{pix-eq44}
	\end{equation}
	for all $t \in (0,T]$, $j \in \{1,\ldots,d\}$ and $\varrho \in [\alpha_L,\|\alpha\|_{\infty}] \subseteq (0,2]$. By iteration, we get \eqref{pix-eq13}.
\end{proof}

\section{Inequalities for H\"{o}lder continuous functions}

We present two inequalities for H\"{o}lder continuous functions which we used in Section~\ref{iso}.

\begin{lem} \label{ineq-1}
	Let $f: \mbb{R}^d \to \mbb{R}$ be a function. If $x \in \mbb{R}^d$ and $M_1,M_2 >0$ are such that \begin{equation*}
		|\Delta_h^2 f(x)| \leq M_1 |h|^2 \quad \text{and} \quad |\Delta_h^2 f(x)| \leq M_2
	\end{equation*}
	for all $h \in \mbb{R}^d$, then \begin{equation*}
		|\Delta_h^2 f(x)| \leq |h|^{\kappa} \max\{M_1 r^{2-\kappa}, M_2 r^{-\kappa}\}
	\end{equation*}
	for any $r>0$, $h \in \mbb{R}^d$ and $\kappa \in [0,2]$.
\end{lem}

\begin{proof}
	Fix $\kappa \in [0,2]$ and $r>0$. If $h \in \mbb{R}^d$ is such that $|h|>r$, then \begin{equation*}
	|\Delta_h^2 f(x)| \leq M_2 \leq M_2 \frac{|h|^{\kappa}}{r^{\kappa}}.
	\end{equation*}
	If $|h| \leq r$ then \begin{equation*}
	|\Delta_h^2 f(x)| \leq M_1 |h|^2 \leq M_1 |h|^{\kappa} r^{2-\kappa}. \qedhere
	\end{equation*}
\end{proof}

\begin{lem} \label{ineq-2}
	Let $f \in \mc{C}_b^{\gamma}(\mbb{R}^d)$ for some $\gamma \in (0,1)$. There exists a constant $C=C(\gamma)>0$ such that \begin{equation}
		|\Delta_h f(x)-\Delta_h f(y)| \leq C \|f\|_{\mc{C}_b^{\gamma}(\mbb{R}^d)} |x-y|^{\alpha} |h|^{\gamma-\alpha} \label{ineq-eq3}
	\end{equation}
	for all $\alpha \in [0,\gamma]$ and $x,y,h \in \mbb{R}^d$.
\end{lem}

If $f:\mbb{R}^d \to \mbb{R}$ is Lipschitz continuous and bounded, then \eqref{ineq-eq3} holds for $\gamma=1$; the norm $\|f\|_{C_b^{\gamma}(\mbb{R}^d)}$ needs to be replaced by the sum of the supremum norm and the Lipschitz constant of $f$.

\begin{proof}
	By the very definition of the H\"{o}lder--Zygmund space $\mc{C}_b^{\gamma}(\mbb{R}^d)$ we have \begin{equation*}
		|f(x+h)-f(x)| \leq \|f\|_{\mc{C}_b^{\gamma}(\mbb{R}^d)} |h|^{\gamma} \I_{\{|h| \leq 1\}} + 2 \|f\|_{\infty} \I_{\{|h|>1\}} \leq 2 \|f\|_{\mc{C}_b^{\gamma}(\mbb{R}^d)} |h|^{\gamma}
	\end{equation*}
	for any $x,h \in \mbb{R}^d$. Hence, \begin{equation}
		|\Delta_h f(x)-\Delta_h f(y)| \leq |f(x+h)-f(x)| + |f(y+h)-f(y)| \leq 4 \|f\|_{\mc{C}_b^{\gamma}(\mbb{R}^d)} |h|^{\gamma} \label{ineq-eq5}
	\end{equation}
	and \begin{equation}
		|\Delta_h f(x)-\Delta_h f(y)| \leq |f(x)-f(y)| + |f(x+h)-f(y+h)| \leq 4 \|f\|_{\mc{C}_b^{\gamma}(\mbb{R}^d)} |x-y|^{\gamma} \label{ineq-eq6}
	\end{equation}
	for all $x,y,h \in \mbb{R}^d$, i.e.\,\eqref{ineq-eq3} holds for $\alpha=0$ and $\alpha=\gamma$. Next we show that \eqref{ineq-eq3} holds for $\alpha=\gamma/2$ and to this end we use interpolation theory. Let $f=u+v$ for $u \in C_b(\mbb{R}^d)$ and $v \in C_b^2(\mbb{R}^d)$. Clearly, \begin{equation*}
		|\Delta_h u(x)-\Delta_h u(y)| \leq 4 \|u\|_{\infty}
	\end{equation*}
	and, by the gradient theorem, \begin{align*}
		|\Delta_h v(x)-\Delta_h v(y)|
		= \left| h \int_0^1 (\nabla v(x+rh)-\nabla v(y+rh)) \, dr \right|
		\leq |h| \, |x-y| \, \|v\|_{C_b^2(\mbb{R}^d)}
	\end{align*}	
	for all $x,y,h \in \mbb{R}^d$. Hence, \begin{equation*}
		|\Delta_h f(x)-\Delta_h f(y)| \leq 4 \|u\|_{\infty} + |h| \, |x-y| \, \|v\|_{C_b^2(\mbb{R}^d)}, \qquad x,y,h \in \mbb{R}^d.
	\end{equation*} 
	Since $\mc{C}_b^{\gamma}(\mbb{R}^d)$ is the real interpolation space\footnote{More precisely, the norm on the interpolation space $(C_b(\mbb{R}^d),C_b^2(\mbb{R}^d))_{\gamma/2,\infty}$ is equivalent to the norm on $\mc{C}_b^{\gamma}(\mbb{R}^d)$.} $(C_b(\mbb{R}^d),C_b^2(\mbb{R}^d)_{\gamma/2,\infty}$, cf.\ \cite[Section 2.7.2]{triebel78}, this implies that there exists a constant $C>0$ such that \begin{equation}
		|\Delta_h f(x)-\Delta_h f(y)| \leq C |h|^{\gamma/2} |x-y|^{\gamma/2} \|f\|_{\mc{C}_b^{\gamma}(\mbb{R}^d)} \label{ineq-eq7}
	\end{equation}
	which shows \eqref{ineq-eq3} for $\alpha=\gamma/2$. Now let $\alpha \in (0,\gamma/2)$. For $|h| \leq |x-y|$ it follows from \eqref{ineq-eq5} that \begin{equation*}
		|\Delta_h f(x)-\Delta_h f(y)|
		\leq 4 \|f\|_{\mc{C}_b^{\gamma}(\mbb{R}^d)} |h|^{\gamma} \leq 4 \|f\|_{\mc{C}_b^{\gamma}(\mbb{R}^d)} |h|^{\alpha} |x-y|^{\gamma-\alpha}.
	\end{equation*}
	If $|h|>|x-y|$ then \eqref{ineq-eq7} gives \begin{equation*}
		|\Delta_h f(x)-\Delta_h f(y)|
		\leq C \|f\|_{\mc{C}_b^{\gamma}(\mbb{R}^d)} |x-y|^{\gamma/2} |h|^{\gamma/2} \leq C \|f\|_{\mc{C}_b^{\gamma}(\mbb{R}^d)} |x-y|^{\alpha} |h|^{\gamma/2+(\gamma/2-\alpha)}
	\end{equation*}
	where we used $\alpha<\gamma/2$ for the second estimate. For $\alpha \in (\gamma/2,\gamma)$ a very similar reasoning shows that \eqref{ineq-eq3} follows from \eqref{ineq-eq6} and \eqref{ineq-eq7}.
\end{proof}

\section{A separation theorem for closed subsets}

In Section~\ref{iso} we used the following result on the smooth separation of closed subsets of $\mbb{R}^d$.

\begin{lem} \label{app-5}
	Let $F,G \subseteq \mbb{R}^d$ be closed sets. If \begin{equation}
	d(F,G) = \inf\{|x-y|; x \in F, y \in G\}>0 \label{app-eq19}
	\end{equation}
	then there exists a function $f \in C_b^{\infty}(\mbb{R}^d)$, $0 \leq f \leq 1$, such that \begin{equation}
	f^{-1}(\{0\}) = F \quad \text{and} \quad f^{-1}(\{1\})=G. \label{app-eq21}
	\end{equation}
\end{lem}

It is well known, see e.\,g.\ \cite{lee}, that for closed sets $F,G \subseteq \mbb{R}^d$ with \eqref{app-eq19} there exists $f \in C^{\infty}(\mbb{R}^d)$, $0 \leq f \leq 1$, satisfying \eqref{app-eq21}; however, we could not find a reference for the fact that \eqref{app-eq19} implies boundedness of the derivatives of $f$. It is not difficult to see that boundedness of the derivatives fails, in general, to hold if $d(F,G)=0$; consider for instance $F:=\mbb{R} \times (-\infty,0]$ and $G:= \{(x,y); y \geq e^x\}$.

\begin{proof}[Proof of Lemma~\ref{app-5}]
	As $d(F,G)>0$ we can choose $\eps>0$ such that the sets $$F_{\eps} := F+\overline{B(0,\eps)} \qquad G_{\eps} := G+\overline{B(0,\eps)}$$ are disjoint. It is known, see e.\,g.\ \cite[Problem 2-14]{lee}, that there exists $h \in C^{\infty}(\mathbb{R}^d)$, $0 \leq h \leq 1$, such that $h^{-1}(\{0\}) = F_{\eps}$ and  $h^{-1}(\{1\})=G_{\eps}$. Pick $\varphi \in C_c^{\infty}(\mathbb{R}^d)$, $\varphi \geq 0$, such that $\spt \varphi = \overline{B(0,\eps)}$ and $\int_{\mbb{R}^d} \varphi(y) \, dy=1$,  and set \begin{equation*}
	f(x) := (h \ast \varphi)(x) = \int_{\mathbb{R}^d} h(y) \varphi(x-y) \, dy, \qquad x \in \mbb{R}^d,
	\end{equation*}
	Since $f$ is the convolution of a bounded continuous function with a smooth function with compact support, it follows that $f$ is smooth and its derivatives are given by \begin{equation*}
	\partial_x^{\alpha} f(x) = \int_{\mathbb{R}^d} h(y) \partial_x^{\alpha} \varphi(x-y) \, dy, \qquad x \in \mbb{R}^d,
	\end{equation*}
	for any multi-index $\alpha \in \mathbb{N}_0^d$, see e.\,g.\ \cite{mims}. This implies, in particular, $\|\partial^{\alpha} f\|_{\infty} \leq \|\partial^{\alpha} \varphi\|_{L^1}< \infty$, and so $f \in C_b^{\infty}(\mathbb{R}^d)$. Moreover, as $\spt \varphi \subseteq \overline{B(0,\eps)}$, it is obvious that $f(x)=0$ for any $x \in F$ and $f(x)=1$ for $x \in G$. It remains to check that $0<f(x)<1$ for any $x \in (F \cup G)^c$. 
	
	Case 1: $x \in \mathbb{R}^d \backslash (F_{\eps} \cup G_{\eps})$. Then $0 < h(x)<1$, and therefore we can choose $r \in (0,\eps)$ such that \begin{equation*}
	0 < \inf_{|y-x| \leq r} h(y) \leq \sup_{|y-x| \leq r} h(y)<1. 
	\end{equation*}
	Since $\spt \varphi = \overline{B(0,\eps)} \supseteq \overline{B(0,r)}$ this implies  \begin{align*} 
	f(x) 
	&\leq \int_{\mbb{R}^d \backslash \overline{B(x,r)}} \varphi(x-y) \, dy + \sup_{|y-x| \leq r} h(y) \int_{\overline{B(x,r)}}\varphi(x-y) \, dy 
	< \int_{\mathbb{R}^d} \varphi(x-y) \, dy =1. 
	\end{align*} 
	A very similar estimate shows $f(x)>0$. \par
	Case 2: $x \in F_{\eps} \backslash F$. We have $\overline{B(x,\eps)} \cap F^c \neq \emptyset$, and therefore there exist $y \in \mathbb{R}^d$ and $r>0$ such that \begin{equation*}
	\overline{B(y,r)} \subseteq F^c \cap \overline{B(x,\eps)}.
	\end{equation*} 
	In particular \begin{equation*}
	0 < \inf_{z \in \overline{B(y,r)}} h(z) \leq \sup_{z \in \overline{B(y,r)}} h(z) < 1.
	\end{equation*}
	As $\spt \varphi = \overline{B(0,\eps)}$ it follows very similar as in the first case that $0<f(x)<1$. \par
	Case 3: $x \in G_{\eps} \backslash G$. Analogous to Case 2.
\end{proof}

\begin{ack}
	I am grateful to Ren\'e Schilling for valuable comments. Moreover, I would like to thank the \emph{Institut national des sciences appliqu\'ees de Toulouse, G\'enie math\'ematique et mod\'elisation} for its hospitality during my stay in Toulouse, where a part of this work was done.
\end{ack}


\begin{thebibliography}{99}\frenchspacing	

\bibitem{foellmer74} 
	Airault, H., F\"ollmer, H.: Relative Densities of Semimartingales. \emph{Invent. Math.} \textbf{27} (1974), 299--327.

\bibitem{bae} 
	Bae, J., Kassmann, M.: Schauder estimates in generalized H\"older spaces. arXiv 1505.05498

\bibitem{bass08} 
	Bass, R.F.: Regularity results for stable-like operators. \emph{J.\ Funct.\ Anal.} \textbf{259} (2009), 2693--2722.
	
\bibitem{bogdan17} 
	Bogdan, K., Sztonyk, P., Knopova, V.: Heat kernel of anisotropic nonlocal operators. Preprint arXiv 1704.03705.


\bibitem{ltp}
    B\"{o}ttcher, B., Schilling, R.\,L., Wang, J.: \emph{L\'evy-Type Processes: Construction, Approximation and Sample Path Properties}. Springer Lecture Notes in Mathematics vol.\ \textbf{2099} (vol.~III of the ``L\'evy Matters'' subseries). Springer, 2014.
	    
\bibitem{bouleau81} 
	Bouleau, N.: Propriet\'es d’invariance du domaine du g\'en\'erateur infinit\'esimal \'etendu d’un processus de Markov. \emph{S\'eminaire de Probabilit\'es (Strasbourg)} \textbf{15} (1981), 167--188.

\bibitem{bouleau91} Bouleau, N., Hirsch, F.: \emph{Dirichlet Forms and Analysis on Wiener Space}. De Gruyter, 1991.

\bibitem{butzer} Butzer, P. P., Berens, H.: \emph{Semi-Groups of Operators and Approximation}. Springer, Berlin 1967.

\bibitem{chen18} 
	Chen, X., Chen, Z.-Q., Wang, J.: Heat kernel for non-local operators with variable order.  Preprint arXiv 1811.09972.


\bibitem{chen17}
	Chen, Z.-Q., Zhang, X., Zhao, G.: Well-posedness of supercritical SDE driven by L\'evy processes with irregular drifts. Preprint arXiv 1709.04632.

\bibitem{meyer4}
	Dellacherie, C., Meyer, P.-A.,: \emph{Théorie du potentiel associ\'ee à une r\'esolvante - th\'eorie des processus de Markov}. Hermann, Paris 1987.

\bibitem{engel} 
	Engel, K.-J., Nagel, R.: \emph{One-Parameter Semigroups for Linear Evolution Equations}. Springer, New York 2000.

\bibitem{ethier} 
	Ethier, S.\,N., Kurtz, T.\,G.: \emph{Markov Processes: Characterization and Convergence}. Wiley, New York 1986.

\bibitem{hairer12} Hairer, M., Hutzenthaler, M., Jentzen, A.: Loss of regularity for Kolmogorov equations. \emph{Ann.\,Probab.} \textbf{43} (2015), 468--527.


\bibitem{hirsch84} 
	Hirsch, F.: Generateurs etendus et subordination au sens de Bochner. In: M. Brelot, G. Choquet, J. Deny, F. Hirsch, G. Mokobodzki (eds.): \emph{S\'eminaire de Th\'eorie du Potentiel Paris, No. 7}. Springer, Berlin 1984, pp.~134--156.


\bibitem{hoh}
	Hoh, W.: \emph{Pseudo-Differential Operators Generating Markov Processes}. Habilitationsschrift. Universit\"{a}t Bielefeld, Bielefeld 1998.

\bibitem{jacob123}
	Jacob, N.: \emph{Pseudo Differential Operators and Markov Processes I, II, III}. Imperial College Press/World Scientific, London 2001--2005.
	
	\bibitem{kol} 
		Kolokoltsov, V. N.: \emph{Markov processes, semigroups and generators}. De Gruyter, Berlin 2011.

	\bibitem{moments} 
		K\"{u}hn, F.: Existence and estimates of moments for Lévy-type processes. \emph{Stoch. Proc. Appl.} \textbf{127} (2017), 1018--1041.

	\bibitem{matters}
		K{\"u}hn, F.: \emph{L\'evy-Type Processes: Moments, Construction and Heat Kernel Estimates}. Springer Lecture Notes in Mathematics vol.\ \textbf{2187} (vol.~VI of the ``L\'evy Matters'' subseries). Springer, 2017.

	\bibitem{mp}
		K\"uhn, F.: On Martingale Problems and Feller Processes. \emph{Electron.\ J.\ Probab.} \textbf{23} (2018), 1--18.

	\bibitem{sde}
		K{\"u}hn, F.: Solutions of L\'evy-driven SDEs with unbounded coefficients as Feller processes. \emph{Proc. Amer. Math. Soc.} \textbf{146} (2018), 3591–3604.

	\bibitem{parametrix} 
		K\"{u}hn, F.: Transition probabilities of Lévy-type processes: Parametrix construction.  \emph{Math.\ Nachr.} \textbf{292} (2019), 358--376.
		
	\bibitem{perpetual} 
		K\"{u}hn, F.: Perpetual integrals via random time changes.  \emph{Bernoulli} \textbf{25} (2019), 1755--1769.
		
	\bibitem{reg-levy}
		K\"uhn, F.: Schauder estimates for equations associated with L\'evy generators. \emph{Integral Equations Operator Theory} \textbf{91}:10 (2019).
		
	\bibitem{schoenberg} 
		K\"{u}hn, F., Schilling, R.L.: A probabilistic proof of Schoenberg's theorem.  \emph{J.\ Math.\ Anal.\ Appl.} \textbf{476} (2019), 13--26.

	\bibitem{ihke}
		K\"{u}hn, F., Schilling, R.L.: On the domain of fractional Laplacians and related generators of Feller processes.   \emph{J.\ Funct.\ Anal.} \textbf{276} (2019), 2397--2439.
	
	\bibitem{euler-maruyama} 
		K\"uhn, F., Schilling, R.L.: Strong convergence of the Euler–Maruyama approximation for a class of Lévy-driven SDEs.  \emph{Stoch. Proc. Appl.} \textbf{129} (2019), 2654--2680.
		
	\bibitem{kulc18} 
		Kulczycki, T., Ryznar, M., Sztonyk, P.: Strong Feller property for SDEs driven by multiplicative cylindrical stable noise. Preprint arXiv 1811.05960.
	\bibitem{kulik15} 
		Kulik, A.: On weak uniqueness and distributional properties of a solution to an SDE with $\alpha$-stable noise. \emph{Stoch. Proc. Appl.} \textbf{129} (2019), 473--506.
	
	\bibitem{kunita69} 
		Kunita, H.: Absolute continuity of Markov processes and generators. \emph{Nagoya Math. J.} \textbf{36} (1969), 1--26.
	
\bibitem{lee} 
	Lee, J.\,M.: \emph{Introduction to Smooth Manifolds}. Springer, New York 2012.
	
\bibitem{wang18} 
	Liang, M., Wang, J.: Gradient Estimates and Ergodicity for SDEs Driven by Multiplicative L\'{e}vy Noises via Coupling. Preprint arXiv 1801.05936.

\bibitem{wang18-2} 
	Liang, M., Wang, J.: Spatial regularity of semigroups generated by L\'{e}vy type operators. \emph{Math.\ Nachr.} \textbf{292} (2019), 1551-1566.

\bibitem{wang18-3} 
	Liang, M., Schilling, R. L., Wang, J.: A Unified Approach to Coupling SDEs driven by L\'{e}vy Noise and Some Applications. Preprint arXiv 1811.08477.

\bibitem{lunardi} Lunardi, A.: \emph{Interpolation Theory}. Scuola Normale Superiore, Pisa 2009.

\bibitem{wang14} 
	Luo, D., Wang, J.: Coupling by reflection and H\"{o}lder regularity for non-local operators of variable order. \emph{Trans.\ Amer.\ Math.\ Soc.} \textbf{371} (2019), 431--459.
	
\bibitem{meyer76} 
	Meyer, P.-A.: D\'emonstration probabiliste de certaines inegalites de Littlewood-Paley. Expos\'e II: l'op\'erateur carr\'e du champ. \emph{S\'eminaire de Probabilit\'es (Strasbourg)}, \textbf{10} (1976), 142--161.

\bibitem{moko78} 
	Mokobodzki, G.: Sur l'algebre contenue dans le domaine etendu d'un generateur infinitesimal. In: F. Hirsch, G. Mokobodzki (eds.): \emph{S\'eminaire de Th\'eorie du Potentiel Paris, No. 3}. Springer, Berlin 1978, pp.~168--187.

\bibitem{priola15}
	Priola, E.: Stochastic flow for SDEs with jumps and irregular drift term. \emph{Stoch.\ Anal. Banach Center Publications} \textbf{105} (2015), 193--210.

\bibitem{rs-growth} 
	Schilling, R. L.: Growth and H\"older conditions for the sample paths of Feller processes. \emph{Probab. Theory Related Fields} \textbf{112} (1998), 565--611.

\bibitem{mims} 
	Schilling, R. L.: \emph{Measures, Integrals and Martingales}. Cambridge University Press, 2017 (2nd edition).

\bibitem{schnurr}
	Schilling, R.\,L., Schnurr, A.: The Symbol Associated with the Solution of a Stochastic Differential Equation. \emph{Electron.\ J.\ Probab.} \textbf{15} (2010), 1369--1393.

\bibitem{szcz18} 
	Szczypkowski, K.: Fundamental solution for super-critical non-symmetric L\'evy-type operators. Preprint arXiv 1807.04257.

\bibitem{stein} 
	Stein, E. M.: \emph{Singular integrals and differentiability properties of functions}. Princeton Univ. Press,  1970.
	
\bibitem{triebel78} 
	Triebel, H.: \emph{Interpolation theory, function spaces, differential operators}. North-Holland Pub. Co, 1978.


\end{thebibliography}
\end{document}